\documentclass[10pt,a4paper]{article}

\usepackage[slovak, english]{babel}

\usepackage{amsthm}
\usepackage{textcomp}

\newtheorem*{thA}{Theorem A}

\newtheorem*{thC}{Theorem C}
\newtheorem*{thD}{Theorem D}

\newtheorem{theorem}{Theorem}[section]
\newtheorem{claim}[theorem]{Claim}
\newtheorem{lemma}[theorem]{Lemma}

\newtheorem{corollary}[theorem]{Corollary}

\newtheorem{definition}[theorem]{Definition}

\newcommand{\aI}{{\alpha_{1}}}
\newcommand{\aII}{{\alpha_{2}}} 
\newcommand{\aIII}{{\alpha_{3}}}

\newcommand{\half}{\tfrac{1}{2}}

\newcommand{\llangle}{{\langle \! \langle}}
\newcommand{\rrangle}{{\rangle \! \rangle}}

\newcommand{\cE}{{\mathcal E}}

\newcommand{\cG}{{\mathcal G}}
\newcommand{\cH}{{\mathcal H}}

\newcommand{\cJ}{{\mathcal J}}
\newcommand{\cL}{{\mathcal L}}
\newcommand{\cM}{{\mathcal M}}

\newcommand{\cR}{{\mathcal R}}

\newcommand{\cW}{{\mathcal W}}
\newcommand{\cV}{{\mathcal V}}
\newcommand{\cX}{{\mathcal X}}
\newcommand{\cY}{{\mathcal Y}}
\newcommand{\cZ}{{\mathcal Z}}

\usepackage{amsfonts}
\usepackage{amsmath}
\usepackage{amssymb} 
\usepackage{mathrsfs}
\usepackage{setspace}
\usepackage{paralist}

\usepackage{geometry}
\usepackage{srcltx}
\usepackage{graphicx}
\usepackage{color}
\usepackage{upref}
\usepackage{enumerate}
\usepackage{latexsym}
\usepackage{amsmath}
\usepackage{amsfonts}
\usepackage{amssymb}
\usepackage[ansinew]{inputenc}
\usepackage{varioref}
\usepackage{subfig} 
\usepackage[final, notcite,notref]{showkeys} 
\usepackage{verbatim} 
\usepackage{soul}
\usepackage{complexity}

\usepackage{changebar}
\nochangebars

\usepackage{bm}
\usepackage{tabularx}
\usepackage{placeins}

\newcommand{\nocontentsline}[3]{}
\newcommand{\tocless}[2]{\bgroup\let\addcontentsline=\nocontentsline#1{#2}\egroup}

\makeatletter
\newcommand{\labitem}[2]{%
\def\@itemlabel{({#1})}
\item
\def\@currentlabel{#1}
\label{#2}}
\makeatother

\makeatletter

\newcommand{\Rmnum}[1]{\expandafter\@slowromancap\romannumeral #1@}
\makeatother

\newcommand{\aIna}{\llangle \aI n \rrangle}

\newcommand{\wX}{\widetilde{X}}

\usepackage{amssymb}
\renewcommand{\qed}{\qquad\hspace*{\fill}$\Box$}

\usepackage{hyperref}

\begin{document}

\title{Topics in Graph Colouring and Graph Structures}
\author{David Ferguson}

\begin{titlepage}
\begin{center}

\textsc{\LARGE \bf }\\[3cm]

{\huge \bf The Ramsey number of \\[12pt] mixed-parity cycles III}

\vspace{20mm}

{\Large David G. Ferguson} 
\end{center}

\vspace{40mm}

\abstract{
\noindent
Denote by $R(G_1, G_2, G_3)$ the minimum integer $N$ such that any three-colouring of the edges of the complete graph on $N$ vertices contains
a monochromatic copy of a graph $G_i$ coloured with colour~$i$ for some $i\in{1,2,3}$.
In a series of three papers of which this is the third, we consider the case where $G_1, G_2$ and $G_3$ are cycles of mixed parity. Specifically, in this in this paper, we consider~$R(C_n,C_m,C_{\ell})$, where $n$ is even and $m$ and $\ell$ are odd.
Figaj and \L uczak determined an asymptotic result for 
this case, which we improve upon to give an exact result. 
We prove that for~$n,m$ and $\ell$ sufficiently large
$$R(C_n,C_m,C_\ell)=\max\{4n-3, n+2m-3, n+2\ell-3\}.$$
}

\end{titlepage}


\pagenumbering{arabic}

\let\L\defaultL

\setcounter{page}{2}
\renewcommand{\baselinestretch}{1.25}\small\normalsize
\

For graphs $G_1,G_2,G_3$, the Ramsey number $R(G_1,G_2,G_3)$ is the smallest integer~$N$ such that every edge-colouring of the complete graph on~$N$ vertices with up to three colours results in the graph having, as a subgraph, a copy of~$G_{i}$ coloured with colour $i$ for some~$i$. 
We consider the case when $G_1,G_2$ and $G_3$ are~cycles.

In 1973, Bondy and Erd\H{o}s~\cite{BonErd} conjectured that, if~$n>3$ is odd, then $$R(C_{n},C_{n},C_{n})=4n-3.$$ 
Later, \L uczak~\cite{Lucz} proved that, for~$n$ odd,  $R(C_{n},C_{n},C_{n})=4n+o(n)$ as $n\rightarrow \infty$. Kohayakawa, Simonovits and Skokan~\cite{KoSiSk}, expanding upon the work of \L uczak, confirmed the Bondy-Erd\H{o}s conjecture for sufficiently large odd values of~$n$ by proving that there exists a positive integer~$n_0$ such that, for all odd $n,m,\ell>n_{0}$, 
$$R(C_{n},C_{m},C_{\ell})=4 \max\{n,m,\ell\} -3.$$   
In the case where all three cycles are of even length, Figaj and \L uczak~\cite{FL2007} proved the following asymptotic. Defining $\llangle x\rrangle$ to be the largest even integer not greater than~$x$, they proved that, for all $\aI,\aII,\aIII>0$,
$$R(C_{\llangle \aI n \rrangle},C_{\llangle \aII n\rrangle},C_{\llangle \aIII n \rrangle})=\half\big(\aI + \aII +\aIII+\max \{\aI,\aII,\aIII\}\big)n+o(n),$$
as $n \rightarrow \infty$. 

Thus, in particular, for even~$n$,
$$R(C_{n},C_{n},C_{n})=2n+o(n),\text{ as }n\rightarrow \infty.$$
Independently, Gy\'{a}rf\'{a}s, Ruszink\'{o}, S\'{a}rk\"{o}zy and Szem\'{e}redi~\cite{GyarSzem} proved a similar, but more precise, result for paths, namely that there exists a positive integer~$n_{1}$ such that, for $n>n_{1}$,
$$R(P_{n},P_{n},P_{n})=\begin{cases} 2n-1, & n \text{ odd,} \\ 2n-2, & n \text{ even.} \end{cases}$$
More recently, Benevides and Skokan~\cite{BenSko} proved that there exists~$n_{2}$ such that, for even $n>n_{2}$, 
$$R(C_{n},C_{n},C_{n})=2n.$$

We look at the mixed-parity case, for which, 
defining $\langle x \rangle$ to be the largest odd number not greater than~$x$, Figaj and \L uczak~\cite{FL2008} proved that, for all $\aI,\aII,\aIII>0$,
\begin{align*}
&\text{(i)}\ R(C_{\llangle \aI n \rrangle },C_{\llangle \aII n \rrangle },C_{\langle \aIII n\rangle  }) = \max \{2\aI+\aII, \aI+2\aII, \half\aI  + \half\aII +\aIII \}n +o(n),\\
&\text{(ii)}\ R(C_{\llangle \aI n \rrangle },C_{\langle \aII n \rangle  },C_{\langle \aIII n\rangle  }) = \max \{4\aI,\aI+2\aII, \aI  +2\aIII \}n +o(n),
\end{align*}
as $n\rightarrow \infty$.

In \cite{DF1} and \cite{DF2}, improving on  the result of Figaj and \L uczak, in the case when exactly one of the cycles is of odd length and the others are even, we proved the following:

\phantomsection
\hypertarget{thA}
\phantomsection
\begin{thA}
\label{thA}

For every $\alpha_{1}, \alpha_{2}, \alpha_{3}>0$ such that $\aI \geq \aII$, there exists a positive integer $n_{A}=n_{A}(\aI,\aII,\aIII)$ such that, for $n> n_{A}$,
 \begin{align*}
 R(C_{\llangle \alpha_{1} n \rrangle},C_{\llangle \alpha_{2} n \rrangle}, C_{\langle \alpha_{3} n \rangle }) = \max\{ 2\llangle \alpha_{1} n \rrangle \!+\! \llangle \alpha_{2} n \rrangle  \!-\!\text{\:}3,\text{\:}\half\llangle  \alpha_{1} n \rrangle  \!+\! \half\llangle \alpha_{2} n \rrangle  \!+\! \langle \alpha_{3} n \rangle \!-\! \text{\:}2\}.
\end{align*}
\end{thA}

In this paper, we consider the complementary case, that is where exactly one of the cycles is of even length and the others are odd. Specifically, again improving on the result of Figaj and \L uczak, we prove the following:

\phantomsection
\hypertarget{thC}
\phantomsection
\begin{thC}
\label{thC}

For every $\alpha_{1}, \alpha_{2}, \alpha_{3}>0$, there exists a positive integer $n_{C}=n_{C}(\aI,\aII,\aIII)$ such that, for $n> n_{C}$,
 \begin{align*}
 R(C_{\llangle \alpha_{1} n \rrangle},C_{\langle \alpha_{2} n \rangle}, C_{\langle \alpha_{3} n \rangle }) = \max\{
4\llangle \aI n \rrangle,
\llangle \aI n \rrangle+2\langle \aII n \rangle,
\llangle \aI n \rrangle+2\langle \aIII n \rangle\}-3
.
 \end{align*}
\end{thC}

\section{Lower bounds}
\label{ram:low}
\setlength{\parskip}{0.1in plus 0.05in minus 0.025in}

The first step in proving Theorem C is to exhibit three-edge-colourings of the complete graph on $$ \max\{
4\llangle \aI n \rrangle,
\llangle \aI n \rrangle+2\langle \aII n \rangle,
\llangle \aI n \rrangle+2\langle \aIII n \rangle\}-4
$$ vertices which do not contain any of the relevant coloured cycles, thus proving that 
$$ R(C_{\llangle \alpha_{1} n \rrangle},C_{\langle \alpha_{2} n \rangle}, C_{\langle \alpha_{3} n \rangle }) \geq \max\{
4\llangle \aI n \rrangle,
\llangle \aI n \rrangle+2\langle \aII n \rangle,
\llangle \aI n \rrangle+2\langle \aIII n \rangle\}-3
.$$
For this purpose, the well-known colourings shown in Figure~\ref{fig:lb0} suffice:

  \begin{figure}[!h]
\centering{
\mbox{\hspace{-8mm}{\includegraphics[width=54mm, page=17]{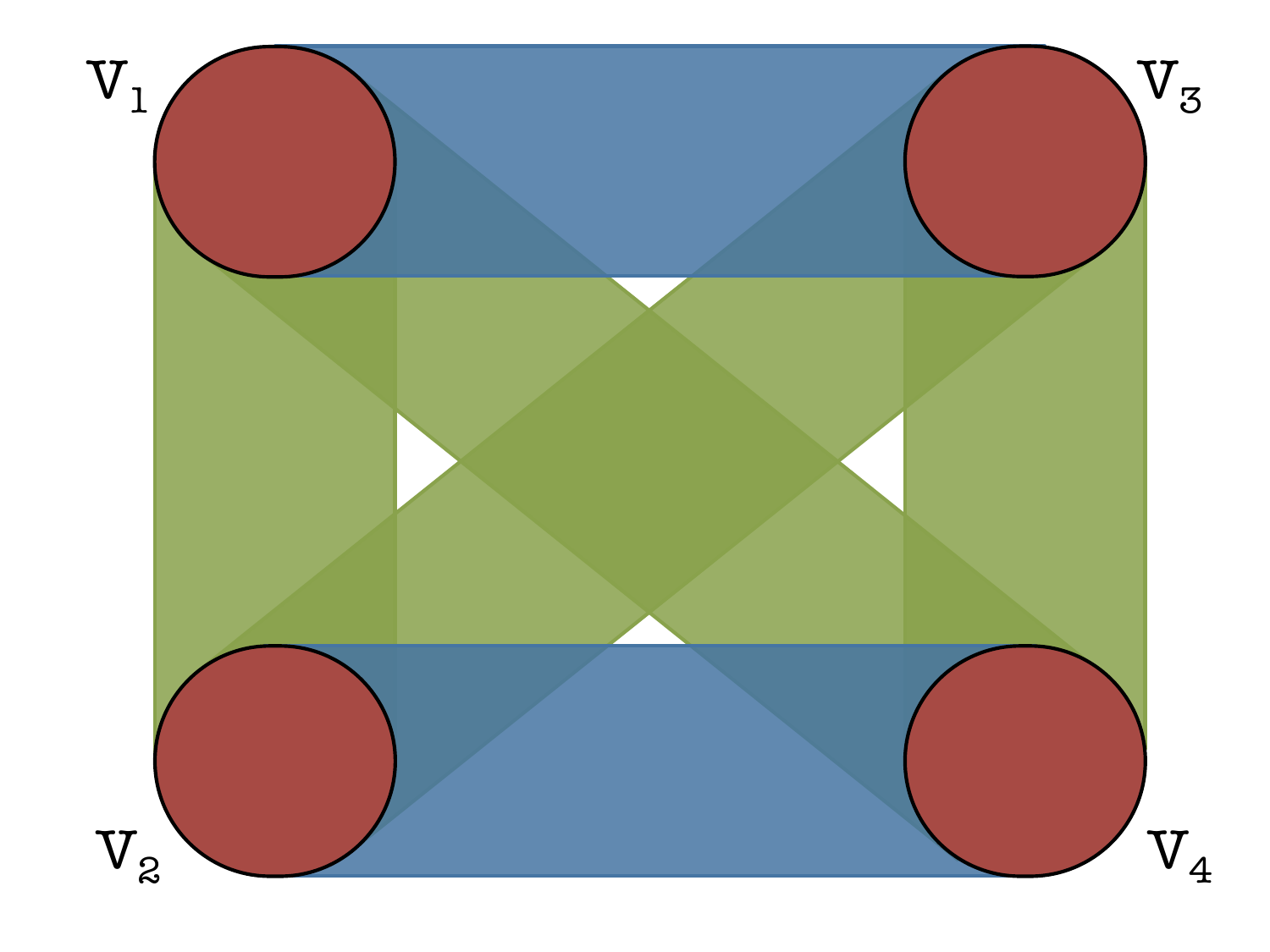}}\quad{\includegraphics[width=54mm, page=18]{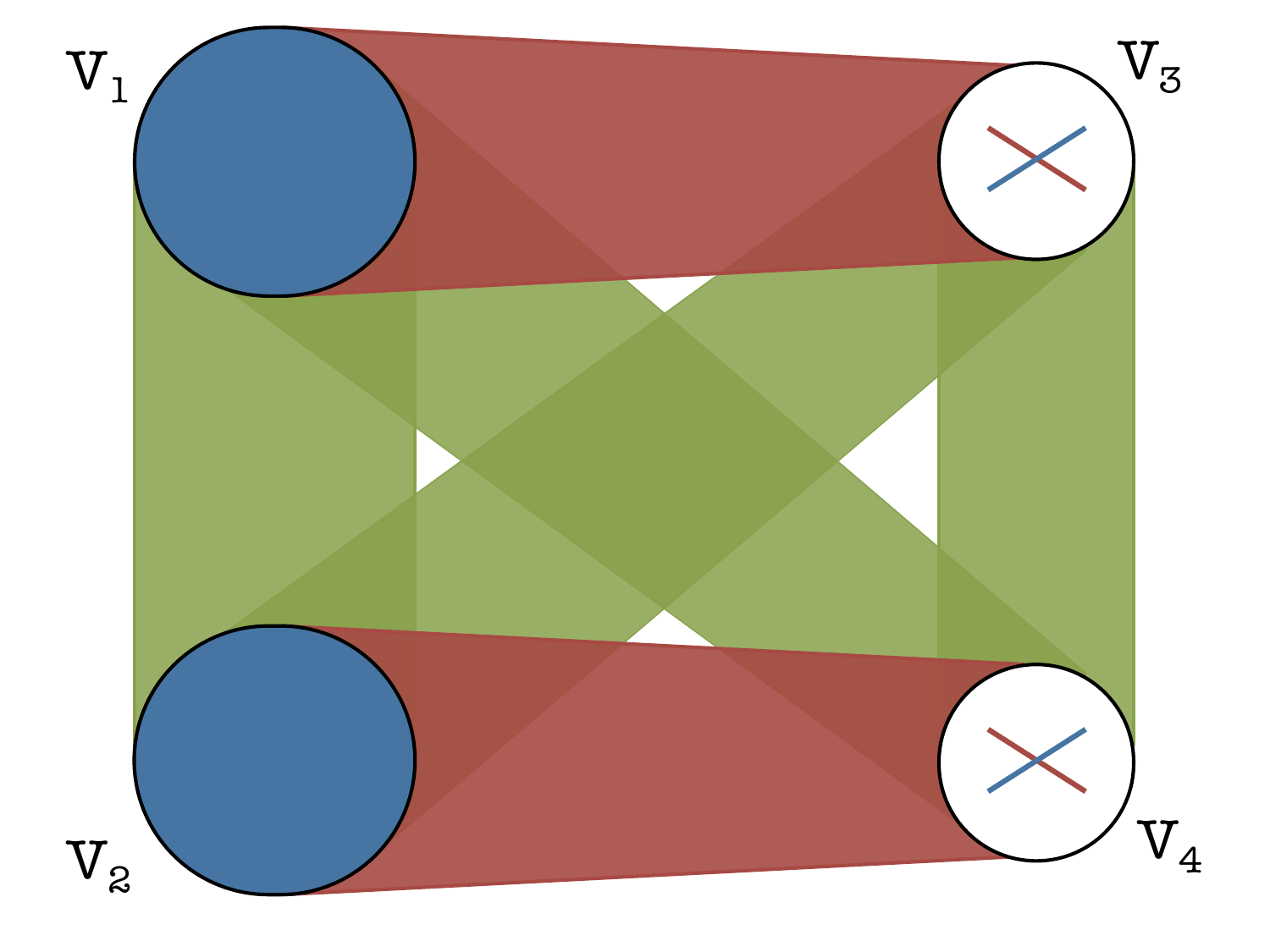}}\quad{\includegraphics[width=54mm, page=27]{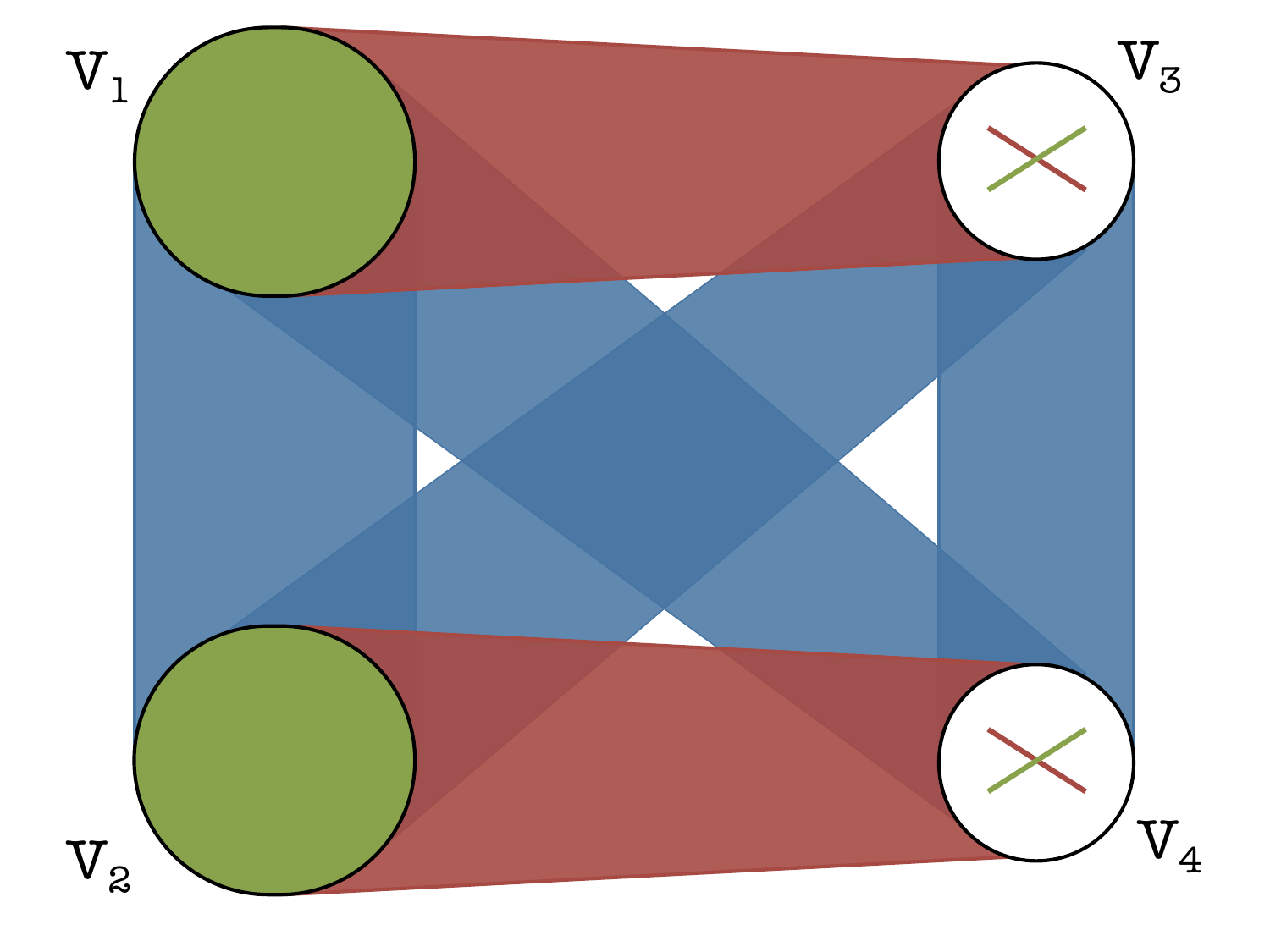}}}}
\setcounter{figure}{0}
\vspace{-4mm}
\caption{Extremal colouring for Theorem~C.}   
\label{fig:lb0}
\setcounter{figure}{0}
\end{figure}
\setcounter{figure}{1}

The first graph shown in Figure \ref{fig:lb0} has $4\llangle \aI n \rrangle -4$ vertices, divided into four equally-sized classes $V_{1}, V_{2}, V_{3}$ and $V_{4}$ such that all edges in $G[V_1]$, $G[V_2]$, $G[V_3]$ and $G[V_4]$ are coloured red, all edges in $G[V_1,V_3]$ and $G[V_2,V_4]$ are coloured blue and all edges in $G[V_1 \cup V_3,V_2 \cup V_4]$ are coloured green. 

The second graph shown in Figure \ref{fig:lb0} has $\llangle \aI n \rrangle + 2\langle \aII n \rangle- 4$ vertices, divided into four classes $V_{1}, V_{2}, V_{3}$ and $V_{4}$ with $|V_{1}|=|V_{2}|=\langle \alpha_{2} n \rangle-1$ and $|V_{3}|=|V_{4}|=\half\llangle \alpha_{1} n \rrangle-1$ such that all edges in $G[V_1,V_3]$ and $G[V_2,V_4]$ are coloured red, all edges in $G[V_1]$ and $G[V_2]$ are coloured blue, all edges in $G[V_1\cup V_3, V_2\cup V_4 ]$ are coloured green and all edges in $G[V_3]$ and $G[V_4]$ are coloured red or blue.

The third graph shown in Figure \ref{fig:lb0} has $\llangle \aI n \rrangle + 2\langle \aIII n \rangle- 4$ vertices, divided into four classes $V_{1}, V_{2}, V_{3}$ and $V_{4}$ with $|V_{1}|=|V_{2}|=\langle \alpha_{3} n \rangle-1$ and $|V_{3}|=|V_{4}|=\half\llangle \alpha_{1} n \rrangle-1$ such that all edges in $G[V_1,V_3]$ and $G[V_2,V_4]$ are coloured red, all edges in $G[V_1]$ and $G[V_2]$ are coloured green, all edges in $G[V_1\cup V_3, V_2\cup V_4 ]$ are coloured blue and all edges in $G[V_3]$ and $G[V_4]$ are coloured red or green.

Thus, it remains to prove the corresponding upper-bound. To do so, we combine regularity (as used in~\cite{Lucz},~\cite{FL2007},~\cite{FL2008}) with stability methods using a similar approach to~\cite{GyarSzem}, \cite{BenSko}, \cite{KoSiSk}, \cite{KoSiSk2}.

\section{Key steps in the proof}
\label{ram:key}

In order to complete the proof of Theorem C, we must show that, for $n$ sufficiently large, any three-colouring of~$G$, the complete graph on $$N= \max\{
4\llangle \aI n \rrangle,
\llangle \aI n \rrangle+2\langle \aII n \rangle,
\llangle \aI n \rrangle+2\langle \aIII n \rangle 
\}-3$$ vertices, will result in either a red cycle on $\llangle \aI n \rrangle$ vertices, a blue cycle on $\llangle \aII n \rrangle$ or a green cycle on~$\langle \aIII n \rangle$ vertices. 

The main steps of the proof are as follows: Firstly, we apply a version of the Regularity Lemma (Theorem~\ref{l:sze}) to give a partition $V_0\cup V_1\cup\dots\cup V_K$ of the vertices which is simultaneously regular for the red, blue and green spanning subgraphs of~$G$. Given this partition, we define the three-multicoloured reduced-graph $\cG$ on vertex set $V_1,V_2,\dots V_K$ whose edges correspond to the regular pairs. We colour the edges of the  reduced-graph with all those colours for which the corresponding pair has density above some threshold.
 \L uczak~\cite{Lucz} showed that, if the threshold is chosen properly, then the existence of a matching in a monochromatic connected-component of  the  reduced-graph implies the existence of a monochromatic cycle of the corresponding length in the original graph.

Thus, the key step in the proof of Theorem C will be to prove a Ramsey-type stability result for so-called connected-matchings (Theorem D). Defining a \textit{connected-matching} to be a matching with all its edges belonging to the same component, 
this result essentially says that, for every $\alpha_{1},\alpha_{2},\alpha_{3}>0$
 and every sufficiently large $k$, every three-multicolouring of a graph $\cG$ on slightly fewer than $K=\max\{4\aI, \aI+2\aII, \aI+2\aIII\}k$ vertices with sufficiently large minimum degree results in either a connected-matching on at least~$\aI k$ vertices in the red subgraph of $\cG$, a connected-matching on at least~$\aII k$ vertices in a non-bipartite component of the blue subgraph of $\cG$, a connected-matching on at least~$\aIII k$ vertices in a non-bipartite component of the green subgraph of $\cG$ or one of a list of particular structures which will be defined later.

In the case that $\cG$ contains a suitably large connected-matching in one of its coloured subgraphs, a blow-up result of Figaj and \L uczak (see Theorem~\ref{th:blow-up}) can be used to give a monochromatic cycle of the same colour in~$G$. If $\cG$ does not contain such a connected-matching, then the stability result gives us information about the structure of $\cG$. We then show that~$G$ has essentially the same structure which we exploit to force the existence of a monochromatic cycle.

In the next section, given a three-colouring of the complete graph on~$N$ vertices, we will define its three-multicoloured  reduced-graph. We will also discuss a version of the blow-up lemma of Figaj and \L uczak, which motivates this approach.
In Section~\ref{ram:defn}, we will deal with some notational formalities before proceeding in Section~\ref{s:struct} to define the structures we need and to give a precise formulation of the connected-matching stability result which we shall call Theorem~D.

In Section~\ref{s:pre1}, we give a number of technical lemmas needed for the proofs of Theorem~C and Theorem~D. Among these is a decomposition result of Figaj and \L uczak  which provides insight into the structure of the  reduced-graph. 
The hard work is done in Section~\ref{s:stabp}, where we prove Theorem~D, and in Sections~\mbox{\ref{s:p10}--\ref{s:p13}}, where we translate this result for connected-matchings into one for cycles, thus completing the proof of Theorem~C. 

\section{Cycles, matchings and regularity}
\label{ram:cmr}

Szemer\'{e}di's Regularity Lemma~\cite{SzemRegu} asserts that any sufficiently large graph can be approximated by the union of a bounded number of random-like bipartite graphs. 

Given a pair $(A,B)$ of disjoint subsets of the vertex set of a graph~$G$, we write $d(A,B)$ for the \textit{density} of the pair, that is, $d(A,B)=e(A,B)/|A||B|$ and say that such a pair  is $(\epsilon,G)$\textit{-regular} for some~$\epsilon>0$ if, for every pair $(A',B')$ with $A'\subseteq A$, $|A'|\geq \epsilon |A|$, $B' \subseteq B$, $|B'|\geq \epsilon |B|$, we have $ \left| d(A',B')-d(A,B) \right| <\epsilon.$
 
We will make use of a generalised version of Szemer\'edi's Regularity Lemma~in order to move from considering monochromatic cycles to considering monochromatic connected-matchings, the version below being a slight modification of one found, for instance, in~\cite{KomSim}: 

\begin{theorem}
\label{l:sze}
For every $\epsilon>0$ and every positive integer~$k_{0}$, there exists $K_{\ref{l:sze}}=K_{\ref{l:sze}}(\epsilon,k_{0})$ such that the following holds: For all graphs $G_{1},G_{2},G_{3}$ with $V(G_{1})=V(G_{2})=V(G_{3})=V$ and $|V|\geq K_{\ref{l:sze}}$, there exists a partition $\Pi =(V_{0},V_{1},\dots,V_{K})$ of $V$ such that
\begin{itemize}
\item [(i)]$k_{0}\leq K \leq K_{\ref{l:sze}}$;
\item [(ii)]$|V_0|\leq \epsilon |V|$;   
\item [(iii)]$|V_1|=|V_2|=\dots =|V_K|$; and
\item [(iv)] for each $i$, all but at most $\epsilon K$ of the pairs $(V_i,V_j)$, $1\leq i<j\leq K$, are simultaneously ($\epsilon,G_r$)-regular for $r=1,2,3$. 
\end{itemize}
\end{theorem}
Note that, given $\epsilon>0$ and graphs $G_1,G_2$ and $G_3$ on the same vertex set $V$, we call a partition $\Pi=(V_0,V_1,\dots,V_K)$ satisfying (ii)--(iv) \textit{$(\epsilon,G_1,G_2,G_3)$-regular}.

In what follows, given a three-coloured graph~$G$, we will use $G_1, G_2, G_3$ to refer to its monochromatic spanning subgraphs. That is $G_1$ (resp. $G_2, G_3$) has the same vertex set as~$G$ and includes, as an edge, any edge which in~$G$ is coloured red (resp. blue, green).
Then, given a three-coloured graph~$G$, we can use Theorem~\ref{l:sze} to define a partition which is simultaneously regular for $G_1$, $G_2$, $G_3$ and then define the three-multicoloured  reduced-graph $\cG$ as follows: 

\begin{definition}
\label{reduced}
Given $\epsilon>0$, $\xi>0$, a three-coloured graph $G=(V,E)$  and an $(\epsilon,G_1,G_2,G_3)$-regular partition $\Pi =(V_{0},V_{1},\dots,V_{K})$, we define the three-multicoloured $(\epsilon,\xi,\Pi)$-reduced-graph $\cG=(\cV,\cE)$ by:
\begin{align*}
\mathcal{V}&=\{V_{1},V_{2},\dots ,V_{K}\}, \text{\quad\quad}
\mathcal{E}&=\{ V_{i}V_{j} : (V_{i},V_{j}) \text{ is simultaneously } (\epsilon,G_{r})\text{-regular for }r=1,2,3\},
\end{align*}
where $V_{i}V_{j}$ is coloured with all colours~$r$ such that $d_{G_r}(V_i,V_j)\geq\xi$.
\end{definition}

One well known fact about regular pairs is that they contain long paths. This is summarised in the following lemma, which is a slight modification of one found in~\cite{Lucz}:

\begin{lemma}
\label{longpath}
For every~$\epsilon$ such that $0\leq \epsilon < 1/600 $ and every $k\geq1/\epsilon$, the following holds: Let~$G$ be a bipartite graph with bipartition $V(G)=V_1\cup V_2$ such that $|V_1|,|V_2|\geq k$, the pair $(V_1,V_2)$ is $\epsilon$-regular and $e(V_1,V_2)\geq \epsilon^{1/2} |V_1||V_2|$. Then, for every integer~$\ell$ such that $1\leq \ell \leq k-2\epsilon^{1/2} k$ and every $v' \in V_1$, $v'' \in V_2$ such that $d(v'),d(v'')\geq \tfrac{2}{3}\epsilon^{1/2}k$,~$G$ contains a path of length $2\ell +1$ between~$v'$ and~$v''$.
\end{lemma}

A \textit{matching} is a collection of pairwise vertex-disjoint edges. Note that, in what follows, we will sometimes abuse terminology and, where appropriate, refer to a matching by its vertex set rather than its edge set.
We call a matching with all its vertices in the same component of~$G$ a \textit{connected-matching} and note that we say a connected-matching is \textit{odd} if the component containing the matching also contains an odd cycle. 

The following theorem makes use of the lemma~above to blow up large connected-matchings in the  reduced-graph to cycles (of appropriate length and parity) in the original. This facilitates our approach to proving Theorem~C in that it allows us to shift our attention away from cycles to connected-matchings, which turn out to be somewhat easier to find. Figaj and \L uczak~\cite[Lemma~3]{FL2008} proved a more general version of this theorem in a slightly different context (they considered any number of colours and any combination of parities and used a different threshold for colouring the reduced-graph):

\begin{theorem}
\label{th:blow-up}

For all $c_1,c_2,c_3,d, \eta>0$ such that
$0<\eta<\min\{0.01, (64c_1+64c_2+64c_3)^{-1}\}$, there exists $n_{\ref{th:blow-up}}=n_{\ref{th:blow-up}}(c_1,c_2,c_3,d,\eta)$ such that, for $n>n_{\ref{th:blow-up}}$, the following holds:

Given $\aI, \aII, \aIII$  such that $0<\aI,\aII,\aIII\leq2$, and $\xi$ such that $\eta\leq \xi\leq\tfrac{1}{3}$, a complete three-coloured graph $G=(V,E)$ on
$$N=c_{1}\llangle \aI n\rrangle+c_{2}\langle \aII n \rangle +c_{3}\langle\alpha_{3}n\rangle-d$$
vertices and an $(\eta^4,G_1,G_2,G_3)$-regular partition $\Pi =(V_{0},V_{1},\dots,V_{K})$ for some $K>8(c_1+c_2+c_3)^2/\eta$, letting $\cG=(\cV,\cE)$ be the three-multicoloured $(\eta^4,\xi,\Pi)$-reduced-graph of~$G$ on~$K$ vertices, and letting $k$ be an integer such that 
$$
c_{1}\alpha_{1}k+c_{2}\aII k+c_{3}\aIII k - \eta k \leq  K \leq c_{1}\alpha_{1}k+c_{2}\aII k+c_{3}\aIII k - \half\eta k,$$
\begin{itemize}
\item[(i)] if $\cG$ contains a red connected-matching on at least $\alpha_{1}k$ vertices, then~$G$ contains a red cycle on $\llangle \alpha_{1} n\rrangle$ vertices;

\item[(ii)] if $\cG$ contains a blue odd connected-matching on at least $\alpha_{2}k$ vertices, then~$G$ contains a blue cycle on $\langle \alpha_{2} n\rangle$ vertices;

\item[(iii)] if $\cG$ contains a green odd connected-matching on at least $\alpha_{3}k$ vertices, then~$G$ contains a green cycle on $\langle\alpha_{3}n\rangle $ vertices.
\end{itemize}
\end{theorem}

\section{Definitions and notation}
\label{ram:defn}

Recall that, given a three-coloured graph~$G$, we refer to the first, second and third colours as red, blue and green respectively and use $G_1, G_2, G_3$ to refer to the monochromatic spanning subgraphs of $G$. In what follows, if~$G_1$ contains the edge $uv$, we say that $u$ and $v$ are \textit{red neighbours} of each other in $G$. Similarly, if $uv\in E(G_2)$, we say that $u$ and $v$ are \textit{blue neighbours} and, if $uv\in E(G_3)$, we say that that $u$ and $v$ are \textit{green neighbours}.

We say that a graph~$G=(V,E)$ on~$N$ vertices is $a$-\emph{almost-complete} for $0\leq a\leq N-1$ if its minimum degree~$\delta(G)$ is at least $(N-1)-a$. Observe that, if~$G$ is $a$-almost-complete and $X\subseteq V$, then $G[X]$ is also $a$-almost-complete.
We say that a graph~$G$ on~$N$ vertices is $(1-c)$-\emph{complete} for $0\leq c\leq 1$ if it is $c(N-1)$-almost-complete, that is, if $\delta(G)\geq (1-c)(N-1)$. Observe that, for $c\leq \half$, any $(1-c)$-complete graph is connected.

We say that a bipartite graph~$G=G[U,W]$ is $a$-\emph{almost-complete} if every $u\in U$ has degree at least $|W|-a$ and every $w\in W$ has degree at least $|U|-a$. Notice that, if~$G[U,W]$ is $a$-almost-complete and $U_1\subseteq U, W_1\subseteq W$, then $G[U_1,W_1]$ is $a$-almost-complete. 
We say that a bipartite graph~$G=G[U,W]$ is $(1-c)$-\emph{complete} if every $u\in U$ has degree at least $(1-c)|W|$ and every $w\in W$ has degree at least $(1-c)|U|$. Again, notice that, for $c< \half$, any $(1-c)$-complete bipartite graph $G[U,W]$ is connected, provided that~$U,W\neq \emptyset$.

We say that a graph~$G$ on~$N$ vertices is $c$-\emph{sparse} for $0<c<1$ if its maximum degree is at most $c(N-1)$. We say a bipartite graph~$G=G[U,W]$ is $c$-\emph{sparse} if every $u\in U$ has degree at most $c|W|$ and every vertex $w\in W$ has degree at most $c|U|$.

For vertices~$u$ and~$v$ in a graph~$G$, we will say that the edge $uv$ is \emph{missing} if~$uv\notin E(G)$.

\section {Connected-matching stability result}
\label{s:struct}

Before proceeding to state Theorem~D, 
we define the coloured structures we will need.

\begin{definition}
\label{d:H}

For $x_{1}, x_{2}, c_1,c_2$ positive, $\gamma_1,\gamma_2$ colours, let $\cH(x_{1},x_{2}, c_1,c_2, \gamma_1,\gamma_2)$ be the class of edge-multicoloured graphs defined as follows:  A given two-multicoloured graph $H=(V,E)$ belongs to~$\cH$ if its vertex set can be partitioned into $X_{1}\cup X_{2}$ such that
\begin{itemize}
\item[(i)] $|X_{1}|\geq x_{1}, |X_{2}|\geq x_{2} $;
\item[(ii)] $H$ is $c_1$-almost-complete; and
\item[(iii)] defining $H_1$ to be the spanning subgraph induced by the colour $\gamma_1$ and $H_2$ to be the subgraph induced by the colour $\gamma_2$,
\begin{itemize}
\item[(a)] $H_1[X_{1}]$ is $(1-c_2)$-complete and $H_2[X_{1}]$ is $c_2$-sparse, 
\item[(b)] $H_2[X_1,X_2]$ is $(1-c_2)$-complete and $H_1[X_1,X_2]$ is $c_2$-sparse.
\end{itemize}
\end{itemize}
\end{definition}

\begin{definition}
\label{d:J}

For $x, c$ positive, $\gamma_1,\gamma_2$ colours, let $\cJ(x,c,\gamma_1,\gamma_2)$ be the class of edge-multicoloured graphs defined as follows:  A given two-multicoloured graph $H=(V,E)$ belongs to~$\cJ$ if the vertex set of~$V$ can be partitioned into $X_{1}\cup X_{2}$ such that

\begin{itemize}
\item[(i)] $|X_{1}|,|X_{2}|\geq x$;
\item[(ii)] $H$ is $c$-almost-complete; and
\item[(iii)] (a) all edges present in $H[X_1], H[X_2]$ are coloured exclusively with colour $\gamma_1$, 
\item[{~}] (b) all edges present in $H[X_1,X_2]$ are coloured exclusively with colour $\gamma_2$. 
\end{itemize}
\end{definition}

  \begin{figure}[!h]
\centering{
\mbox{\hspace{-4mm}
{\includegraphics[height=22mm, page=1]{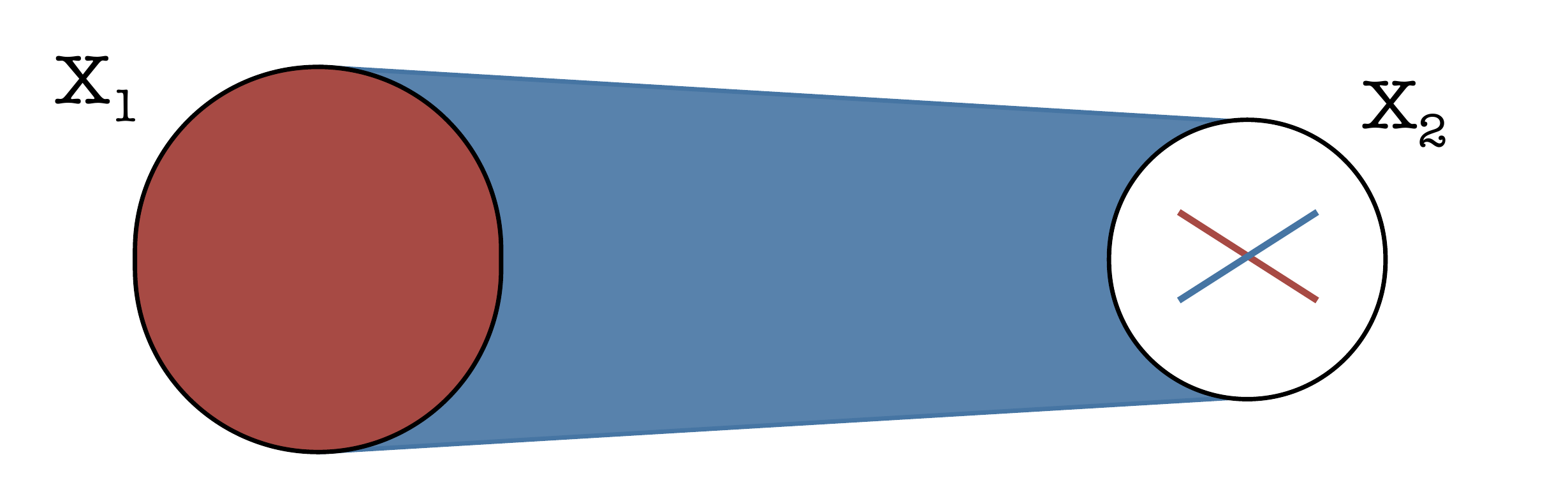}}
\quad{\includegraphics[height=22mm, page=3]{CH2-Figs-wide.pdf}}}}
\caption{$H\in\cH(x_1,x_2,c_1,c_2,\text{red},\text{blue})$ and $H\in\cJ(x,c,\text{red},\text{green})$.}   
\label{fig:lb1}
\end{figure}
\begin{definition}
\label{d:L}

For $x, c$ positive, $\gamma_1, \gamma_2, \gamma_3$ colours, let $\cL(x, c, \gamma_1, \gamma_2,\gamma_3)$ be the class of edge-multicoloured graphs defined as follows:  A given three-multicoloured graph $H=(V,E)$ belongs to~$\cL$, if its vertex set can be partitioned into $X_{1}\cup X_{2}\cup Y_{1}\cup Y_{2}$ such that
\begin{itemize}
\item[(i)] $|X_{1}|, |X_{2}|, |Y_{1}|, |Y_{2}|\geq x$;
\item[(ii)] $H$ is $c$-almost-complete; and
\item[(iii)] (a) all edges present in $H[X_1]$, $H[X_2]$, $H[Y_1]$ and $H[Y_2]$ are coloured exclusively with colour $\gamma_1$,
\item[{~}] (b) all edges present in $H[X_1,Y_1]$ and $H[X_2,Y_2]$ are coloured exclusively with colour $\gamma_2$,
\item[{~}] (c) all edges present in $H[X_1,X_2]$ and $H[Y_1,Y_2]$ are coloured exclusively with colour $\gamma_3$,
\item[{~}] (d) all edges present in $H[X_1,Y_2]$ and $H[X_2,Y_1]$ are coloured colours $\gamma_2$ or $\gamma_3$ only.
\end{itemize}
\end{definition}

\begin{figure}[!h]
\vspace{-1mm}
\centering
\includegraphics[width=64mm, page=1]{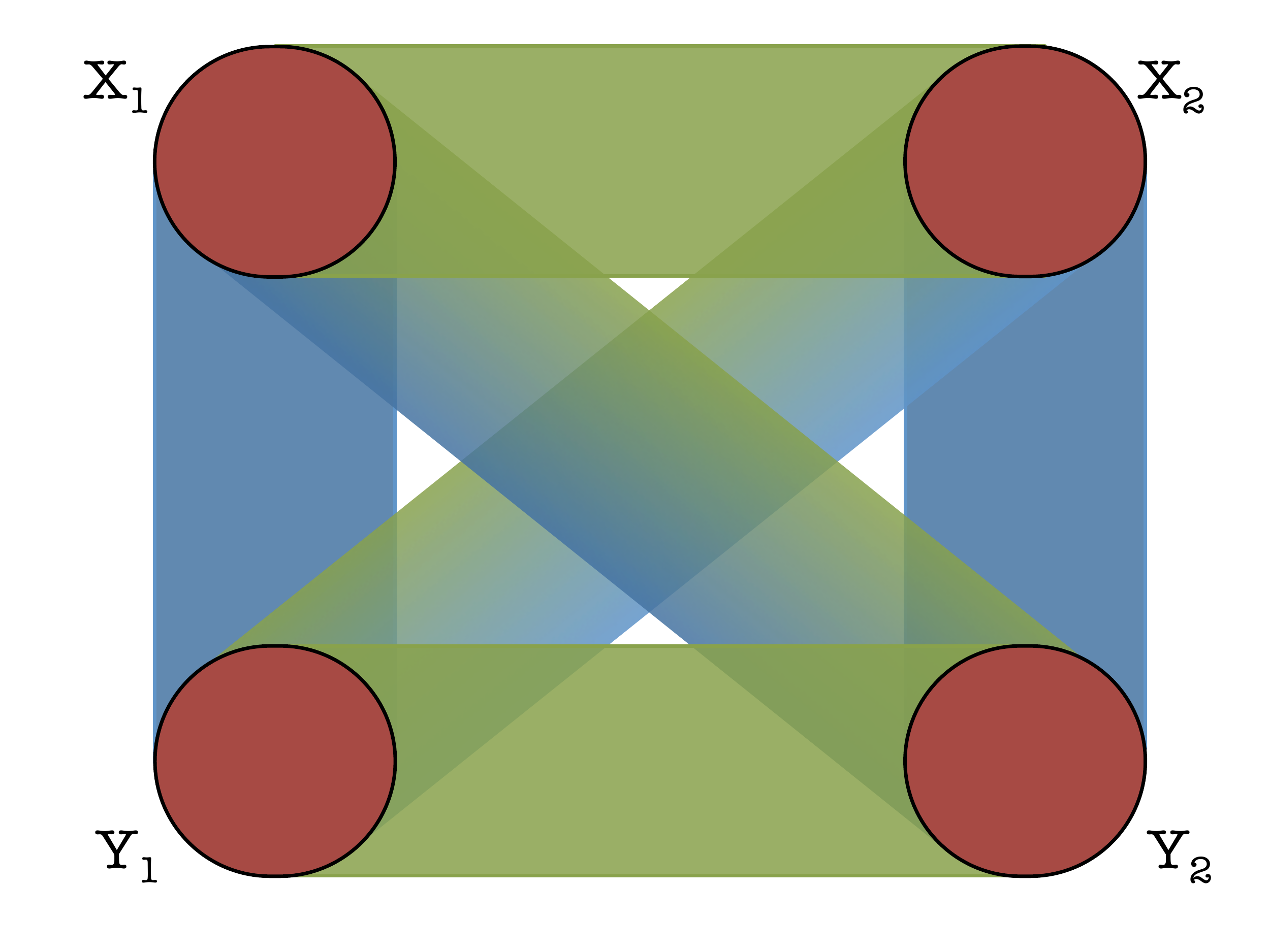}
\vspace{-3mm}\caption{$H\in \cL(x_1,c,\text{red},\text{blue},\text{green})$.}
\end{figure}

Having defined the coloured structures, we are in a position to state the main technical result, that is, the connected-matching stability result. The proof of this result follows in Section~\ref{s:stabp}.

\begin{thD}
\label{thD}
\label{th:stabnew}
For every $\alpha_{1},\alpha_{2},\alpha_{3}>0$, letting $$c=\max\{4\aI, \aI+2\aII,\aI+2\aIII\},$$
there exists $\eta_{D}=\eta_{D}(\aI,\aII,\aIII)$ and $k_{D}=k_{D}(\aI,\aII,\aIII,\eta)$ 
such that, for every $k>k_{D}$ and 
every~$\eta$ such that $0<\eta<\eta_{D}$, every 
three-multicolouring of~$G$, a $(1-\eta^4)$-complete graph on 
$$(c-\eta)k\leq K\leq(c-\half\eta)k$$ vertices  results in the graph containing at least one of the following:
\vspace{-2mm}
\begin{itemize}
\item [(i)]	a red connected-matching on at least $\aI
 k$ vertices;
\item [(ii)]  a blue odd connected-matching on at least $\aII k$ vertices;
\item [(iii)]  a green odd connected-matching on at least $\aIII k$ vertices;
\item [(iv)] subsets of vertices $W$, $X$ and $Y$ such that $X\cup Y\subseteq W$, $X\cap Y=\emptyset$, $|W|\geq(c-\eta^{1/2})k$, every $\gamma$-component of $G[W]$ is odd, $G[X]$ contains a two-coloured spanning subgraph $H$ from $\cH_{1}\cup\cH_{2}$ and $G[Y]$ contains a a two-coloured spanning subgraph $K$ from $\cH_{1}\cup\cH_{2}$, where~{\phantom{nnn}}
\vspace{-2mm}
\begin{align*}
\cH_1&=\cH\left((\aI-2\eta^{1/64})k,(\half\alpha_{*}-2\eta^{1/64})k,4\eta^2 k,\eta^{1/64},\text{red},\gamma\right),\text{ } 
\\ \cH_2&=\cH\left((\alpha_{*}-2\eta^{1/64})k,(\half\aI-2\eta^{1/64})k,4\eta^2 k,\eta^{1/64},\gamma,\text{red}\right),
\end{align*}

\vspace{-3mm}
for $(\alpha_{*},\gamma)\in\{(\aII,\text{blue}),(\aIII,\text{green})\}$;
\item[(v)] 
disjoint subsets of vertices $X$ and $Y$ such that $G[X]$ contains a two-coloured spanning subgraph~$H$ from $\cH_{2}^*\cup\cJ_b$ and $G[Y]$ contains a two-coloured spanning subgraph $K$ from $\cH_{2}^*\cup\cJ_b$, where
\vspace{-2mm}
\begin{align*}
\cH_{2}^*&=\cH\left((\beta-2\eta^{1/32})k,(\half\aI-2\eta^{1/32})k,4\eta^4 k,\eta^{1/32},\gamma,\text{red}\right),\text{ } 
\\ \cJ_b&=\cJ\left((\aI-18\eta^{1/2}), 4\eta^4 k, \text{red}, \gamma\right),
\end{align*}

\vspace{-3mm}
for $\beta=\max\{\aII,\aIII\}$ and $\gamma\in\{\text{blue}, \text{green}\}$;
\item[(vi)]  a subgraph $H$ from $\cL=\cL\left((\half\alpha+\tfrac{1}{4}\eta)k,4\eta^4k, \text{red}, \text{blue}, \text{green}\right).$
\end{itemize}

Furthermore, 
\begin{itemize}
\item[(iv)] occurs only if $\aI\leq\max\{\aII,\aIII\}\leq\alpha_{*}+24\eta^{1/4}$. Additionally,  $H$ and $K$ belong to $\cH_{1}$ if $\alpha_{*}\leq(1-\eta^{1/16})\aI$ and belong to $\cH_{2}$ if $\aI\leq(1-\eta^{1/16})\alpha_{*}$;
 \item[(v)] occurs only if $\aI\leq\beta$. Additionally, $H$ and $K$ may belong to $\cH_2^*$ only if $\aI\leq(1-\eta^{1/16})\beta$ and may belong to $\cJ$ only if $\beta<(\tfrac{3}{2}+2\eta^{1/4})\aI$; and
 \item[(vi)] occurs only if $\aI\geq\max\{\aII,\aIII\}$.
 \end{itemize} 
\end{thD}

This result forms a partially strengthened analogue of the main technical result of the paper of Figaj and \L uczak~\cite{FL2008}. In that paper, Figaj and \L uczak considered a similar graph but on slightly more than $\max\{4\aI, \aI+2\aII, \aI+2\aIII\}k$ vertices and proved the existence of a connected-matching, whereas we consider a graph on slightly fewer vertices and prove the existence of either a monochromatic connected-matching or a particular structure.

\section{{Tools}}
\label{s:pre1}
In this section, we summarise results that we shall use later in our proofs beginning with some results on Hamiltonicity including Dirac's Theorem, which gives us a minimum-degree condition for Hamiltonicity:

\begin{theorem}[Dirac's Theorem~\cite{Dirac52}]
\label{dirac}
If~$G$ is a graph on~$n\geq3$ vertices such that every vertex has degree at least $\half n$, then~$G$ is Hamiltonian, that is, $G$ contains a cycle of length exactly $n$.
\end{theorem}

Observe then that, by Dirac's Theorem, any $c$-almost-complete graph on $n$ vertices is Hamiltonian, provided that $c\leq \half n-1$. Then, since almost-completeness is a hereditary property, we may prove the following corollary:

\begin{corollary}
\label{dirac1a}
If $G$ is a $c$-almost-complete graph on $n$ vertices, then, for any integer~$m$ such that $2c+2\leq m\leq n$, $G$ contains a cycle of length $m$.
\end{corollary}

Dirac's Theorem may be used to assert the existence of Hamiltonian paths in a given graph as follows:

\begin{corollary}
\label{dirac2}
If $G=(V,E)$ is a simple graph on~$n\geq4$ vertices such that every vertex has degree at least $\half n+1$, then any two vertices of~$G$ are joined by a Hamiltonian path.
\end{corollary}

For balanced bipartite graphs, we make use of the following result of Moon and Moser:

\begin{theorem}[\cite{moonmoser}]
\label{moonmoser}
If~$G=G[X,Y]$ is a simple bipartite graph on~$n$ vertices such that $|X|=|Y|=\half n$ and $d(x)+d(y)\geq \half n+1$ for every $xy\notin E(G)$, then~$G$ is Hamiltonian. 
\end{theorem}

Observe that, by the above, any $c$-almost-complete balanced bipartite graph on $n$ vertices is Hamiltonian, provided that $c\leq \tfrac{1}{4}n-\half$. Then, since almost-completeness is a hereditary property, we may prove the following corollary:

\begin{corollary}
\label{moonmoser2}
If $G=G[X,Y]$ is $c$-almost-complete bipartite graph, then, for any even integer $m$ such that $4c+2\leq m\leq 2\min\{|X|,|Y|\}$, $G$ contains a cycle on $m$ vertices.
\end{corollary}

For bipartite graphs which are not balanced, we make use of the Lemma below:
\begin{lemma}
\label{bp-dir}
If $G=G[X_1,X_2]$ is a simple bipartite graph on $n\geq 4$ vertices such that $|X_1|>|X_2|+1$ and every vertex in~$X_2$ has degree at least $\half n +1$, then any two vertices $x_1,x_2$ in $X_1$ such that $d(x_2)\geq 2$ are joined by a path which visits every vertex of~$X_2$.
\end{lemma}

\begin{proof}
Observe that $\half n+1=\half|X_1|+\half|X_2|+1=|X_1|-(\half|X_1|-\half|X_2|-1)$ so any pair of vertices in~$X_2$ have at least $|X_1|-(|X_1|-|X_2|-2)$ common neighbours and, thus, at least $|X_1|-(|X_1|-|X_2|)\geq |X_2| $ common neighbours distinct from~$x_1,x_2$. 
Then, ordering the vertices of~$X_2$ such that the first vertex is a neighbour of~$x_1$ and the last is a neighbour of~$x_2$, greedily construct the required path from~$x_1$ to~$x_2$.
\end{proof} 

\begin{corollary}
\label{bp-dir2}
If $G=G[X_1,X_2]$ is a simple bipartite graph on $n\geq 5$ vertices such that 
$|X_1|>|X_2|$ 
and every vertex in~$X_2$ has degree at least $\half(n +1)$, then any two vertices $x_1\in X_1$ and $x_2\in X_2$ such that $d(x_1),d(x_2)\geq 2$ are joined by a path which visits every vertex of~$X_2$.
\end{corollary}

For graphs with 
a few vertices of small degree, we make use of the following result of Chv\'atal:
\begin{theorem}[\cite{Chv72}]
\label{chv}
If $G$ is a simple graph on $n\geq3$ vertices with degree sequence $d_1\leq d_2 \leq \dots \leq d_n$ such that $$d_k\leq k \leq \frac{n}{2} \implies d_{n-k} \geq n-k,$$ then~$G$ is Hamiltonian.
\end{theorem}

We also make extensive use of the theorem of Erd\H{o}s and Gallai:

\begin{theorem}[\cite{ErdGall59}]
\label{th:eg}
Any graph on~$K$ vertices with at least~$\frac{1}{2}(m-1)(K-1)+1$ edges, where $3\leq m \leq K$, contains a cycle of length at least~$m$.
\end{theorem}

Observing that a cycle on $m$ vertices contains a connected-matching on at least~$m-1$ vertices, the following is an immediate consequence of the above.

\begin{corollary}
\label{l:eg}
For any graph~$G$ on~$K$ vertices and any~$m$ such that $3 \leq m \leq K$, if the average degree $d(G)$ is at least $m$, then~$G$ contains a connected-matching on at least~$m$ vertices.
\end{corollary}

The following decomposition lemma~of Figaj and \L uczak~\cite{FL2008} also follows from the theorem of Erd\H{o}s and Gallai and is crucial in establishing the structure of a graph not containing large connected-matchings of the appropriate parities:

\begin{lemma}[{\cite[Lemma~9]{FL2008}}]
\label{l:decomp}
For any graph~$G$ on~$K$ vertices and any~$m$ such that $3\leq m \leq K$, if no odd component of~$G$ contains a matching on at least~$m$ vertices, then there exists a partition $V=V'\cup V''$ such that
\begin{itemize}
\item [(i)] $G[V']$ is bipartite;
\item [(ii)] every component of $G''=G[V'']$ is odd;
\item [(iii)] $G[V'']$ has at most $\half m |V(G'')|$ edges; and
\item [(iv)] there are no edges in $G[V',V'']$.
\end{itemize}
\end{lemma}

The following pair of lemmas allow us to find large connected-matchings in almost-complete bipartite graphs:

\begin{lemma}[{\cite[Lemma~10]{FL2008}}]
\label{l:ten}
Let $G=G[V_1,V_2]$ be a bipartite graph with bipartition $(V_{1},V_{2})$, where $|V_{1}|\geq|V_{2}|$, which has at least $(1-\epsilon)|V_{1}||V_{2}|$ edges for some $\epsilon$ such that $0<\epsilon<0.01$. Then,~$G$ contains a connected-matching on at least $2(1-3\epsilon)|V_{2}|$ vertices.
\end{lemma}

Notice that, if~$G$ is a $(1-\epsilon)$-complete bipartite graph with bipartition $(V_1,V_2)$, then we may immediately apply the above to find a large connected-matching in~$G$.

\begin{lemma}
\label{l:eleven}
Let $G=G[V_1,V_2]$ be a bipartite graph with bipartition $(V_1,V_2)$. If $\ell$ is a positive integer such that $|V_1|\geq|V_2|\geq \ell$ and~$G$ is $a$-almost-complete for some $a$ such that $0<a/\ell<0.5$, then~$G$ contains a connected-matching on at least $2|V_2|-2a$ vertices.
\end{lemma}

\begin{proof}
Observe that~$G$ is $(1-a/\ell)$-complete. Therefore, since $a/\ell<0.5$,~$G$ is connected. Thus, it suffices to find a matching of the required size. Suppose that we have found a matching with vertex set~$M$ such that $|M|=2k$ for some $k<|V_2|-a$ and consider a vertex $v_2\in V_2\backslash M$. Since $G$ is $a$-almost-complete, $v_2$ has at least $|V_1|-a$ neighbours in~$|V_1|$ and thus at least one neighbour in $v_1\in V_1\backslash M$. Then, the edge $v_1v_2$ can be added to the matching and thus, by induction, we may obtain a matching on $2|V_2|-2a$ vertices.
\end{proof}

We recall two further results of Figaj and \L uczak: The first is a technical result from~\cite{FL2008}. 
The second is part of the main result from~\cite{FL2008}. 
Note that these results can be immediately extended to multicoloured graphs:

\begin{lemma}[{\cite[Lemma~13]{FL2008}}]
\label{l:thirteen}
For every $\alpha, \beta>0$, $v\geq0$ and $\eta$ such that
$0<\eta<0.01 \min\{ \alpha,\beta\}$, 
there exists $k_{\ref{l:thirteen}}=k_{\ref{l:thirteen}}(\alpha,\beta,v,\eta)$ such that, for every $k>k_{\ref{l:thirteen}}$, the following holds:

Let $G=(V,E)$ be a graph obtained from a $(1-\eta^4)$-complete graph on at least $$ \half\Big(\max\Big\{\alpha+\beta+\max \left\{ 2v, \alpha, \beta \right\},3\alpha+\max\left\{2v,\alpha\right\} \Big\} + 10\eta^{1/2} \Big)k $$
vertices by removing all edges contained within a subset $W \subseteq V$ of size at most~$vk$. Then, every two-multicolouring of the edges of~$G$ results in {either} a red connected-matching on at least $(\alpha+\eta)k$ vertices {or} a blue odd connected-matching on at least $(\beta+\eta)k$ vertices.
\end{lemma}

\begin{lemma}[{\cite[Theorem~1ii]{FL2008}}]
\label{l:fourteen}
For every $\aI, \aII, \aIII>0$, there exists $\eta_{\ref{l:fourteen}}=\eta_{\ref{l:fourteen}}(\aI,\aII,\aIII)$ such that, for every $0<\eta<\eta_{\ref{l:fourteen}}$, 
there exists $k_{\ref{l:fourteen}}=k_{\ref{l:fourteen}}(\alpha,\beta,v,\eta)$ such that the following holds:

For every $k>k_{\ref{l:thirteen}}$ and every $(1-\eta^4)$-complete graph $G$ on $$ K\geq\left(\max\left\{
\aI+2\aII,
2\aI+\aII,
\half\aI+\half\aII+\aIII
\right\}
+10\eta^{1/4}\right)k$$
vertices, every three-colouring of the edges of $G$ results in one of the following:
\begin{itemize}
\item [(i)]	a red connected-matching on at least $\alpha
 k$ vertices;
\item [(ii)]  a blue connected-matching on at least $\alpha
 k$ vertices;
\item [(iii)]  a green odd connected-matching on at least $\alpha
 k$ vertices.
\end{itemize}
\end{lemma}

We also make use of the following pair of results dealing with different combinations of parities. The first is an immediate consequence of the main technical result of \cite{DF1} and \cite{DF2} which can be also be found (with additional detail) in \cite{FERT}. The second is an immediate consequence of the key technical result from \cite{KoSiSk}:

\begin{theorem}[{\cite[Theorem~B]{FERT}}]
\label{thBc}
For every $0<\alpha_{1},\alpha_{2},\alpha_{3}\leq 1$ such that $\aIII\leq \tfrac{3}{2}\max\{\aI,\aII\}+\half\min\{\aI,\aII\}-11\eta^{1/2},$ letting $c=\max\{2\aI+\aII, \aI+2\aII, \half\aI+\half\aII+\aIII\},$
there exists $\eta_{\ref{thBc}}=\eta_{\ref{thBc}}(\aI,\aII,\aIII)$ and $k_{\ref{thBc}}=k_{\ref{thBc}}(\aI,\aII,\aIII,\eta)$ 
such that, for every $k>k_{\ref{thBc}}$ and 
every~$\eta$ such that $0<\eta<\eta_{\ref{thBc}}$, every 
three-multicolouring of~$G$, a $(1-\eta^4)$-complete graph on 
$(c-\eta)k\leq K\leq(c-\half\eta)k$ vertices,  results in the graph containing at least one of the following:
\begin{itemize}
\item [(i)]	a red connected-matching on at least $\aI
 k$ vertices;
\item [(ii)]  a blue connected-matching on at least $\aII k$ vertices;
\item [(iii)]  a green odd connected-matching on at least $\aIII k$ vertices;
\item [(iv)] disjoint subsets of vertices $X$ and $Y$ such that $G[X]$ contains a two-coloured spanning subgraph~$H$ from $\cH_1^B\cup\cH_2^B$ and $G[Y]$ contains a two-coloured spanning subgraph $K$ from $\cH_1^B\cup\cH_2^B$ where 
\vspace{-2mm}
\begin{align*}
\cH_1^B&=\cH\left((\aI-2\eta^{1/32})k,(\half\aII-2\eta^{1/32})k,4\eta^4 k,\eta^{1/32},\text{red},\text{blue}\right),\text{ } 
\\ \cH_2^B&=\cH\left((\aII-2\eta^{1/32})k,(\half\aI-2\eta^{1/32})k,4\eta^4 k,\eta^{1/32},\text{blue},\text{red}\right).
\end{align*}
\end{itemize}
Furthermore, in (iv) $H_1,H_2\in \cH_1$ if $\aII\leq\aI-\eta^{1/16}$ and $H_1,H_2\in \cH_2$ if $\aI\leq\aII-\eta^{1/16}$.
\end{theorem}

\begin{lemma}[{\cite[Theorem 6]{KoSiSk}}]
\label{kss07a-7}
There exists $\eta_{\ref{kss07a-7}}>0$ such that, for every $0<\alpha\leq1$ and $\eta$ such that $0<\eta<\eta_{\ref{kss07a-7}}$, there exists $k_{\ref{kss07a-7}}=k_{\ref{kss07a-7}}(\alpha,\eta)$ such that the following holds: For every $k>k_{\ref{kss07a-7}}$ and every $(1-\eta^4)$-complete graph $G$ on $(4\alpha-\eta)k\leq K \leq (4\alpha+\eta)k$ vertices, every three-colouring of the edges of $G$ results in the graph containing at least one of the following:
\begin{itemize}
\item [(i)]	a red connected-matching on at least $\alpha k$ vertices;
\item [(ii)]  a blue odd connected-matching on at least $\alpha k$ vertices;
\item [(iii)]  a green odd connected-matching on at least $\alpha k$ vertices;
\item [(iv)]  a subgraph from
 $\cL\left((\half\alpha+\tfrac{1}{4}\eta)k,4\eta^4k, \text{red}, \text{blue}, \text{green}\right).$
 \end{itemize}
\end{lemma}

We also make use of the following pair of stability results for cycles from~\cite{KoSiSk2}, which, given a sufficiently large two-coloured almost-complete graph allow us to find either a large matching or a particular structure.

\begin{lemma}[\cite{KoSiSk2}]
\label{l:SkB}

For every~$\eta$ such that $0<\eta<10^{-20}$, there exists  $k_{\ref{l:SkB}}=k_{\ref{l:SkB}}(\eta)$ such that, for every $k>k_{\ref{l:SkB}}$ and every $\alpha,\beta>0$ such that $\alpha \geq \beta \geq 100\eta^{1/2}\alpha$, if $K>(\alpha + \half\beta-\eta^{1/2}\beta)k$ and $G=(V,E)$ is a red-blue-multicoloured $\beta \eta^2 k$-almost-complete graph on $K$ vertices, then at least one of the following occurs:
\begin{itemize}
\item[(i)]~$G$ contains a red connected-matching on at least $(1+\eta^{1/2})\alpha k$ vertices;
\item[(ii)]~$G$ contains a blue connected-matching on at least $(1+\eta^{1/2})\beta k$ vertices;
\item[(iii)] the vertices of~$G$ can be partitioned into  $W$, $V'$, $V''$ such that
\begin{itemize}
\item[(a)] $|V'| < (1+\eta^{1/2})\alpha k$, 
$|V''|\leq \half(1+\eta^{1/2})\beta k$,
$|W|\leq \eta^{1/16} k$,
\item[(b)] $G_1[V']$ is $(1-\eta^{1/16})$-complete and $G_2[V']$ is $\eta^{1/16}$-sparse,
\item[(c)] $G_2[V',V'']$ is $(1-\eta^{1/16})$-complete and $G_1[V',V'']$ is $\eta^{1/16}$-sparse;
\end{itemize}
\item[(iv)] we have $\beta > (1-\eta^{1/8})\alpha$ and the vertices of~$G$ can be partitioned into $W$, $V'$ and $V''$ such that
\begin{itemize}
\item[(a)] $|V'| < (1+\eta^{1/2})\beta k$,
$|V''|\leq \half(1+\eta^{1/2})\alpha k$,
$|W|\leq \eta^{1/16} k$,
\item[(b)] $G_2[V']$ is $(1-\eta^{1/16})$-complete and $G_1[V']$ is $\eta^{1/16}$-sparse, 
\item[(c)] $G_1[V',V'']$ is $(1-\eta^{1/16})$-complete and $G_2[V',V'']$ is $\eta^{1/16}$-sparse.
\end{itemize}
\end{itemize}
Furthermore, if $\alpha+ \half\beta \geq 2(1+\eta^{1/2})\beta$, then we can replace (i) with
\begin{itemize}
\item[(i')]~$G$ contains a red odd connected-matching on $(1+\eta^{1/2})\alpha k$ vertices.
\end{itemize}
\end{lemma}

\begin{lemma}[\cite{KoSiSk2}]
\label{l:SkA}

For every $0<\alpha\leq 1$ and every~$\eta$ such that $0<\eta<0.001\alpha$, there exists $k_{\ref{l:SkA}}=k_{\ref{l:SkA}}(\eta)$ such that, for every $k>k_{\ref{l:SkA}}$, if $K>(\tfrac{3}{2}\alpha + 80\eta)k$ and $G=(V,E)$ is a red-blue-multicoloured $\eta^2 k$-almost-complete graph on $K$ vertices, then at least one of the following occurs:
\begin{itemize}
\item[(i)]~$G$ contains a red connected-matching on at least $(1+\eta^{1/2})\alpha k$ vertices;
\item[(ii)]~$G$ contains a blue odd connected-matching on at least $(1+\eta^{1/2})\alpha k$ vertices;
\item[(iii$\ast$)] the vertices of~$G$ can be partitioned into $V'$, $V''$ such that
\begin{itemize}
\item[(a)] $|V'|,|V''| < (\alpha+\eta)k$, 
\item[(b)] all edges present in $G[V']$ and $G[V'']$ are coloured exclusively red and all edges present in $G[V',V'']$ are coloured exclusively blue. \end{itemize}
\end{itemize}
\end{lemma}

Combining the two results above, we obtain:

\begin{lemma}[\cite{KoSiSk2}]
\label{l:SkAB}
For every $0<\alpha,\beta\leq1$ such that $\beta\geq\alpha\geq100\eta^{1/2}\beta$, for every $0<\eta<\min\{10^{-20},0.001\alpha,(\alpha/2)^8\}$, there exists $k_{\ref{l:SkAB}}=k_{\ref{l:SkAB}}(\eta)$ such that, for every $k>k_{\ref{l:SkAB}}$, if $K>(\max\{2\alpha,\half\alpha + \beta\}-\eta^{1/2}\alpha)k$ and $G=(V,E)$ is a red-blue-multcoloured~$\eta^{3}k$-almost-complete graph on $K$ vertices, then at least one of the following occurs:
\begin{itemize}
\item[(i)]~$G$ contains a red connected-matching on at least $(1+\eta^{1/2})\alpha k$ vertices;
\item[(ii)]~$G$ contains a blue odd connected-matching on at least $(1+\eta^{1/2})\beta k$ vertices;
\item[(iii)] we have $\alpha\leq(1-\eta^{1/8})\beta$ and the vertices of~$G$ can be partitioned into  $W$, $V'$, $V''$ such that
\begin{itemize}
\item[(a)] $|V'| < (1+\eta^{1/2})\beta k$, 
$|V''|\leq \half(1+\eta^{1/2})\alpha k$,
$|W|\leq \eta^{1/16} k$,
\item[(b)] $G_2[V']$ is $(1-\eta^{1/16})$-complete and $G_1[V']$ is $\eta^{1/16}$-sparse,
\item[(c)] $G_1[V',V'']$ is $(1-\eta^{1/16})$-complete and $G_2[V',V'']$ is $\eta^{1/16}$-sparse;
\end{itemize}
\item[(iii$\ast$)] we have $\beta < (\tfrac{3}{2}+2\eta^{1/2})\alpha$  and the vertices of~$G$ can be partitioned into $V'$, $V''$ such that
\begin{itemize}
\item[(a)] $|V'|,|V''| < (\alpha+\eta)k$, 
\item[(b)] all edges present in $G[V']$ and $G[V'']$ are coloured exclusively red and all edges present in $G[V',V'']$ are coloured exclusively blue. \end{itemize}
\end{itemize}
\end{lemma}

\begin{proof}
Suppose that $\beta\geq\alpha>(1-\eta^{1/8})\beta$.
Then, since $\eta\leq\min\{(\alpha/2)^8,0.001\alpha\}$, we have $(\max\{2\alpha,\half\alpha+\beta\}-\eta^{1/2}\alpha)k\geq(\tfrac{3}{2}\beta+80\eta)k$. Thus, since $\eta<0.001\alpha\leq0.001\beta$, we may apply Lemma~\ref{l:SkA} (with $\beta$ taking the role of $\alpha$) to find that $G$ either contains a red connected-matching on at least $(1+\eta^{1/2})\beta k\geq(1+\eta^{1/2})\alpha k$ vertices or a blue odd connected-matching on at least $(1+\eta^{1/2})\beta k$ vertices or admits a partition satisfying~(iii$\ast$).

Thus, we may assume that $\alpha\leq(1-\eta^{1/8})\beta$. Applying Lemma~\ref{l:SkB} (with the roles of $\alpha$ and $\beta$ and the colours exchanged) we find that $G$ either contains a red connected-matching on at least $(1+\eta^{1/2})\alpha k$ vertices, a blue connected-matching on at least $(1+\eta^{1/2})\beta k$ vertices or admits a partition $W\cup V'\cup V''$ satisfying (iii). Except in the second case, the proof is complete. 

Thus, we may assume that $G$ contains a blue connected-matching $M$ on at least $(1+\eta^{1/2})\beta k$ vertices which is not odd and that $\beta+\half\alpha<2(1+\eta^{1/2})\alpha$. Partitioning the vertices spanned by $M$ into $V' \cup V''$ such that the edges of $M$ belong to $G[V',V'']$, we then have $|V'|,|V''|\geq\half(1+\eta^{1/2})\beta k$ and, since~$M$ is not odd, know that all edges present in $G[V']$ and $G[V'']$ must be red. Then, since $G$ is $\eta^3 k$-almost-complete, by Corollary \ref{dirac1a}, $G[V']$ and $G[V'']$ each contain a red connected-matching on at least $\half(1+\eta^{1/2})\beta k-1\geq (\half\alpha+\tfrac{1}{4}\eta^{1/2})k$ vertices. Thus, the presence of a red edge in $G[V',V'']$ would imply the existence of a red connected-matching on at least~$\alpha k$ vertices. Thus, we may assume that all edges present in  $G[V',V'']$ are coloured blue. 

Thus, the partition of $V(M)$ obtained, resembles that described in (iii$\ast$). To complete the proof in this case, we attempt to extend this into a partition of $V(G)$: Recall that $G[V',V'']$ contains a blue connected-matching $M$ on at least $\beta k$ vertices but that this connected-matching is not odd. Then, consider a vertex $v\in V\backslash (V'\cup V'')$. Such a vertex cannot have blue edges to both $V'$ and $V''$ since this would allow $M$ to be extended into an odd connected-matching on at least $\beta k$ vertices. Thus, either every edge present in $G[v,V']$ is red or every edge present in $G[v,V'']$ is red. In the former case, every edge present in $G[v,V'']$ must be blue (to avoid having a red connected-matching on at least $\alpha k$ vertices). In the latter case, every edge present in $G[v,V']$ must be blue.

Thus, $v$ can be added to either $V'$ or $V''$ while maintaining the property that all edges present in $G[V']$ and $G[V'']$ are coloured red and all edges in $G[V',V'']$ are coloured blue. Therefore, assigning the vertices of $V\backslash (V'\cup V'')$ in turn to either $V'$ or $V''$, we obtain a partition satisfying (iii$\ast$), completing the proof.
\end{proof}

It is a well-known fact that either a graph is connected or its complement is. We now prove
three
simple extensions of this fact for two-coloured almost-complete graphs, all of which can be immediately extended to two-multicoloured almost-complete graphs. 

\begin{lemma}\label{l:dgf0}
For every $\eta$ such that $0<\eta<1/3$ and every $K\geq 1/\eta$, if $G=(V,E)$ is a two-coloured $(1-\eta)$-complete graph on~$K$ vertices and~$F$ is its largest monochromatic component, then $|F|\geq (1-3\eta)K$.
\end{lemma}

\begin{proof}
If the largest monochromatic (say, red) component in~$G$ has at least $(1-3\eta)K$ vertices, then we are done. Otherwise, we may partition the vertices of~$G$ into sets~$A$ and~$B$ such that $|A|,|B|\geq3\eta K\geq 2$ and there are no red edges between~$A$ and~$B$. Since~$G$ is $(1-\eta)$-complete, any two vertices in~$A$ have a common neighbour in~$B$, and any two vertices in~$B$ have a common neighbour in~$A$. Thus, $A\cup B$ forms a single blue component.
\end{proof}

The following lemmas form analogues of the above, the first concerns the structure of two-coloured almost-complete graphs with one hole and the second concerns the structure of two-coloured almost-complete graphs with two holes, that is, bipartite graphs. 

\begin{lemma}
\label{l:dgf1} 

For every $\eta$ such that $0<\eta<1/20$ and every $K\geq 1/\eta$, the following holds. For~$W$, any subset of~$V$ such that $|W|,|V\backslash W|\geq 4\eta^{1/2}K$, let $G_{W}=(V,E)$ be a two-coloured graph obtained from~$G$, a $(1-\eta)$-complete graph on~$K$ vertices with vertex set~$V$ by removing all edges contained entirely within~$W$. Let~$F$ be the largest monochromatic component of $G_W$ and define the following two sets:
\begin{align*} 
 W_{r}&= \{\text{$w \in W$ : $w$ has red edges to all but at most  $3\eta^{1/2} K$ vertices in $V \backslash W$}\};
\\ 
 W_{b}&=\{ w \in W : w \text{ has blue edges to all but at most } 3\eta^{1/2} K \text{ vertices in } V \backslash W\}.
\end{align*}
Then, at least one of the following holds:
\begin{itemize}
\item [(i)] $|F|\geq (1-2\eta^{1/2})K$; 
\item [(ii)] $|W_{r}|,|W_{b}|>0$.
\end{itemize}
\end{lemma}

\begin{proof}
Consider $G[V\backslash W]$. 
Since~$G$ is $(1-\eta)$-complete, $|V\backslash W|\geq 4\eta^{1/2}K$ and $\eta<1/20$, we see that every vertex in $G[V\backslash W]$ has degree at least $|V\backslash W|-\eta(K-1)\geq(1-\tfrac{1}{4}\eta^{1/2})(|V\backslash W|-1)$, that is, $G[V\backslash W]$ is $(1-\tfrac{1}{4}\eta^{1/2})$-complete. Thus, provided $4\eta^{1/2}K \geq 1/(\tfrac{1}{4}\eta^{1/2})$, that is, provided $K\geq 1/\eta$, we can apply Lemma~\ref{l:dgf0}, which tells us that the largest monochromatic component in $G[V\backslash W]$ contains at least $|V\backslash W|-\eta^{1/2} K$ vertices. We assume, without loss of generality, that this large component is red and call it~$R$.

Now,~$G$ is $(1-\eta)$-complete so either every vertex in~$W$ has a red edge to~$R$ (giving a monochromatic component of the required size) or there is a vertex $w\in W$ with at least $|R|-2\eta K$ {blue neighbours} in~$R$, that is, a vertex $w\in W_{b}$. Denote by~$B$ the set of $u\in R$ such that $uw$ is blue. Then, $|B|\geq |V\backslash W|-2\eta^{1/2} K$ and either every point in~$W$ has a blue edge to~$B$, giving a blue component of size at least $|B\cup W|>(1-2\eta^{1/2})K$, or there is a vertex $w_1\in W_{r}$.
\end{proof}

\section{Proof of the stability result}
\label{s:stabp}

In order to prove Theorem~D, we need to show that any three-multicoloured graph on slightly fewer than $$\left(\max\{4\aI,\aI+2\aII,\aI+2\aIII\}\right)k$$ vertices with sufficiently large minimum degree will contain a red connected-matching on at least~$\alpha_{1}k$ vertices, a blue odd connected-matching on at least~$\alpha_{2}k$ vertices or a green odd connected-matching on at least~$\alpha_{3}k$ vertices, or will have a particular structure.

Thus, given $\aI, \aII, \aIII$, we set 
$$c=\max\{4\aI,\aI+2\aII,\aI+2\aIII\}=\aI+\max\{3\aI,2\aII,2\aIII\},$$ let
$$\eta_{D}(\aI,\aII,\aIII)=\min \left\{10^{-40},\left(\frac{\aI}{50}\right)^{16}, \left(\frac{\aII}{50}\right)^{16}, \left(\frac{\aIII}{50}\right)^{16}, 
\left(\frac{\min\{\aI,\aII,\aIII\}}{100\max\{\aI,\aII,\aIII\}}\right)^4
 \right\},$$
 chose $\eta$ such that 
 $$\eta<\min\left\{\eta_D(\aI,\aII,\aIII),\half\eta_{\ref{l:fourteen}}(\aI,\aII,\aIII),\left(\eta_{\ref{thBc}}(\aI,\aII,\aIII)\right)^2, \eta_{\ref{kss07a-7}}(\alpha_1),\right\}$$
and consider $G=(V,E)$, a $(1-\eta^4)$-complete graph on~$K\geq 100/\eta$ vertices, where
$$(c - \eta)k \leq K \leq (c - \half\eta)k$$ for some integer $k>k_{D}$
, where $k_{D}=k_{D}(\aI,\aII,\aIII,\eta)$ will be defined implicitly during the course of this section, in that, on a finite number of occasions, we will need to bound~$k$ below in order to apply results from Section~\ref{s:pre1}.

Note that, by scaling, we may assume that $\aI,\aII,\aIII\leq1$. Notice, then, that $G$ is~$4\eta^4k$-almost-complete and, thus, for any $X\subset V$, $G[X]$ is also $4\eta^4k$-almost-complete. 

In this section, we seek to prove that~$G$ contains at least one of the following:
\begin{itemize}
\item [(i)]	  a red connected-matching on at least $\aI k$ vertices;
\item [(ii)]  a blue odd connected-matching on at least $\alpha_{2}k$ vertices;
\item [(iii)]  a green odd connected-matching on at least $\alpha_{3}k$ vertices; 
\item [(iv)] subsets of vertices $W$, $X$ and $Y$ such that $X\cup Y\subseteq W$, $X\cap Y=\emptyset$, $|W|\geq(c-\eta^{1/2})k$, every $\gamma$-component of $G[W]$ is odd, $G[X]$ contains a two-coloured spanning subgraph $H$ from $\cH_{1}\cup\cH_{2}$ and $G[Y]$ contains a a two-coloured spanning subgraph $K$ from $\cH_{1}\cup\cH_{2}$, where~{\phantom{nnn}}
\vspace{-2mm}
\begin{align*}
\cH_1&=\cH\left((\aI-2\eta^{1/64})k,(\half\alpha_{*}-2\eta^{1/64})k,4\eta^2 k,\eta^{1/64},\text{red},\gamma\right),\text{ } 
\\ \cH_2&=\cH\left((\alpha_{*}-2\eta^{1/64})k,(\half\aI-2\eta^{1/64})k,4\eta^2 k,\eta^{1/64},\gamma,\text{red}\right),
\end{align*}

\vspace{-3mm}
for $(\alpha_{*},\gamma)\in\{(\aII,\text{blue}),(\aIII,\text{green})\}$;
\item[(v)] 
disjoint subsets of vertices $X$ and $Y$ such that $G[X]$ contains a two-coloured spanning subgraph~$H$ from $\cH_{2}^*\cup\cJ_b$ and $G[Y]$ contains a two-coloured spanning subgraph $K$ from $\cH_{2}^*\cup\cJ_b$, where
\vspace{-2mm}
\begin{align*}
\cH_{2}^*&=\cH\left((\beta-2\eta^{1/32})k,(\half\aI-2\eta^{1/32})k,4\eta^4 k,\eta^{1/32},\gamma,\text{red}\right),\text{ } 
\\ \cJ_b&=\cJ\left((\aI-18\eta^{1/2}), 4\eta^4 k, \text{red}, \gamma\right),
\end{align*}

\vspace{-3mm}
for $\beta=\max\{\aII,\aIII\}$ and $\gamma\in\{\text{blue}, \text{green}\}$;
\item[(vi)]  a subgraph $H$ from $\cL=\cL\left((\half\alpha+\tfrac{1}{4}\eta)k,4\eta^4k, \text{red}, \text{blue}, \text{green}\right).$
\end{itemize}

Observe that, for $\aI\geq\max\{\aII,\aIII\}$, since $\eta_{\ref{kss07a-7}}(\alpha_1)$,  the result follows immediately from Lemma~\ref{kss07a-7}. Thus, in what follows, we may assume that $\max\{\aII,\aIII\}\geq \aI$.

We consider the average degrees of the coloured spanning subgraphs. Notice that, if $d(G_1)\geq\aI k$, then, by Corollary~\ref{l:eg}, $G$ contains a red connected-matching on $\aI k$ vertices. Thus, since the number of missing edges at each vertex can be bounded above, we see that either $d(G_2)>\half(c-\aI-2\eta)k$ or $d(G_3)>\half(c-\aI-2\eta)k$. Without loss of generality, we assume the former and, thus, have
\begin{equation}
\label{ubnew}
e(G_2)>\tfrac{1}{4}(c-\aI-2\eta)(c-\eta)k^2.
\end{equation}

If $G$ contained a blue odd connected-matching on at least $\aII k$ vertices, the proof would be complete, thus we may instead use Lemma~\ref{l:decomp} to decompose the blue graph and, thus, partition the vertices of~$G$ into $W, X$ and $Y$ such that
\begin{itemize}
\item[(i)] $X$ and $Y$ contain only red and green edges; 
\item[(ii)] $W$ has at most $\half \aII k|W|$ blue edges; and 
\item[(iii)] there are no blue edges between $W$ and $X\cup Y$. 
\end{itemize}

\begin{figure}[!h]
\vspace{-2mm}
\centering
\includegraphics[width=64mm, page=3]{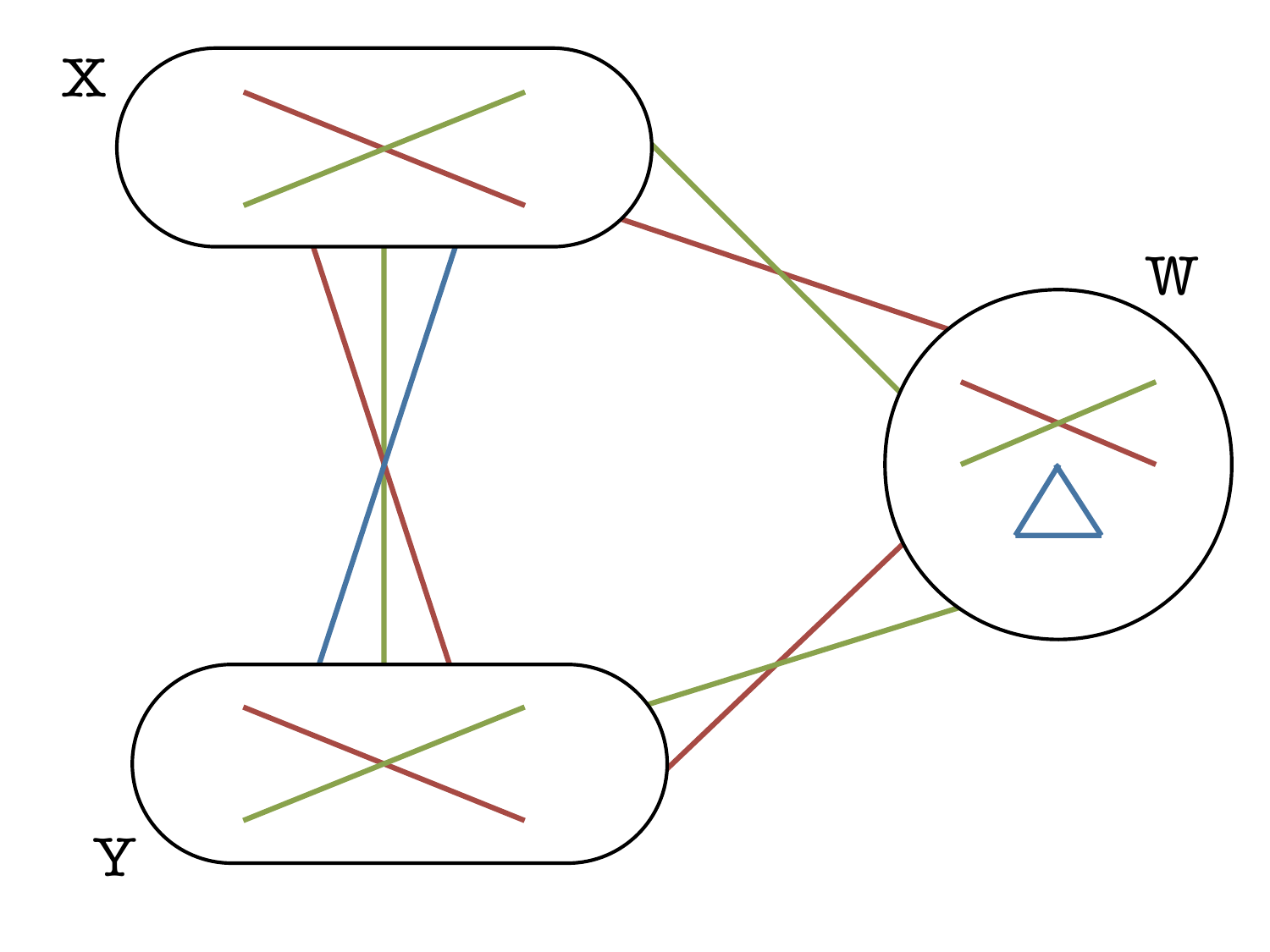}
\vspace{-3mm}\caption{Decomposition of the blue graph.}
  \end{figure}

Thus, writing $wk$ for $|W|$ and noticing that $e(G[X,Y])$ is maximised when~$X$ and~$Y$ are equal in size, we find that
\begin{equation}
\label{lbnew}
e(G_2)\leq\half\aII wk^2+\tfrac{1}{4}(c-w)^2 k^2.
\end{equation}

Comparing (\ref{ubnew}) and (\ref{lbnew}), we obtain a quadratic inequality in $w$, solving which, since $\eta\leq\eta_D$, results in two possibilities:
\begin{itemize}
\item[(F)] $w>c-\eta^{1/2}$;
\item[(G)] $w<\aI+\eta^{1/2}$.
\end{itemize}

 In {\bf Case F}, almost all of the vertices of $G$ belong to $W$. Since $G[W]$ is the union of the odd blue-components of $G$, any blue matching found there is, by definition, odd. Thus, in $G[W]$, any result which provides a blue connected-matching of unspecified parity can be used to provide a blue odd connected-matching. Thus, we consider Theorems~\ref{l:fourteen} and~\ref{thBc} which relate to the even-even-odd case.

Since $G$ is $(1-\eta^4)$-complete and $|W|>(c-\eta^{1/2})k
\geq \tfrac{9}{10}K$, $G[W]$ is $(1-2\eta^4)$-complete.
Thus, since $\eta\leq\half\eta_{\ref{l:fourteen}}(\aI,\aII,\aIII)$, provided $k\geq k_{\ref{l:fourteen}}(\aI,\aII,\aIII,2^{1/4}\eta)$ and $$c-\eta^{1/2}\geq\max\left\{
\aI+2\aII,
2\aI+\aII,
\half\aI+\half\aII+\aIII
\right\}
+14\eta^{1/4},$$ by Theorem~\ref{l:fourteen},~$G$ contains either a red connected-matching on at least $\aI k$ vertices, a blue connected-matching on at least $\aII k$ vertices (which by the nature of the decomposition is odd) or a green odd connected-matching on at least $\aIII k$ vertices.

Thus, in the case that $$\max\{4\aI,\aI+2\aII,\aI+2\aIII\}\geq\max\{2\aI+\aII,\aI+2\aII,\half\aI+\half\aII+\aIII\}+15\eta^{1/4},$$ the proof is complete.
Since $\eta\leq\eta_D$, this condition holds provided $\aIII\geq\aII+24\eta^{1/4}$. Thus, we may instead assume that $\aIII\leq\aII+24\eta^{1/4}$. Also, since $\eta\leq\eta_D$, we have 
$$\aIII\leq 
\aII+24\eta^{1/4} \leq \tfrac{3}{2}\aII \leq \tfrac{3}{2}\max\{\aI,\aII\}
\leq \tfrac{3}{2}\max\{\aI,\aII\}+\half\min\{\aI,\aII\}-12\eta^{1/4}$$
and, since $\eta\leq\left(\eta_{\ref{thBc}}(\aI,\aII,\aIII)\right)^2$, provided $k\geq k_{\ref{thBc}}(\aI,\aII,\aIII,\eta^{1/2})$, we may apply Theorem~\ref{thBc}, to find that $G[W]$ contains either a red connected-matching on at least~$\aI
 k$ vertices, a blue connected-matching on at least $\aII k$ vertices (which by the nature of the decomposition is odd), a green odd connected-matching on at least $\aIII k$ vertices or
two disjoint subsets of vertices $X$ and $Y$ such that $G[X]$ contains a two-coloured spanning subgraph~$H$ from $\cH_1^B\cup\cH_2^B$ and $G[Y]$ contains a two-coloured spanning subgraph $K$ from $\cH_1^B\cup\cH_2^B$ where 
\vspace{-2mm}
\begin{align*}
\cH_1=&\cH\left((\aI-2\eta^{1/64})k,(\half\aII-2\eta^{1/64})k,4\eta^2 k,\eta^{1/64},\text{red},\text{blue}\right),\text{ } 
\\ \cH_2=&\cH\left((\aII-2\eta^{1/64})k,(\half\aI-2\eta^{1/64})k,4\eta^2 k,\eta^{1/64},\text{blue},\text{red}\right).
\end{align*}
Furthermore, $H,K\in \cH_1$ if $\aII\leq\aI-\eta^{1/32}$ and $H,K\in \cH_2$ if $\aI\leq\aII-\eta^{1/32}$, thus completing Case~F.

Moving on to {\bf Case G}, recall that we have a decomposition of the vertices of $G$ into $X\cup Y\cup W$ such that $G[X], G[Y], G[X,W]$ and $G[Y,W]$ contain only red and green edges and that we have $|W|=wk<(c-2\aII+\eta^{1/2})k$.

We assume that $|X|\geq|Y|$ and consider the subgraph $G_1[X\cup W] \cup G_3[X\cup W]$, that is, the subgraph of~$G$ on $X\cup W$ induced by the red and green edges. Recall that $G$, is $(1-\eta^4)$-complete and notice that $|X\cup W|\geq\half K$. Thus, $G[X\cup W]$, is $(1-2\eta^4)$-complete. Therefore, since $\eta\leq\eta_{D}$, provided that $k>k_{\ref{l:thirteen}}(\aI,\aII,w,2^{1/4}\eta)$, by Lemma~\ref{l:thirteen}, if 
\begin{equation*}
|X|+|W| \geq \half\Big(\max\Big\{\aI+\aIII+\max \left\{ 2w, \aI, \aIII \right\},3\aI+\max\left\{2w,\aI\right\} \Big\} + 11\eta^{1/2} \Big)k,
\end{equation*}
then $G[X]\cup G[X, W]$ has a red connected-matching on at least $(\aI+\eta)k$ vertices or a green odd connected-matching on at least $(\alpha_{3}+\eta)k$ vertices.

We may therefore assume that
\begin{equation}
\label{z8}
|X|+|W|< \half\Big(\max\Big\{\aI+\aIII+\max \left\{ 2w, \aI, \aIII \right\},3\aI+\max\left\{2w,\aI\right\} \Big\} + 11\eta^{1/2} \Big)k.
\end{equation}
Also, since $K=|X|+|Y|+|W|$ and $|X|\geq|Y|$, we have
\begin{equation}
\label{z9}
|X|+|W|\geq \frac{K+|W|}{2}=\frac{(c-\eta)k+wk}{2} = 
\half\left(\aI +\max\{3\aI,2\aII,2\aIII\}-\eta+w \right)k.
\end{equation}
We consider four subcases:
\begin{itemize}
\item[(G.i)] $\aI\geq\aIII,2w$; {\bf or} $\tfrac{3}{2}\aI\geq\aIII\geq\aI\geq 2w$;
\item[(G.ii)] $\aIII\geq\tfrac{3}{2}\aI\geq\aI\geq 2w$; {\bf or} $\aIII\geq 2w\geq \aI$ and $\aIII\geq\aI +w$;
\item[(G.iii)] $2w\geq\aIII\geq 2\aI$;
\item[(G.iv)] $\aI+w\geq\aIII\geq 2w\geq \aI$; {\bf or} $2w\geq\aI,\aIII$ and $\aIII\leq 2\aI$.
\end{itemize}

In {\bf Cases G.i} and {\bf G.ii} it can easily be shown that, together, equations (\ref{z8}) and (\ref{z9}) result in a contradiction unless $w\leq \eta+11\eta^{1/2}$ in which case almost all the vertices of~$G$ belong to~$X\cup Y$. In that case, (\ref{z9}) gives $$|X|\geq\half(c-\eta)-\half|W|\geq \left( \max\left\{2\aI,\half\aI+\aII,\half\aI+\aIII\right\}-6\eta^{1/2}\right) k.$$

Recalling (\ref{z8}), in {\bf Case G.i}, we have 
$|X|\leq|X|+|W|<(2\aI+\tfrac{11}{2}\eta^{1/2})k$.

Then, since $Y=K-|X|-|W|$, we have
\begin{align*}
|X|\geq|Y|\geq(c-\eta)k - |X| - |W| &\geq (\max\{2\aI,2\aII-\aI,2\aIII-\aI\}-17\eta^{1/2})k\\&
\geq(\max\{2\aI,\half\aI+\aII,\half\aI+\aII\}-17\eta^{1/2})k.
\end{align*}

Similarly, in {\bf Case G.ii}, we have $|X|\leq|X|+|W|<(\half\aI+\aIII+\tfrac{11}{2}\eta^{1/2})k$, giving
\begin{align*}
|X|\geq|Y|\geq(c-\eta)k - |X| - |W| &\geq (\max\{\tfrac{7}{2}\aI-\aIII,\half\aI+2\aII-\aIII,\half\aI+\aIII\}-17\eta^{1/2})k\\&
\geq(\max\{2\aI,\half\aI+\aII,\half\aI+\aIII\}-17\eta^{1/2})k.
\end{align*}

Thus, provided $\eta<(\aI/17)^4$, letting $\beta=\max\{\aII,\aIII\}$, in each of the {\bf Cases G.i-G.ii}, we have $$|X|\geq|Y|\geq (\max\{2\aI,\half\aI+\beta\}-\aI\eta^{1/4})k.$$

Thus, since $\eta<\min\{10^{-40},(\aI/1000)^4,(\aI/100\beta)^4\}$, provided $k>k_{\ref{l:SkAB}}(\eta^{1/2})$, we may apply Corollary~\ref{l:SkAB} to each of~$G[X]$ and $G[Y]$ to find that $G$ contains either a red connected-matching on at least~$\aI k$ vertices, a green odd connected-matching on at least~$\aIII k$ vertices or two 
disjoint subsets of vertices $X$ and $Y$ such that $G[X]$ contains a two-coloured spanning subgraph~$H$ from $\cH_{2}^*\cup\cJ_b$ and $G[Y]$ contains a two-coloured spanning subgraph $K$ from $\cH_{2}^*\cup\cJ_b$, where
\begin{align*}
\cH_2^*=&\cH\left((\beta-2\eta^{1/32})k,(\half\aI-2\eta^{1/32})k,4\eta^4 k,\eta^{1/32},\text{green},\text{red}\right),\text{ } 
\\ \cJ_b=&\cJ\left((\aI-18\eta^{1/2}), 4\eta^4 k, \text{red}, \text{green}\right).
\end{align*}
Furthermore, $H, K$ may belong to $\cH_2^*$ only if $\aI\leq(1-\eta^{1/16})\beta$ and may belong to $\cJ$ only if $\beta<(\tfrac{3}{2}+2\eta^{1/4})\aI$.

In {\bf Case G.iii}  by (\ref{z8}) and (\ref{z9}) we have 
$$w>\max\{3\aI-\aIII,2\aII-\aIII,\aIII\}-11.5\eta^{1/2}\geq\half(3\aI-\aIII)+\half\aIII-11.5\eta^{1/2}k\geq\tfrac{3}{2}\aI-11.5\eta^{1/2}k,$$
contradicting the assumption that $w<\aI+\eta^{1/2}$.

In case {\bf G.iv}, by (\ref{z8}) and (\ref{z9}), we have $w>\max\{3\aI,2\aII,2\aIII\}-2\aI-11.5\eta^{1/2}$.

Thus, since $\aI\leq\max\{\aII,\aIII\}$ and $w<\aI+\eta^{1/2}$, in what follows we may assume that
\begin{equation}
\label{z10-}
\aI\leq\max\{\aII,\aIII\}\leq\tfrac{3}{2}\aI+6\eta^{1/2}.\end{equation}
Then, recalling (\ref{z8}) and (\ref{z9}), since $K=|X|+|Y|+|W|$ and $|X|\geq|Y|$, it follows that 
\begin{equation}
\label{z10}
\left.
\begin{aligned}
\,\,\,\quad\quad
\left(\half\max\{3\aI,2\aII,2\aIII\}-\eta^{1/2}\right)k
 \leq |X| & <  
 (\tfrac{3}{2}\aI+6\eta^{1/2})k,
 \quad\quad\quad\quad\quad\,\,\,\,\,
 \\
 \left(\half\max\{3\aI,2\aII,2\aIII\}-8\eta^{1/2}\right)k
 \leq |Y| & <   
 (\tfrac{3}{2}\aI+6\eta^{1/2})k,
\\
(\aI-11.5\eta^{1/2})k \leq |W| & <   (\aI+4\eta)k. 
\end{aligned}
\right\}\!
\end{equation}
Thus, $W$ contains around $\aI k$ vertices and each of $X$ and $Y$ contain close to half the remaining
vertices. By scaling, we may assume that $\half \leq \aI, \max\{\aII,\aIII\}\leq 1$. Recall that $G$ is $(1-\eta^4)$-complete and that for any $V'\subseteq V(G)$, $G[V']$ is $4\eta^4 k$-almost-complete.

By~(\ref{z10}), letting $\beta=\max\{\aII,\aIII\}$, recalling that $|X|\geq|Y|$, we have $$|X|,|Y|\geq\left(\max\left\{\tfrac{3}{2}\aI,\tfrac{1}{4}(\tfrac{3}{2}\aI)+\tfrac{3}{4}\beta\right\}-8\eta\right)k\geq\tfrac{3}{4}(\max\left\{2\aI,\half\aI+\beta\right\}-\half(10^4\eta)^{1/2} )k$$
and may prove the following claim:
\begin{claim}
Either $G[X]$ contains a red connected-matching on at least $(\tfrac{3}{4}\aI+100\eta^{1/2})k$ vertices,  $G[X]$ contains a green odd connected-matching on at least $(\tfrac{3}{4}\beta-3\eta^{1/16})k$ vertices or $G[X]\in\cJ(\tfrac{3}{4}(\aI-102\eta^{1/2})k,4\eta^4k,\text{red}, \text{green})$.
\end{claim}
\begin{proof}
Since $\eta\leq10^{-24}$, provided $k\geq \tfrac{4}{3} k_{\ref{l:SkAB}}(10^4\eta)$, applying Lemma~\ref{l:SkAB} (with $\alpha=\alpha_1$ and green taking them place of blue), we find that at least one of the following occurs:
\begin{itemize}
\item[(i)] $G[X]$ contains a red connected-matching on at least $(\tfrac{3}{4}\aI+100\eta^{1/2})k$ vertices;
\item[(ii)] $G[X]$ contains a green odd connected-matching on at least $(\tfrac{3}{4}\beta+100\eta^{1/2})k$ vertices; 
\item[(iii)] $G[X]$ admits a partition of its vertices into  $W$, $V'$, $V''$ such that
\begin{itemize}
\item[(a)] $|V'| < \tfrac{3}{4}(1+100\eta^{1/2})\beta k$, 
$|V''|\leq \tfrac{3}{8}(1+100\eta^{1/2})\aI k$,
$|W|\leq \tfrac{3}{2}\eta^{1/16} k$,
\item[(b)] $G_2[V']$ is $(1-2\eta^{1/16})$-complete and $G_1[V']$ is $2\eta^{1/16}$-sparse,
\item[(c)] $G_1[V',V'']$ is $(1-2\eta^{1/16})$-complete and $G_2[V',V'']$ is $2\eta^{1/16}$-sparse; 
\end{itemize}
\item[(v)] $G[X]$ admits a partition of its vertices into $V'$, $V''$ such that
\begin{itemize}
\item[(a)] $\tfrac{3}{4}(\aI-102\eta^{1/2})k<|V'|,|V''| < \tfrac{3}{4}(\aI+10^4\eta)k$, 
\item[(b)] all edges present in $G[V']$ and $G[V'']$ are coloured red and all edges present in $G[V',V'']$ are coloured blue. 
\end{itemize}
\end{itemize}
Furthermore, (iii) only occurs if $1\geq\max\{\aII,\aIII\}\geq\aI\geq \half$ and (iv) only occurs if $\half\aI+ \beta < 2(1+\eta^{1/2})\aI$. Note that the remaining situation in Lemma~\ref{l:SkAB}, cannot occur since \mbox{$\half\leq \aI, \max\{\aII,\aIII\}\leq1$}. 

In case (iii), since $|X|=|V'|+|V''|+|W|$, we have $|V'|>(\tfrac{3}{4}\beta-2\eta^{1/16})k$. Then, since $G$ is $(10^4\eta)^3(\tfrac{3}{4})k$-almost-complete, by Corollary~\ref{dirac1a} $G[V']$ contains a green cycle of length~$m$ for every $2(10^4\eta)^3(\tfrac{3}{4})k+2\leq m\leq|V'|$. Thus, provided $k\geq \eta^{-1/16}$, $G[V']$ contains a green odd connected-matching on at least $(\tfrac{3}{4}\beta-2\eta^{1/16})k-1\geq(\tfrac{3}{4}\beta-3\eta^{1/16})k$ vertices, completing the proof of the claim.
\end{proof}

The same result applies to $G[Y]$, thus, letting $M_1$ be the largest monochromatic connected-matching in $G[X]$ and $M_2$ the largest monochromatic connected-matching in $G[Y]$, we may distinguish a number of subcases as follows:
 
\begin{itemize}
\item[(G.a)] 
$M_1, M_2$ are both red, 
$|V(M_1)|, |V(M_2)|\geq (\tfrac{3}{4}\aI+100\eta^{1/2})k$,
\\ \hphantom{~} $M_1$ and $M_2$ each share vertices with odd component(s) of the green graph;
\item[(G.b)] 
$M_1, M_2$ are both red, 
$|V(M_1)|, |V(M_2)|\geq (\tfrac{3}{4}\aI+100\eta^{1/2})k$,
\\ \hphantom{~} $M_1$ does not share any vertices with any odd component of the green graph;
\item[(G.c)] 
$M_1, M_2$ are both green (and odd), 
$|V(M_1)|, |V(M_2)|\geq(\tfrac{3}{4}\beta-3\eta^{1/16})k$;
\item[(G.d)] 
$M_1$ is red, $|V(M_1)|\geq (\tfrac{3}{4}\aI+100\eta^{1/2})k$,
$M_2$ is green (and odd), $|V(M_2)|\geq(\tfrac{3}{4}\beta-3\eta^{1/16})k$;
\item[(G.e)] 
$M_1$ is red, $|V(M_1)|\geq (\tfrac{3}{4}\aI+100\eta^{1/2})k$,
 $G[Y]\in\cJ(\tfrac{3}{4}(\aI-102\eta^{1/2})k,4\eta^4k,\text{red}, \text{green})$;
\item[(G.f)] 
$G[X]\in\cJ(\tfrac{3}{4}(\aI-102\eta^{1/2})k,4\eta^4k,\text{red}, \text{green})$, 
$M_2$ is green (and odd), $|V(M_2)|\geq(\tfrac{3}{4}\beta-3\eta^{1/16})k$;
\item[(G.g)] $G[X],G[Y]\in\cJ(\tfrac{3}{4}(\aI-102\eta^{1/2})k,4\eta^4k,\text{red}, \text{green})$. 
\end{itemize}
 
{\bf Case G.a:} There cannot exist a triple of vertices $w\in W, x\in M_{1}$ and $y\in M_{2}$ such that the edges $wx$ and $wy$ are both coloured red, since such an edge would imply the existence of a red connected-matching on at least $\aI k$ vertices. 
Thus, we may partition $W$ into $W_{1} \cup W_{2}$, such that all edges present in $G[V(M_1),W_2]$ and $G[V(M_2),W_1]$ are coloured exclusively green. Thus, in particular, all vertices in $V(M_1)$ belong to a single green component which, by assumption, is odd.

\begin{figure}[!h]
\centering
\vspace{-2mm}
\includegraphics[width=64mm, page=31]{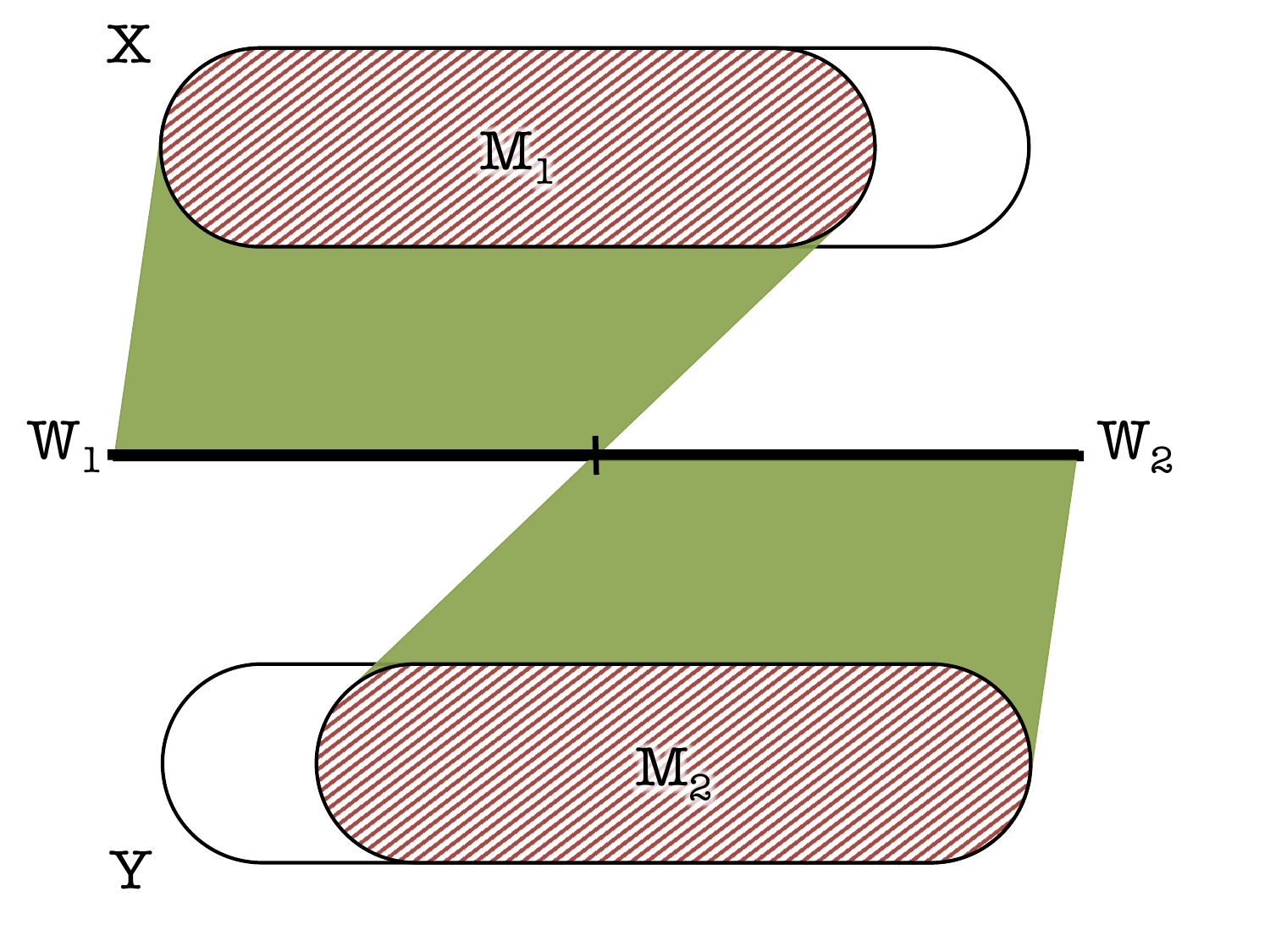}
\vspace{-3mm}\caption{Partition of $W$ into $W_1\cup W_2$.}
\label{redMgreenP}
\end{figure}

Suppose, then, that $|W_1|\geq (\half\aIII +8\eta^4)k$ and recall that $G[W,V(M_1)]$ is $4\eta^4 k$-almost-complete. Since $\eta<\eta_D$, we have $|V(M_1)|\geq (\tfrac{3}{4}\aI+100\eta^{1/2})k\geq(\tfrac{1}{2}\aIII+90\eta^{1/2})k$ and so we may apply Lemma~\ref{l:eleven} to $G[V(M_1),W_1]$ with $\ell=(\half\aIII +8\eta^4) k$ and $a=4\eta^4 k$ to give a green connected-matching on at least $\aIII k$ vertices. This connected-matching is odd since we know that $V(M_1)$ belongs to an odd green component.  

The result is the same in the event that $|W_2|\geq(\half\aIII +8\eta^4)k$. Thus, we may assume that 
$|W_{1}|, |W_{2}|\leq (\half\aIII+8\eta^4) k\leq|V(M_1)|,|V(M_2)|$.
 In that case, we have, by (\ref{z10}),  $|W_{1}|=|W|-|W_{2}|\geq (\aI-11.5\eta^{1/2})k-(\half\aIII+8\eta^4)k\geq(\tfrac{1}{6}\aIII-24\eta^{1/2})k$ and, likewise, $|W_{2}|\geq (\tfrac{1}{6}\aIII-24\eta^{1/2})k$. Recall that $G$ is $4\eta^4k$-almost-complete. Then, since $\eta\leq(\aIII/300)^2$, we have $4\eta^4 k < \half(\tfrac{1}{6}\aIII-24\eta^{1/2})k\leq \half |W_1|$ and so, by Lemma~\ref{l:eleven}, $G[V(M_1),W_1]$ has a green connected-matching on at least $2|W_1|-8\eta^4k$ vertices. Similarly, $G[V(M_2),W_2]$ has a green connected-matching on at least $2|W_2|-8\eta^4k$ vertices. By (\ref{z10-}) and (\ref{z10}), we have $2|W_1|+2|W_2|-16\eta^4 k = 2|W|-16\eta^4 k\geq 2(\aI-11.5\eta^{1/2})k-16\eta^4 k\geq\aIII k$. Thus, since these connected-matchings are odd, they must belong to different components of the green graph. Therefore, we may assume that all edges present in $G[V(M_1),W_2]$ and $G[V(M_2),W_1]$ are coloured red. 
 
\begin{figure}[!h]
\centering
\vspace{2mm}
\includegraphics[width=64mm, page=32]{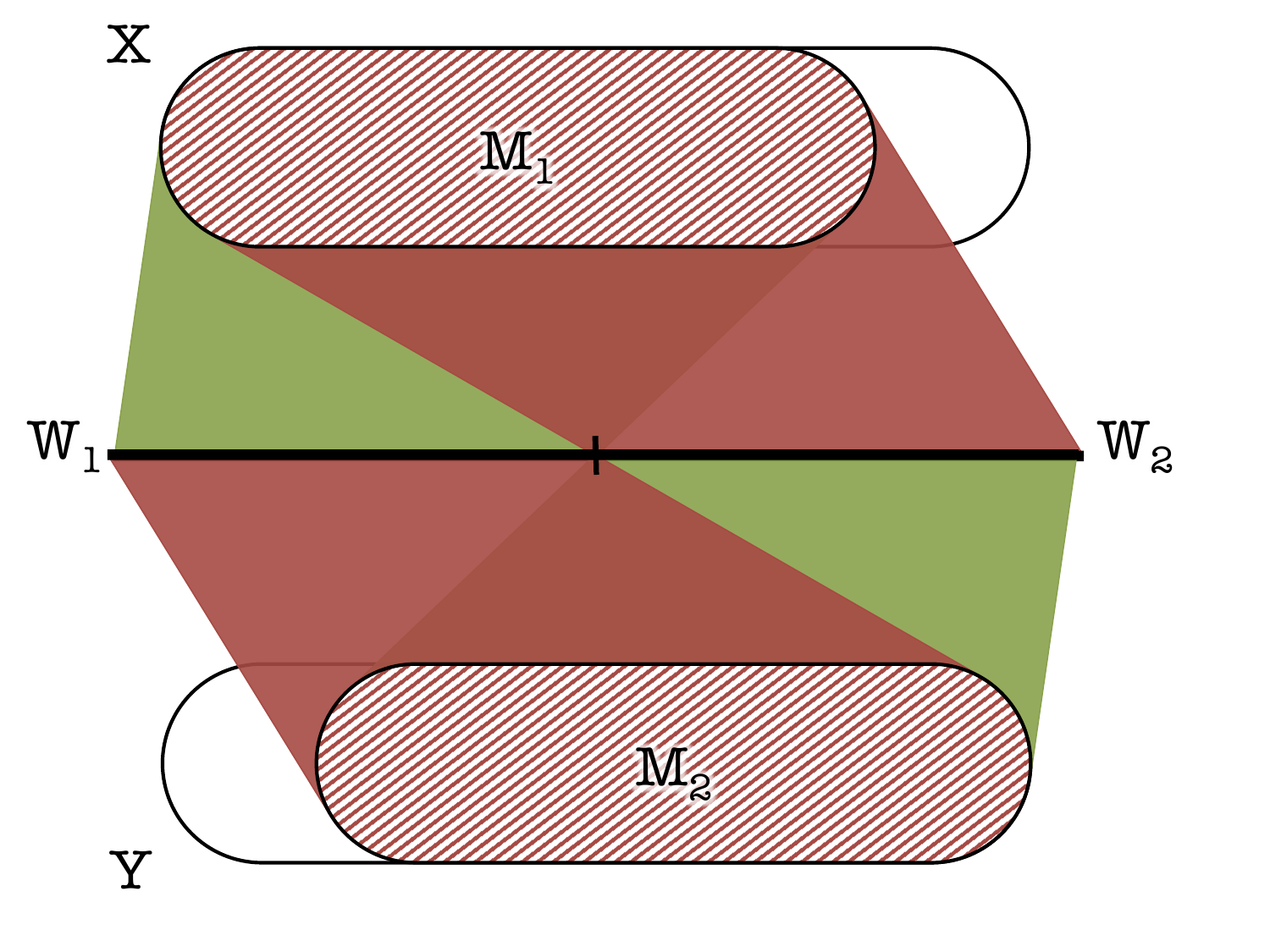}
\vspace{-3mm}\caption{Colouring of the edges of $G[M_1,W_2]\cup G[M_2,W_1]$.}
\label{2ndRedGreen}
  \end{figure}

Observe, then, that, without loss of generality, $|W_1|\geq|W_2|\geq\half(|W_1|+|W_2|)\geq(\half\aI-6\eta^{1/2})k$. Now, choose any set~$R_1$ of $8\eta^{1/2} k$ of the edges from the matching~$M_1$, let $M_1'=M\backslash R_1$ and consider $G[V(M_1'),W_2]$. We have $|V(M_1')|\geq (\tfrac{3}{4}\aI+84\eta^{1/2})k\geq(\half\aI  - 6\eta^{1/2})k$ and, thus, may apply Lemma~\ref{l:eleven} to $G[V(M_1'),W_2]$ to obtain a collection $R_2$ of edges from $G[V(M_1'),W_2]$ which form a red connected-matching on at least $(\aI-14\eta^{1/2})k$ vertices. Since $R_1$ and $R_2$ do not share any vertices but do belong to the same red-component of~$G$, the collection of edges~$R_1\cup R_2$ forms a red connected-matching on at least $\aI k$ vertices, completing this case.

{\bf Case G.b:} Again, there can be no triple of vertices $w\in W, x\in M_{1}$ and $y\in M_{2}$ such that $wx$ and $wy$ are both coloured red. Thus, we may partition $W$ into $W_{1} \cup W_{2}$, where $W_1,W_2$ are defined as in Case G.a, giving the situation illustrated in Figure~\ref{redMgreenP}, specifically all edges present in $G[V(M_1), W_1]$ are coloured exclusively green.

Thus, in particular, every vertex in $V(M_1)$ belongs to the same green component, which is assumed not to be odd. Therefore, no vertex in $P=X\backslash V(M_1)$ can have more than one green edge to $V(M_1)$ and instead each such vertex must have red edges to all but at most one of the vertices of $V(M_1)$. By maximality of $M_1$, there can be no red edges in $G[P]$, that is, all edges present in $G[P]$ are coloured exclusively green. Then, since $|P|\geq(\half\aI-10\eta^{1/2})k$ and~$G$ is $4\eta^4 k$-almost-complete, $G[P]$ contains a triangle and has a single green component. Thus, all edges present in $G[V(M_1),P]$ must, in fact,  be coloured red so as to avoid having $V(M_1)$ belonging to an odd green component.

Given this colouring, it can easily be shown that $G[X]$ contains a red connected-matching on at least $\aI k$ vertices. Indeed, firstly, choose any set~$R_1$ of $6\eta^{1/2} k$ of the edges from the matching~$M_1$, let $M_1'=M\backslash R_1$ and consider $G[V(M_1'),P]$. Since $|V(M_1')|,|P|\geq(\half\aI-\eta^{1/2})k$, by Lemma~\ref{l:eleven}, there exists, in $G[V(M_1'),P]$, a collection $R_2$ of edges  which form a red connected-matching on at least $(\aI-4\eta^{1/2})k$ vertices. Since $R_1$ and $R_2$ do not share any vertices but do belong to the same red-component of~$G$, the collection of edges $R_1\cup R_2$ forms a red connected-matching on at least $\aI k$ vertices, completing this case.

{\bf Case G.c:} Suppose that there exists a triple $w\in W, x\in M_{1}$ and $y\in M_{2}$ such that $wx$ and $wy$ are both green. Such a triple would give a green odd connected-matching on at least $\aIII k$ vertices. 
 
Thus, we may partition $W$ into $W_{1} \cup W_{2}$, such that all edges present in $G[V(M_1),W_2]$ and $G[V(M_2),W_1]$ are coloured exclusively red. 
Suppose, then, that $|W_1|\geq (\half\aI +8\eta^4)k$. Recalling that $G[V(M_2),W_1]$ is $4\eta^4 k$-almost-complete, since $\eta<\eta_D$, we have $|V(M_2)|\geq  (\half\aI +8\eta^4)k$ and we may apply Lemma~\ref{l:eleven} with $\ell=(\half\aI +8\eta^4)k$ and $a=4\eta^4 k$ to give a red connected-matching on at least $\aI k$ vertices. The result is the same in the event that $|W_2|\geq(\half\aI +6\eta^4)k$ with the matching being found in $G[V(M_1), W_2]$.

Therefore, we may assume that $|W_{1}|, |W_{2}|\leq (\half\aI+8\eta^4) k$. In that case, we have $|W_{1}|=|W|-|W_{2}|\geq (\half\aI-12\eta^{1/2})k$ and, likewise, $|W_{2}|\geq (\half\aI-12\eta^{1/2})k$. Thus, since $\eta<(\aI/100)^2$, Lemma~\ref{l:eleven} gives a red connected-matching on at least $({\aI}-26\eta^{1/2})k$ vertices in each of $G[V(M_1),W_1]$ and $G[V(M_2),W_2]$.

Then, since $\eta<(\aI/100)^2$, $V(M_1)\cup W_1$ and $V(M_2)\cup W_2$ must belong to different red components (so as to avoid having a red connected-matching on at least $(2\aI-52\eta^{1/2})k\geq \aI k$ vertices). Thus, all edges present in $G[V(M_{1}),W_{2}]$ and $G[V(M_{2}),W_{1}]$ must be coloured exclusively green. 

\begin{figure}[!h]
\centering
\vspace{2mm}
\includegraphics[width=64mm, page=30]{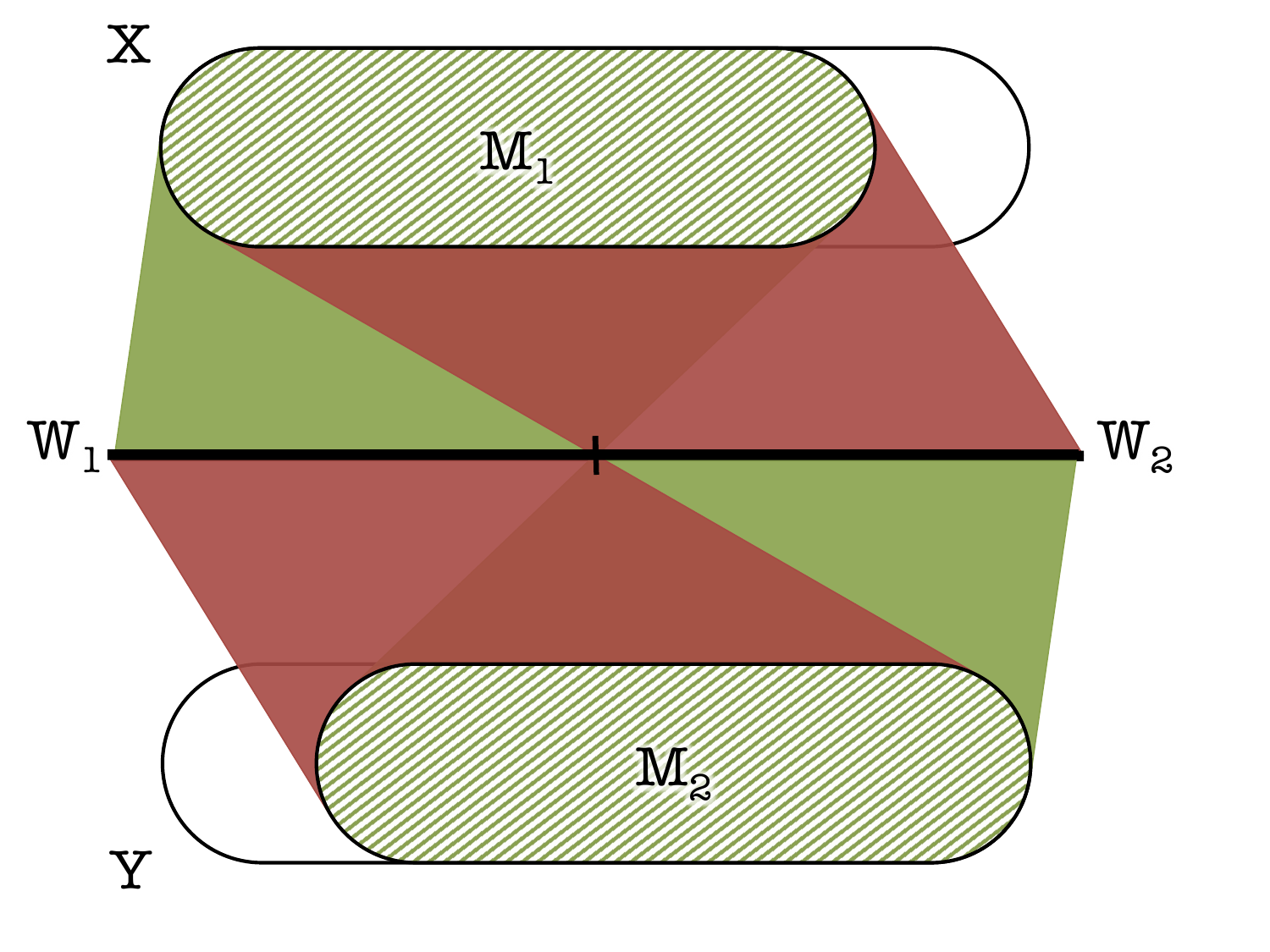}
\vspace{-3mm}\caption{Colouring of the edges of $G[M_1,W_2]\cup G[M_2,W_1]$.}
  
\end{figure}

Given this colouring, it can easily be shown that $G$ contains a green odd connected-matching on at least $\aIII k$ vertices. Indeed, provided $k\geq \eta^{-1/2}$, we can choose a set $E_1$ of edges from $M_1$ such that $(\tfrac{1}{6}\aIII+20\eta^{1/2})k\leq |E_1|\leq (\tfrac{1}{6}\aIII+22\eta^{1/2})k$. Then, the edges of $E_1$ form a matching on at least $(\tfrac{1}{3}\aIII+40\eta^{1/2})k$ vertices and letting $M_1'=M_1\backslash E_1$, since $\eta<(\aIII/50)^{16}$, we have 
$$
|V(M_1')|\geq  (\tfrac{3}{4}\aIII-3\eta^{1/16})-(\tfrac{1}{3}\aIII+44\eta^{1/2})k\geq  (\tfrac{5}{12}\aIII-4\eta^{1/16})k\geq(\tfrac{1}{3}\aIII-18\eta^{1/2})k.
$$
Then, recalling that $|W_1|\geq(\half\aI-12\eta^{1/2})k\geq(\tfrac{1}{3}\aIII-18\eta^{1/2})k$, by Lemma~\ref{l:eleven}, $G[V(M_1'),W_1]$ contains a collection of edges $E_2$, which form a connected-matching on at least $(\tfrac{2}{3}\aIII-38\eta^{1/2})k$ vertices. Then, since $E_1$ and $E_2$ do not share any vertices but do belong to the same odd green-component of~$G$, the collection of edges $E_1\cup E_2$ forms a green odd connected-matching on at least $\aIII k$ vertices, completing this case.

{\bf Case G.d:}

By (\ref{z10-}) and (\ref{z10}), we have $|X|+|W|, |Y|+|W|\geq \half K$, so $G[X\cup W]$ and $G[Y\cup W]$ are each $(1-2\eta^4)$-complete.
Additionally, $|W|,|V\backslash W|\geq 4(2\eta^{4})^{1/2}|X\cup W|$. Thus, provided that $\half K\geq 1/2\eta^4$, we may apply Lemma~\ref{l:dgf1} separately to $G[X\cup W]$ and $G[Y\cup W]$ with the result being that at least one of the following occurs:
\begin{itemize}
\item[(i)] $X\cup W$ has a connected red 
component~$F$ on at least $|X \cup W| - \eta k$ vertices; 
\item[(ii)] $X\cup W$ has a connected green
 component~$F$ on at least $|X \cup W| - \eta k$ vertices; 
\item[(iii)] $Y\cup W$ has a connected red
 component~$F$ on at least $|Y \cup W| - \eta k$ vertices;
\item[(iv)] $Y\cup W$ has a connected green
component~$F$ on at least $|Y \cup W| - \eta k$ vertices;
\item[(v)] there exist points $w_{1}, w_{2}, w_{3}, w_{4} \in W$ such that the following hold:
\begin{itemize}
\item[(a)] $w_1$ has red edges to all but at most $\eta k$ vertices in $X$,
\item[(b)] $w_2$ has green edges to all but at most $\eta k$ vertices in $X$,
\item[(c)] $w_3$ has red edges to all but at most $\eta k$ vertices in $Y$,
\item[(d)] $w_4$ has green edges to all but at most $\eta k$ vertices in $Y$.
\end{itemize}
\end{itemize}

In subcase (i), we discard from~$W$ the at most~$\eta k$ vertices not contained in~$F$ and consider~$G[W,Y]$. Either there are at least~$\tfrac{1}{8}\aI k$ mutually independent red edges present in~$G[W,Y]$ (which can be used to augment $M_1$) or we may obtain $W'\subset W$, $Y' \subset Y$ with $|W'|,|Y'| \geq (\tfrac{7}{8} \aI  - 12\eta^{1/2})k\geq(\tfrac{7}{12}\aIII-24\eta^{1/2})k$ such that all the edges present in $G[W',Y']$ are coloured exclusively green. Since~$G$ is $4\eta^4 k$-almost-complete, so is $G_3[W',Y']$ and we may apply Lemma~\ref{l:eleven} to obtain a green connected-matching on at least $\aIII k$ vertices in $G[W',Y]$ which is odd by virtue of sharing vertices with $M_2$.

In subcase (ii), suppose there exists a green edge in $G[M_2,F]$. Then, at least $|M_2\cup W|-\eta k$ of the vertices of $M_2\cup W$ would belong to the same green component in~$G$.  Discard the at most $\eta k$ vertices of $W$ not contained in that component and consider $G[X,W]$. Either there are at least $(\tfrac{1}{8}\aIII+2\eta^{1/16})k$ mutually independent green edges in $G[X,W]$ which can be used to augment $M_2$ or we may obtain $W'\subset W$, $X' \subset X$ with $|W'|,|X'| \geq (\tfrac{7}{8} \aI  - 3\eta^{1/16})k$ such that all the edges present in $G[W',X']$ are coloured exclusively red. Then, since~$G_1[W',X']$ is $4\eta^4 k$-almost-complete, we may apply Lemma~\ref{l:eleven} to obtain a red connected-matching on at least $\aI k$ vertices in $G[W',X']$.

Thus, we may instead, after discarding at most $\eta k$ vertices from $W$, assume that all edges present in $G[V(M_2),W]$ are coloured exclusively red and apply Lemma~\ref{l:eleven} to obtain a red connected-matching on at least~$\aI k$ vertices in $G[V(M_2),W]$.

In subcase (iii), if there exists a red edge in $G[M_1,W]$, then the same argument as given in case (i) gives a red connected-matching on at least $\aI k$ vertices. Thus, we may assume that, after deleting at most $\eta k$ vertices from $W$, all edges present in $G[M_1,W]$ are coloured green.

Since $\eta<\eta_D$, we have $|V(M_1)|,|W|\geq (\half\aIII+90\eta^{1/2})k$.
Thus, by Lemma~\ref{l:eleven}, there exists a green connected-matching on at least $\aIII k$ vertices in $G[V(M_1),W]$. However, this connected-matching is not necessarily odd. The existence of a green edge in $G[V(M_2),W]$ would suffice to complete the proof. Thus, we may assume that all edges present in $G[V(M_2),W]$ are be coloured exclusively red. But, then, we may apply Lemma~\ref{l:eleven} to obtain a red connected-matching on at least~$\aI k$ vertices in $G[V(M_2),W]$.

In subcase (iv), we can then discard the at most $\eta k$ vertices from of $W$ not contained in $F$ and and consider $G[X,W]$. Either there are at least $(\tfrac{1}{8}\aIII+2\eta^{1/16})k$ mutually independent green edges in $G[X,W]$ which can be used to augment $M_2$ or we may obtain $W'\subset W$, $X' \subset X$ with $|W'|,|X'| \geq (\tfrac{7}{8} \aI  - 3\eta^{1/16})k$ such that all the edges present in $G[W',X']$ are coloured exclusively red. In the latter case, we  apply Lemma~\ref{l:eleven} to obtain a red connected-matching on at least $\aI k$ vertices in $G[W',X']$.

In subcase (v), there exist points $w_{1}, w_{2}, w_{3}, w_{4} \in W$ such that $w_1$ has red edges to all but at most $\eta k$ vertices in $X$, $w_2$ has green edges to all but at most $\eta k$ vertices in $X$, $w_3$ has red edges to all but at most $\eta k$ vertices in $Y$, and $w_4$ has green edges to all but at most $\eta k$ vertices in $Y$.
\begin{figure}[!h]
\centering
\vspace{2mm}
\includegraphics[width=68mm, page=33]{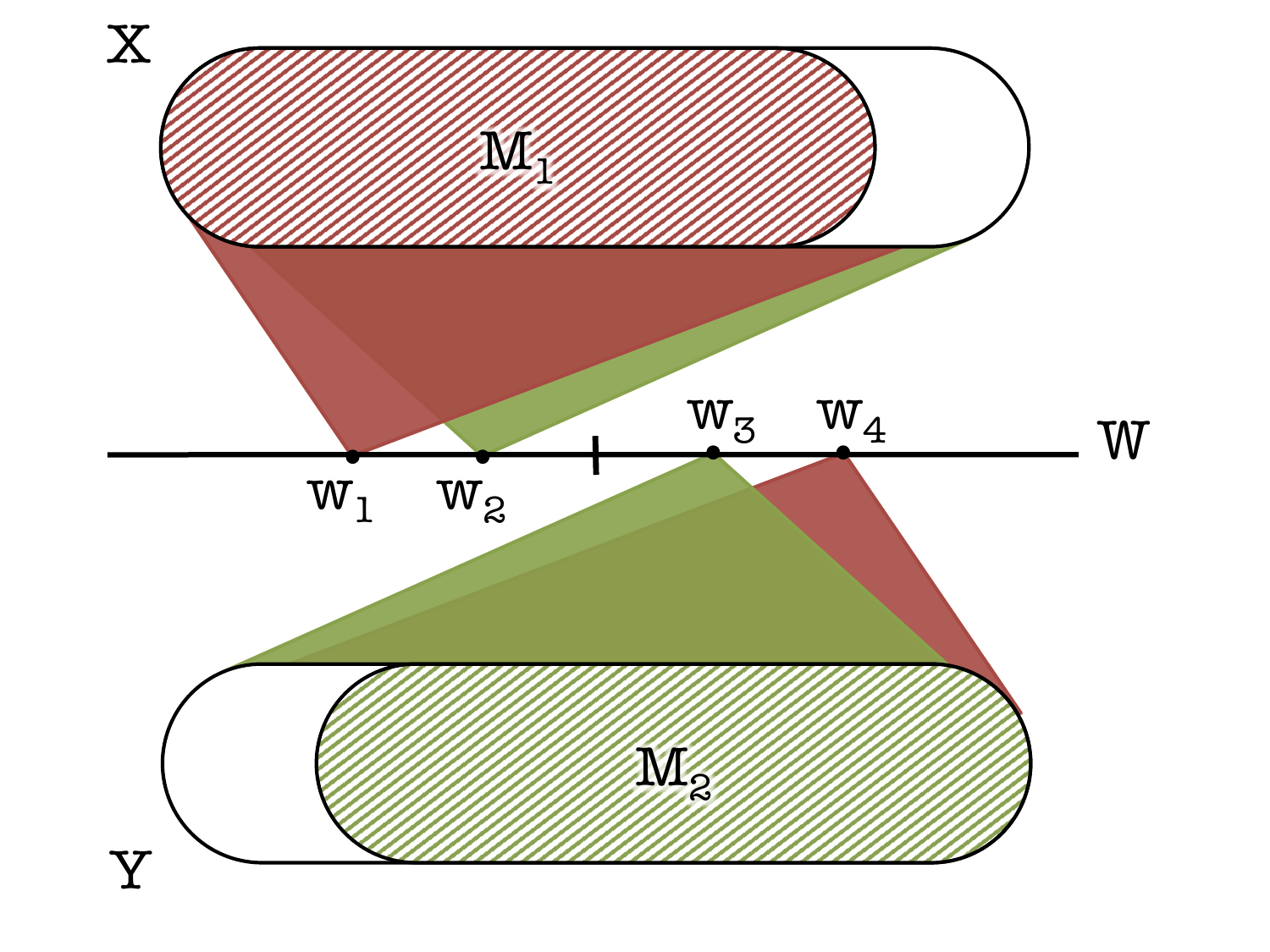}
\vspace{-3mm}\caption{Vertices $w_1,w_2,w_3$ and $w_4$ in case (e).}
\vspace{-3mm}
  
\end{figure}

Thus, defining
\begin{align*}
X_{S}&=\{x\in X \text{ such that } xw_1 \text{ is red and } xw_2 \text{ is green}\},\\
Y_{S}&=\{y\in Y \text{ such that } yw_3 \text{ is red and } yw_4 \text{ is green}\},
\end{align*}
by (\ref{z10}), we have $|X_S|,|Y_S|\geq(\tfrac{3}{2}\aI-10\eta^{1/2})k$. Suppose there exists $w\in W$, $x\in X_S$,~$y\in Y_S$ such that~$wx$ and~$wy$ are red. In that case, $X_S\cup Y_S$ belong to the same red component of~$G$. Recall that $M_1\subseteq G[X]$ is a red connected-matching on $(\tfrac{3}{4}\aI+100\eta^{1/2})k$ vertices and consider $G[W,Y_S]$. Either we can find $\tfrac{1}{8} \aI k$ mutually independent red edges in $G[W,Y_S]$ (which together with $M_1$ give a red connected-matching on at least $\aI k$ vertices) or we may obtain $W'\subset W$, $Y' \subset Y_S$ with $|W'|,|Y'| \geq (\tfrac{7}{8}\aI  - 12\eta^{1/2})k$ such that all the edges present in $G[W',Y']$ are coloured exclusively green. Then, as in case (i), we may apply Lemma~\ref{l:eleven} to obtain a green odd connected-matching on at least $\aIII k$ vertices.

Thus, we may assume that no such triple exists and, similarly, we may assume there exists no triple $w\in W$, $x\in X_S$, $y\in Y_S$ such that $wx$ and $wy$ are both green. Therefore, we may partition~$W$ into $W_1 \cup W_2$ such that all edges present in $G[W_1,X_S]$ and $G[W_2,Y_S]$ are coloured exclusively red and all edges present in $G[W_1,Y_S]$ and $G[W_2,X_S]$ are coloured exclusively green. Thus, we may assume that $|W_1|,|W_2|\leq(\half\aI+\eta^{1/2})k$ (else Lemma~\ref{l:eleven} could be used to give a red connected-matching on at least $\aI k$ vertices) and therefore also that $|W_1|,|W_2|\geq(\half\aI-12\eta^{1/2})k$, in which case we can easily show that there exists a red connected-matching on at least $\aI k$ vertices in $G[X_S\cup W_1]$ as follows: Choose any set~$R_1$ of $14\eta^{1/2} k$ of the edges from the matching~$M_1$, let $X'=X\backslash V(R_1)$ and consider $G[X',W_1]$. We have $|X'|,|W_1|\geq (\tfrac{1}{2}\aI-12\eta^{1/2})k$ and, thus, may apply Lemma~\ref{l:eleven} to $G[X',W_1]$ to obtain a collection $R_2$ of edges from $G[V(M_1'),W_2]$ which form a red connected-matching on at least $(\aI-26\eta^{1/2})k$ vertices. Since $R_1$ and $R_2$ do not share any vertices but do belong to the same red-component of~$G$, the collection of edges $R_1\cup R_2$ forms a red connected-matching on at least $\aI k$ vertices, completing this case.

{\bf Case G.e:} Since $|Y_1|,|Y_2|\geq\tfrac{3}{4}(\aI-102\eta^{1/2})k$ and $G$ is $4\eta^4 k$-almost-complete, by Corollary~\ref{dirac1a}, $G[Y_1]$ contains a red connected-matching $R_1$ on at least $(\tfrac{3}{4}\aI-78\eta^{1/2})k$ vertices. Similarly, $G[Y_2]$ contains a red connected-matching $R_2$ on at least $(\tfrac{3}{4}\aI-78\eta^{1/2})k$ vertices. 

Thus, the existence of a triple $w\in W$, $y_1\in Y_1$, $y_2\in Y_2$ such that both the edges $wy_1$ and $wy_2$ are coloured red would give a red connected-matching on at least $\aI k$ vertices. Therefore, there exists a partition of $W$ into $W_1\cup W_2$ such that all edges present in $G[W_1,Y_1]$ and $G[W_2,Y_2]$ are coloured exclusively green.
\begin{figure}[!h]
\centering
\vspace{2mm}
\includegraphics[width=68mm, page=36]{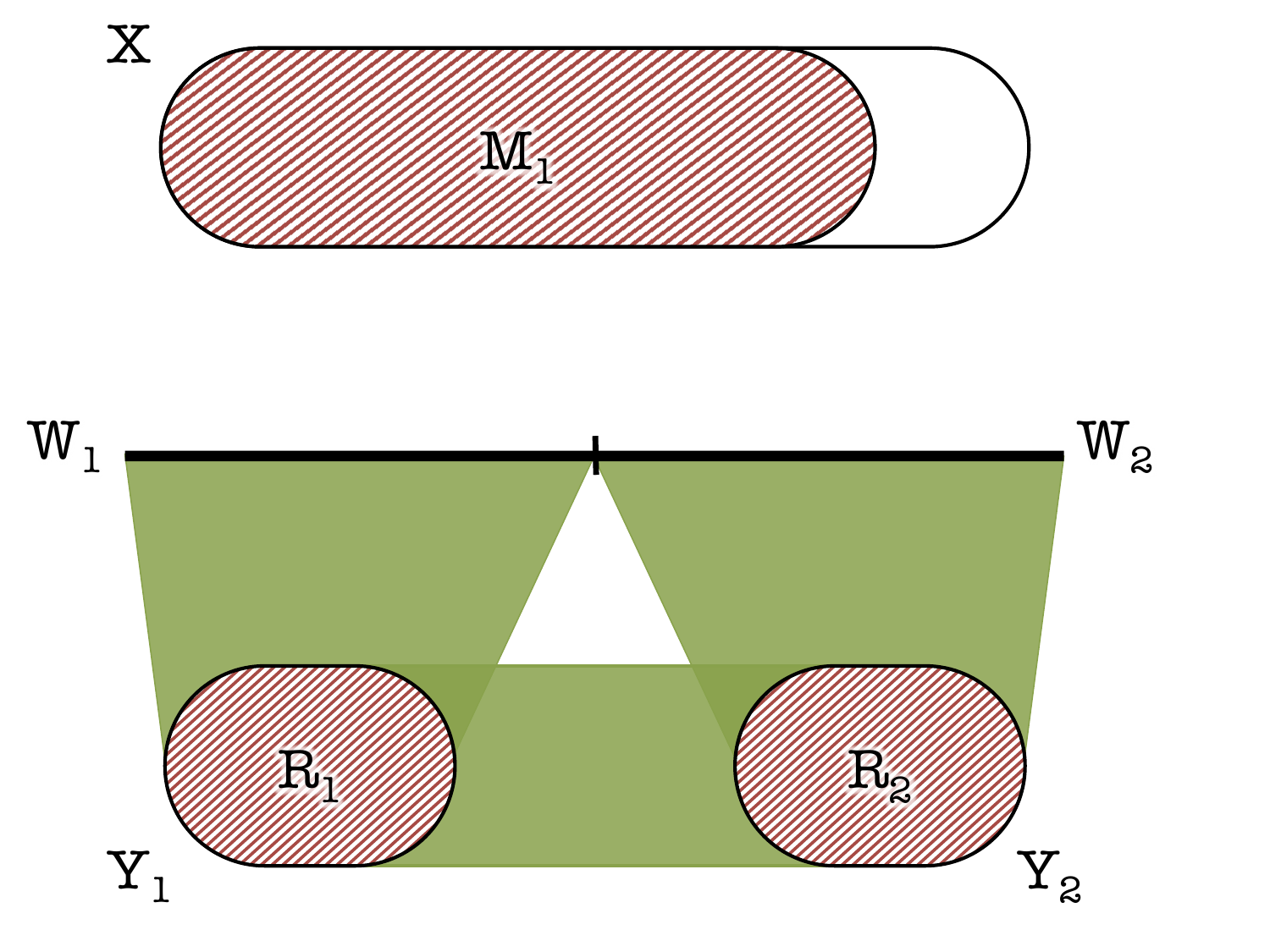}
\vspace{-3mm}\caption{Colouring of $G$ in Case G.e.}
\end{figure}

Recall from (\ref{z10-}) that $\aI>\tfrac{2}{3}\aIII-4\eta^{1/2}$. Thus, $|Y_1|,|Y_2|\geq\tfrac{3}{4}(\aI-102\eta^{1/2})k\geq(\half\aIII-80\eta^{1/2})k$. So, by Lemma~\ref{l:eleven}, $G[Y_1,Y_2]$ contains a green connected-matching $M_2$ on at least $(\aIII-164\eta^{1/2})k$ vertices which need not be odd.

We then consider two possibilities:
\begin{itemize}
\item[(i)] $|W_1|,|W_2|\geq 84\eta^{1/2}k$;
\item[(ii)] $\min\{|W_1|,|W_2|\}\leq 84\eta^{1/2}k$.
 \end{itemize} 

In subcase (i), provided $k\geq \eta^{-1/2}$, we may choose subsets $Y_1'\subset Y_1$, $Y_2'\subset Y_2$  such that $84\eta^{1/2}k\leq|Y_1|,|Y_2|\leq 86\eta^{1/2}k$. Then, since $|W_1|,|W_2|\geq84\eta^{1/2}k$, by Lemma~\ref{l:eleven}, $G[W_1,Y_1]$ contains a red connected-matching $E_1$ on at least $166\eta^{1/2}k$ vertices and  $G[W_2,Y_2]$ contains a red connected-matching $E_2$ on at least $166\eta^{1/2}k$ vertices.
Also, we have $|Y_1\backslash Y_1'|, |Y_2\backslash Y_2'|\geq (\half\aIII-166\eta^{1/2})k$ so, by Lemma~\ref{l:eleven}, $G[Y_1\backslash Y_1', Y_2\backslash Y_2']$ contains a green connected-matching $E_3$ on at least $(\aIII-330\eta^{1/2})k$ vertices. Together $E_1, E_2$ and $E_3$ form a green connected-matching on at least $\aIII k$ vertices although this connected-matching need not be odd.

The existence of green edge in $G[W_1,Y_2]$ or $G[W_2,Y_1]$ would give an odd green cycle in the same green component as $E_1\cup E_2\cup E_3$. Thus, we may assume that all edges present in $G[W_1,Y_2]$ and $G[W_2,Y_1]$ are coloured red.

\begin{figure}[!h]
\centering
\vspace{2mm}
\includegraphics[width=68mm, page=37]{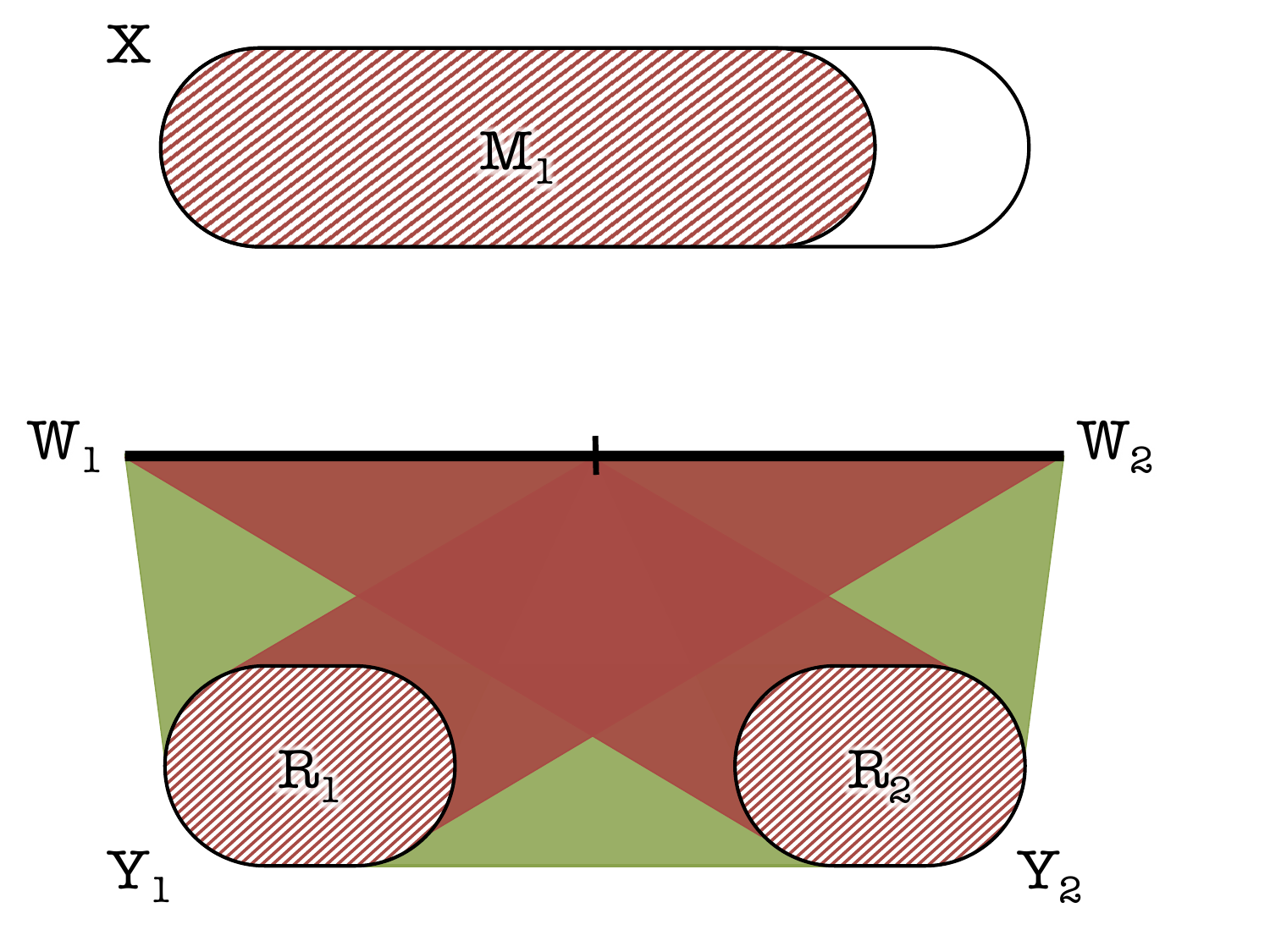}
\vspace{-3mm}\caption{Colouring of $G[W,Y]$ in Case G.e.i.}
\end{figure}

In that case, the existence of a red edge in $G[V(M_1), W_1]$ would give a red connected-matching on $M_1\cup R_1$ on at least $\aI k$ vertices. Similarly, the existence of a red edge in $G[V(M_1), W_2]$ would give a red connected-matching  $M_1\cup R_2$ on at least $\aI k$ vertices. Thus, we may assume that all edges present in $G[V(M_1),W]$ are coloured green, thus giving a green five-cycle in the same component as $E_1\cup E_2 \cup E_3$, completing this subcase.

In subcase (ii), without loss of generality, $|W_2|\leq 84\eta^{1/2}k$. Thus, after discarding at most $84\eta^{1/2}k$ vertices from $W$, we may assume that every vertex in $W$ belongs to the same green component as $M_2$ and that $|W|\geq (\aI-98\eta^{1/2})k$. 

Suppose there is a green edge in $G[W,Y_2]$. Then, the existence of $84\eta^{1/2}k$ mutually independent green edges in $G[V(M_1),W]$ would give an odd green connected-matching on at at least $\aIII k$ vertices. Thus, for now assume there exist subsets $X'\subseteq M_1$ and $W'\subseteq W$ such that $|X'|\geq(\tfrac{3}{4}\aI+2\eta^{1/2})k$ vertices, $|W'|\geq(\aI-182\eta^{1/2})k$ and all edges present in $G[X',W']$ are coloured red. Then, by Lemma~\ref{l:eleven}, $G[X',W']$ would contain a red connected-matching on at least $\aI k$ vertices. 

Therefore, we may instead assume that all edges present in $G[W,Y_2]$ are coloured red. But then, since $\eta\leq(\aI/320)^2$, $|W|\geq(\aI-98\eta^{1/2})k\geq(\half\aI+\eta^{1/2})k$ and $|Y_2|\geq\tfrac{3}{4}(\aI-102\eta^{1/2})k\geq(\half\aI+\eta^{1/2})k$. So, by Lemma~\ref{l:eleven}, $G[W,Y_2]$ contains a red connected-matching on at least $\aI k$ vertices, completing this case.

{\bf Case G.f:} Since $|X_1|,|X_2|\geq\tfrac{3}{4}(\aI-102\eta^{1/2})k$ and $G$ is $4\eta^4 k$-almost-complete, by Corollary~\ref{dirac1a}, $G[X_1]$ contains a red connected-matching $R_1$ on at least $(\tfrac{3}{4}\aI-78\eta^{1/2})k$ vertices. Similarly, $G[X_2]$ contains a red connected-matching $R_2$ on at least $(\tfrac{3}{4}\aI-78\eta^{1/2})k$ vertices. 

Thus, the existence of a triple $w\in W$, $x_1\in X_1$, $x_2\in X_2$ such that both the edges $wx_1$ and $wx_2$ are coloured red would give a red connected-matching on at least $\aI k$ vertices. Therefore, there exists a partition of $W$ into $W_1\cup W_2$ such that all edges present in $G[W_1,X_1]$ and $G[W_2,X_2]$ are coloured exclusively green.

Recall from (\ref{z10-}) that $\aI>\tfrac{2}{3}\aIII-4\eta^{1/2}$. Thus, $|X_1|,|X_2|\geq(\half\aIII-80\eta^{1/2})k$, so, by Lemma~\ref{l:eleven}, $G[X_1,X_2]$ contains a green connected-matching $M_1$ on at least $(\aIII-164\eta^{1/2})k$ vertices.

Since $\eta\leq\eta_D$, $(\aIII-164\eta^{1/2})k+(\tfrac{3}{4}\aI-3\eta^{1/16})k\geq \aIII k$. Thus, in order to avoid having a green odd connected-matching on at least $\aIII k$ vertices, $M_1$ and $M_2$ must be in different green components. Therefore, we may assume that all edges present in $G[W,V(M_2)]$ are coloured red. Then, since $|W|,|V(M_2)|\geq(\tfrac{3}{4}\aI-3\eta^{1/16})\geq(\half\aI+\eta^{1/2})k$, by Lemma~\ref{l:eleven}, $G[W,V(M_2)]$ contains a red connected-matching on at least $\aI k$ vertices, completing the proof in this case.

{\bf Case G.g:} Since $|X_1|\geq\tfrac{3}{4}(\aI-102\eta^{1/2})k$ and $G$ is $4\eta^4 k$-almost-complete, by Corollary~\ref{dirac1a}, $G[X_1]$ contains a red connected-matching $R_{11}$ on at least $(\tfrac{3}{4}\aI-78\eta^{1/2})k$ vertices. Similarly, $G[X_2]$ contains a red connected-matching $R_{12}$ on at least $(\tfrac{3}{4}\aI-78\eta^{1/2})k$ vertices, $G[Y_1]$ contains a red connected-matching $R_{21}$ on at least $(\tfrac{3}{4}\aI-78\eta^{1/2})k$ vertices and $G[Y_2]$ contains a red connected-matching $R_{22}$ on at least $(\tfrac{3}{4}\aI-78\eta^{1/2})k$ vertices. 

Thus, the existence of a triple $w\in W$, $x\in X$, $y\in Y$ such that both the edges $wx$ and $wy$ are coloured red would give a red connected-matching on at least $\aI k$ vertices. Therefore, there exists a partition of $W$ into $W_1\cup W_2$ such that all edges present in $G[W_1,X]$ and $G[W_2,Y]$ are coloured exclusively green. Thus, $G[W_1\cup X]$ and $G[W_2\cup Y]$ each  consists of a single odd green component.
\begin{figure}[!h]
\centering
\vspace{2mm}
\includegraphics[width=68mm, page=40]{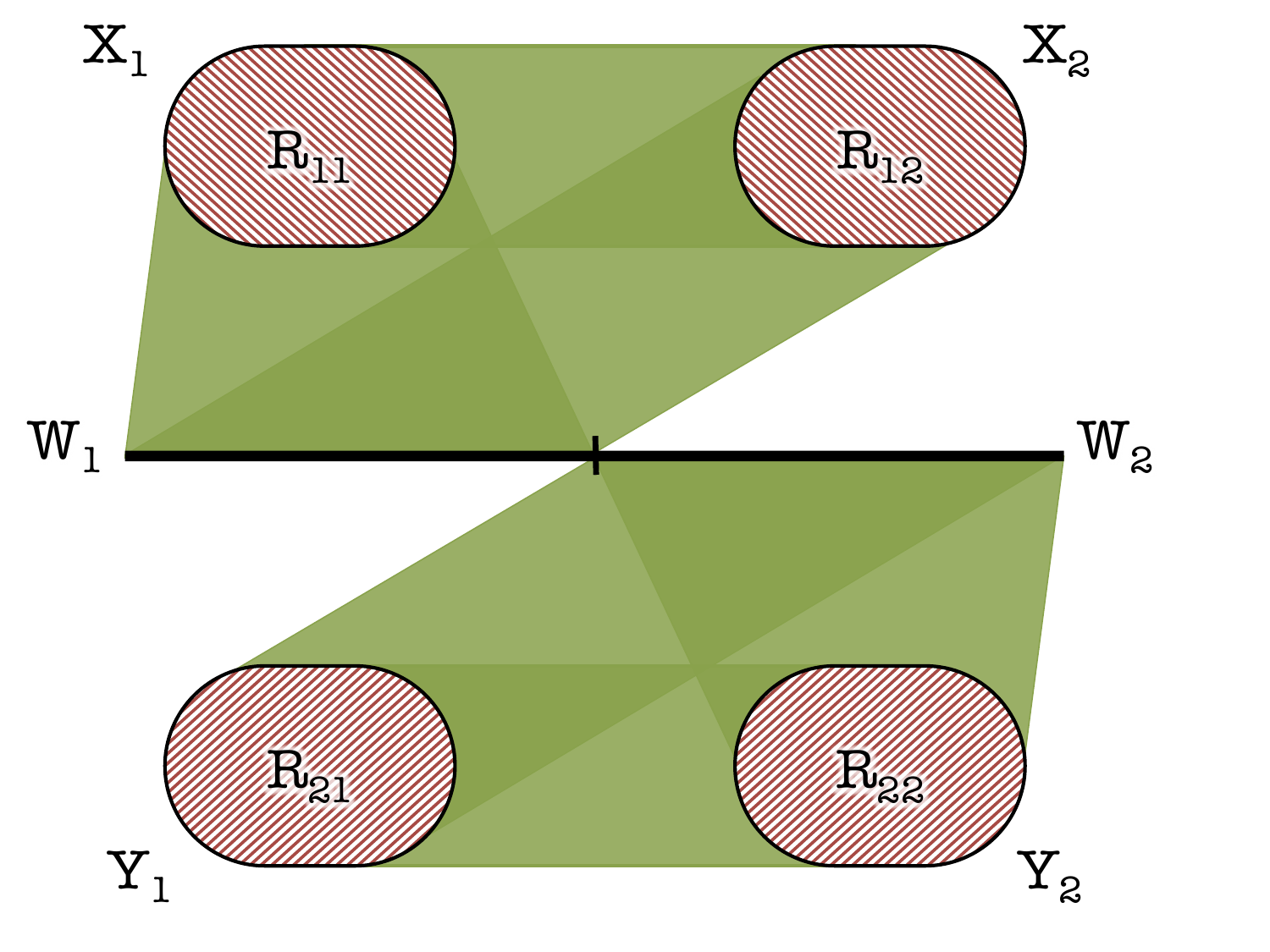}
\vspace{-3mm}\caption{Colouring of $G$ in Case G.g.}
\end{figure}

Without loss of generality, $|W_1|\geq\half(|W_1|+|W_2|)\geq \half(\aI-11.5\eta^{1/2})k\geq(\tfrac{1}{3}\aIII-24\eta^{1/2})k\geq200\eta^{1/2}k$. Thus, we may partition $W_1$ into $W_{11}\cup W_{12}$ such that $|W_{11}|,|W_{12}|\geq 100\eta^{1/2}$. 

Since $|X_1|,|X_2|\geq(\half\aIII-90\eta^{1/2})k$, we may partition $X_1$ into $X_{11}\cup X_{12}$ and $X_2$ $X_{21}\cup X_{22}$ such that $|X_{11}|,|X_{21}|\geq 100\eta^{1/2}k$ and $|X_{21}|,|X_{22}|\geq (\half\aIII-192\eta^{1/2})k$.
Then, by Lemma~\ref{l:eleven}, $G[X_{11},W_{11}]$ and $G[X_{21},W_{12}]$ each contains a green connected-matching on at least $198\eta^k$ vertices and $G[X_{12},X_{22}]$ contains a green connected-matching on at least $(\aIII-386\eta^{1/2})k$ vertices. These green connected-matchings share no vertices but belong to the same odd green component and thus form a green connected-matching on at least $\aIII k$ vertices completing the proof of this case.

\begin{center}
***
\end{center}

Recall that early in the proof, we assumed that $d(G_2)>\half(c-\aI-2\eta)k$. If instead we assume that $d(G_3)>\half(c-\aI-2\eta)k$, the result is the same but with the roles of $\aII$ and $\aIII$ and blue and green exchanged.
\qed

\label{Biiiend}

\section{Proof of the main result -- Setup}
\label{s:p10}

For $\aI,\aII,\aIII>0$,  we set $c=\aI+\max\{3\aI, 2\aII, 2\aIII\}$, 
$$ \eta=\frac{1}{2}\min\left\{\eta_{D}(\aI,\aII,\aIII), 10^{-50}, \left(\frac{\aI}{100}\right)^{128}, \left(\frac{\aII}{100}\right)^{128}, \left(\frac{\aIII}{100}\right)^{128} \right\}$$ 
and let~$k_0$ be the smallest integer such that
$$k_0\geq \max \left\{\left(c-\half{\eta}\right)k_{D}(\aI,\aII,\aIII,\eta), \frac{100}{\eta}\right\}.$$
We let
\vspace{-0.8em}
$$N=\llangle \aI n\rrangle+\max\left\{3\llangle \aI n\rrangle, 2\langle \aII n \rangle, 2\langle\aIII n \rangle\right\} - 3,$$ for some integer $n$ such that $N \geq  K_{\ref{l:sze}}(\eta^4,k_0)$ and
$$n>n^*=\max\{n_{\ref{th:blow-up}}(4,0,0,\eta), n_{\ref{th:blow-up}}(1,2,0,\eta), n_{\ref{th:blow-up}}(1,0,2,\eta),1/\eta, 100/\min\{\aI,\aII,\aIII\}\}$$
and consider a three-colouring of~$G=(V,E)$, the complete graph on~$N$ vertices. 

In order to prove Theorem C, we must prove that $G$ contains either a red cycle on $\llangle \aI n \rrangle$ vertices, a blue cycle on $\langle \aII n \rangle$ vertices or a green cycle on $\langle \aIII n \rangle$ vertices.

Recall that we use $G_1, G_2, G_3$ to refer to the monochromatic spanning subgraphs of $G$. That is,~$G_1$ (resp. $G_2, G_3$) has the same vertex set as~$G$ and includes as an edge any edge which in~$G$ is coloured red (resp. blue, green). 

By Theorem~\ref{l:sze}, there exists an $(\eta^4,G_1,G_2,G_3)$-regular partition $\Pi=(V_0,V_1,\dots,V_K)$ for some $K$ such that $k_0\leq K \leq K_{\ref{l:sze}}(\eta^4,k_0)$. Given such a partition, we define the $(\eta^4,\eta,\Pi)$-reduced-graph $\cG=(\cV,\cE)$ on~$K$ vertices as in Definition~\ref{reduced}. The result is a three-multicoloured graph $\cG=(\cV,\cE)$ with
\begin{align*}
\cV&=\{V_1,V_2,\dots,V_K\}, &
\cE&=\{V_iV_j : (V_i,V_j) \text{ is } (\eta^4,G_r)\text{-regular for }r=1,2,3\},
\end{align*}
such that a given edge $V_iV_j$ of~$\cG$ is coloured with every colour for which there are at least $\eta|V_i||V_j|$ edges of that colour between~$V_i$ and~$V_j$ in~$G$. 

In what follows, we will use $\cG_1, \cG_2, \cG_3$ to refer to the monochromatic spanning subgraphs of the reduced graph $\cG$. That is,~$\cG_1$ (resp. $\cG_2, \cG_3$) has the same vertex set as~$\cG$ and includes as an edge any edge which in~$\cG$ is coloured red (resp. blue, green). 

Note that, by scaling, we may assume that $\max\{\aI,\aII,\aIII\}= 1$. Thus, since $K\geq k_0 \geq 100/\eta$, we may fix an integer~$k$ such that 
\begin{align}
\left(c-\eta\right)k \leq K \leq \left(c-\half\eta\right)k,
\label{sizeK}
\end{align}
and may assume that $2k\leq K\leq 4k$, $2n\leq N\leq 4n$.

Notice, also, that since the partition is~$\eta^4$-regular, we have $|V_0|\leq \eta^4 N$ and, for $1\leq i \leq K$,
\begin{equation}
\label{NK}
(1-\eta^4)\frac{N}{K}\leq |V_i|\leq \frac{N}{K}.
\end{equation}

\hypertarget{reD}
Applying Theorem~\hyperlink{thD}{D}, we find that $\cG$ contains at least one of the following:
\begin{itemize}
\item [(i)]	a red connected-matching on at least $\aI
 k$ vertices;
\item [(ii)]  a blue connected-matching on at least $\aII k$ vertices;
\item [(iii)]  a green odd connected-matching on at least $\aIII k$ vertices;
\item [(iv)] subsets of vertices $\cW$, $\cX$ and $\cY$ such that $\cX\cup \cY\subseteq \cW$, $\cX\cap \cY=\emptyset$, $|\cW|\geq(c-\eta^{1/2})k$, every $\gamma$-component of $G[\cW]$ is odd, $G[\cX]$ contains a two-coloured spanning subgraph $H$ from $\cH_{1}\cup\cH_{2}$ and $G[Y]$ contains a a two-coloured spanning subgraph $K$ from $\cH_{1}\cup\cH_{2}$, where~{\phantom{nnn}}
\vspace{-2mm}
\begin{align*}
\cH_1&=\cH\left((\aI-2\eta^{1/64})k,(\half\alpha_{*}-2\eta^{1/64})k,4\eta^2 k,\eta^{1/64},\text{red},\gamma\right),\text{ } 
\\ \cH_2&=\cH\left((\alpha_{*}-2\eta^{1/64})k,(\half\aI-2\eta^{1/64})k,4\eta^2 k,\eta^{1/64},\gamma,\text{red}\right),
\end{align*}

\vspace{-3mm}
for $(\alpha_{*},\gamma)\in\{(\aII,\text{blue}),(\aIII,\text{green})\}$;
\item[(v)] 
disjoint subsets of vertices $\cX$ and $\cY$ such that $G[\cX]$ contains a two-coloured spanning subgraph~$H$ from $\cH_{2}^*\cup\cJ_b$ and $G[\cY]$ contains a two-coloured spanning subgraph $K$ from $\cH_{2}^*\cup\cJ_b$, where
\vspace{-2mm}
\begin{align*}
\cH_{2}^*&=\cH\left((\beta-2\eta^{1/32})k,(\half\aI-2\eta^{1/32})k,4\eta^4 k,\eta^{1/32},\gamma,\text{red}\right),\text{ } 
\\ \cJ_b&=\cJ\left((\aI-18\eta^{1/2}), 4\eta^4 k, \text{red}, \gamma\right),
\end{align*}

\vspace{-3mm}
for $\beta=\max\{\aII,\aIII\}$ and $\gamma\in\{\text{blue}, \text{green}\}$;
\item[(vi)]  a subgraph $H$ from $\cL=\cL\left((\half\alpha+\tfrac{1}{4}\eta)k,4\eta^4k, \text{red}, \text{blue}, \text{green}\right).$
\end{itemize}

Furthermore, 
\begin{itemize}
\item[(iv)] occurs only if $\aI\leq\max\{\aII,\aIII\}\leq\alpha_{*}+24\eta^{1/4}$. Additionally,  $H$ and $K$ belong to $\cH_{1}$ if $\alpha_{*}\leq(1-\eta^{1/16})\aI$ and belong to $\cH_{2}$ if $\aI\leq(1-\eta^{1/16})\alpha_{*}$;
 \item[(v)] occurs only if $\aI\leq\beta$. Additionally, $H$ and $K$ may belong to $\cH_2^*$ only if $\aI\leq(1-\eta^{1/16})\beta$ and may belong to $\cJ$ only if $\beta<(\tfrac{3}{2}+2\eta^{1/4})\aI$; and
 \item[(vi)] occurs only if $\aI\geq\max\{\aII,\aIII\}$.
 \end{itemize} 

Since $n>\max\{n_{\ref{th:blow-up}}(4,0,0,\eta), n_{\ref{th:blow-up}}(1,2,0,\eta), n_{\ref{th:blow-up}}(1,0,2,\eta)\}$ and $\eta<10^{-20}$, in cases (i)--(iii), Theorem~\ref{th:blow-up} gives a cycle of appropriate length, colour and parity to complete the proof. 
Thus, we need only concern ourselves with cases (iv)--(vi). We divide the remainder of the proof into three parts, each corresponding to one of the possible coloured structures.

\section{Proof of the main result -- Part I -- Case (iv)}
\label{s:p11}

Suppose that $\cG$ contains subsets of vertices $\cW$, $\cX$ and $\cY$ such that $\cX\cup \cY\subseteq \cW$, $\cX\cap \cY=\emptyset$, $|\cW|\geq(c-\eta^{1/2})k$, every blue-component of $G[\cW]$ is odd, $G[\cX]$ contains a red-blue-coloured spanning subgraph~$H$ from $\cH_{1}\cup\cH_{2}$ and $G[Y]$ contains a a two-coloured spanning subgraph $K$ from $\cH_{1}\cup\cH_{2}$, where
\begin{align*}
\cH_1&=\cH\left((\aI-2\eta^{1/64})k,(\half\aII-2\eta^{1/64})k,4\eta^2 k,\eta^{1/64},\text{red},\text{blue}\right),\text{ } 
\\ \cH_2&=\cH\left((\aII-2\eta^{1/64})k,(\half\aI-2\eta^{1/64})k,4\eta^2 k,\eta^{1/64},\text{blue},\text{red}\right).
\end{align*}
Recall from Theorem~D, that we may additionally assume that
\begin{equation}
\label{sizesI} 
\aI\leq\max\{\aII,\aIII\}\leq\aII+24\eta^{1/2}.
\end{equation}
We divide the proof that follows into three sub-parts depending on the colouring of the subgraphs~$H$~and~$K$, that is, whether $H$ and $K$ belong to $\cH_1$ or $\cH_2$:

\subsection*{Part I.A: $H, K \in \cH_1$.}

From Theorem~\hyperlink{reD}{D}, we know that this case can arise only when $\aI\geq(1-\eta^{1/16})\aII$. Thus, recalling~(\ref{sizesI}), we know that 
\begin{equation}
\label{sizesIA} 
\aII-\eta^{1/16}\leq\aI\leq\max\{\aII,\aIII\}\leq\aII+24\eta^{1/2}.
\end{equation}
We have a natural partition of
the vertex set~$\cV$ of~$\cG$ into $\cX _1 \cup \cX _2 \cup \cY _1 \cup \cY _2\cup \cZ$, where $\cX_1\cup\cX_2$ is the partition of the vertices of $H$ given by Definition~\ref{d:H} and  $\cY_1\cup\cY_2$ the corresponding partition of the vertices of $K$. Thus, $\cX_1\cup\cX_2\cup\cY_1\cup\cY_2\subseteq \cW$ and 
\begin{align}
\label{IA0a-}
|\cX_1|,|\cY_1| & \geq (\aI-2\eta^{1/64})n,
& |\cX_2|,|\cY_2| & \geq (\half \aII-2\eta^{1/64})n.
\end{align}
By the definition of $\cH_1$, we know that ~$\cG_1[\cX_1]$ is $(1-\eta^{1/64})$-complete and so, by Theorem~\ref{dirac}, it contains a red connected-matching on at least $|\cX_1|-1\geq(\aI-4\eta^{1/64})k$ vertices. Similarly,~$\cG[\cY_1]$ contains a red connected-matching on at least $(\aI-4\eta^{1/64})k$ vertices. Thus, the existence of a red edge in~$\cG[\cX_1,\cY_1]$ would imply the existence of a red-connected-matching on at least $\aI k$ vertices and, therefore, we may assume that there are no red edges present in~$\cG[\cX_1,\cY_1]$.

Again, by the definition of $\cH_1$, we know that~$\cG_2[\cX_1,\cX_2]$ is $(1-\eta^{1/64})$-complete and so, by Lemma~\ref{l:ten}, it contains a blue connected-matching on at least $(\aII-8\eta^{1/64})k$ vertices. Similarly,~$\cG_2[\cY_1,\cY_2]$ contains a blue connected-matching on at least $(\aII-8\eta^{1/64})k$ vertices. Thus, recalling that every component of~$\cG[\cW]$ may be assumed to be odd, the existence of a blue edge in $\cG[\cX_1\cup \cX_2,\cY_1\cup \cX_2]$ would imply the existence of a blue odd connected-matching on at least $\aII k$ vertices. Therefore, we may assume that there are no blue edges present in~$\cG[\cX_1\cup \cX_2,\cY_1\cup \cY_2]$.

Suppose there exists a red matching $\cR_1$ in $\cG[\cX_1,\cY_2]$ such that $8\eta^{1/64}k\leq V(\cR_1)\leq 10\eta^{1/64}k$. Then, recalling (\ref{IA0a-}), there exists a subset $\widetilde{\cX}_1$ of at least $(\aI-7\eta^{1/64})k$ vertices from $\cX_1$ such that $\cX_1$ and $V(\cR_1)$ share no vertices. By Theorem~\ref{dirac},~$\cG[\widetilde{\cX}_1]$ contains a red connected-matching $\cR_2$ on at least $|\widetilde{\cX_1}|-1\geq(\aI-8\eta^{1/64})k$ vertices. Observe then that $\cR_1$ and $\cR_2$ share no vertices and therefore, since~$\cG_1[\cX_1]$ is connected, form a red connected-matching on at least $\aI k$ vertices. Thus, after moving at most~$5\eta^{1/64}k$ vertices from each of $\cX_1, \cX_2, \cY_1$ and $\cY_2$ into $\cZ$, we may assume that there are no red edges present in~$\cG[\cX_1,\cY_2]$ or~$\cG[\cX_2,\cY_1]$.

In summary, moving vertices from $\cX_1\cup\cX_2\cup\cY_1\cup\cY_2$ to $\cZ$, we may now assume that we have a partition of $\cV(\cG)$ into $\cX _1 \cup \cX _2 \cup \cY _1 \cup \cY _2\cup \cZ$ with 
\begin{equation}
\left.
\label{IA0a}
\begin{aligned}
\quad\quad\quad\quad\quad\quad\quad\quad\,\,\,
(\aI-7\eta^{1/64})k\leq|\cX _1|&=|\cY _1|=p\leq \aI k, 
\quad\quad\quad\quad\quad\quad\quad\,\,\,\\
(\half\aII-7\eta^{1/64})k\leq|\cX _2|&=|\cY _2|=q\leq \half\aII k, 
\end{aligned}
\right\}\!
\end{equation}
such that 
\begin{itemize}
\labitem{HA1}{HM1} 
$\cG_1[\cX _1]$ and $\cG_1[\cY _1]$ are each $(1-2\eta^{1/16})$-complete (and thus connected), \\
$\cG_2[\cX _1]$ and $\cG_2[\cY _1]$ are each $2\eta^{1/16}$-sparse, \\
 $\cG_3[\cX _1]$ and $\cG_3[\cY _1]$ each contain no edges;

\labitem{HA2}{HM2} $\cG_1[\cX_1,\cX_2]$ and $\cG_1[\cY_1,\cY_2]$ are each   $2\eta^{1/16}$-sparse, \\
$\cG_2[\cX_1,\cX_2]$ and $\cG_2[\cY_1,\cY_2]$ are each $(1-2\eta^{1/16})$-complete (and thus connected),\\ 
$\cG_3[\cX_1,\cX_2]$ and $\cG_3[\cY_1,\cY_2]$ each contain no green edges;

\labitem{HA3}{HM4} $\cG[\cX_2]$ and $\cG[\cY_2]$ each contain no green edges, \\ all edges present in $\cG[\cX_1,\cY_2] \cup \cG[\cX_2,\cY_1] \cup \cG[\cX_1,\cY_1]$ are coloured exclusively green;

\labitem{HA4}{HM3} $\cG[\cX_1\cup\cX_2\cup\cY_1\cup\cY_2]$ is $4\eta^4 k$-almost-complete (and thus connected),  \\
every component of $\cG_2[\cX_1\cup \cX_2 \cup \cY_1 \cup \cY_2]$ is odd.
\end{itemize}
By (\ref{sizesIA}) and (\ref{IA0a}), since $\eta\leq\left(\frac{\aIII}{100}\right)^{128}$, we have $$|\cX_1|\geq(\aI-7\eta^{1/64})k\geq(\half\aIII+4\eta^2)k.$$
By (\ref{HM3}), $\cG[\cX_1\cup\cX_2\cup\cY_1\cup\cY_2]$ is $4\eta^2 k$-almost-complete, and by (\ref{HM4}) all edges present in $\cG[\cX_1,\cY_1]$ are coloured exclusively green. Thus, by Lemma~\ref{l:eleven}, $\cG[\cX_1,\cY_1]$ contains a green connected-matching $\cM$ on at least $\aIII k$ vertices. By (\ref{HM4}), all the vertices of $\cX_1\cup \cX_2 \cup \cY_1 \cup \cY_2$ belong to the same green component. If there existed a pair of green edges $xz$, $yz$ with $x\in\cX_1\cup\cX_2$, $y\in\cY_1\cup\cY_2$, $z\in\cZ$, then the component of $\cG_3$ containing $\cM$ would be odd. Thus, we may instead assume that we can partition $\cZ$ into $\cZ_X\cup\cZ_Y$ such that there are no green edges in $\cG[\cZ_X,\cX_1\cup\cX_2]\cup\cG[\cZ_Y,\cY_1\cup\cY_2]$. 

\begin{figure}[!h]
\centering
\vspace{-1mm}\hspace{2mm}
{\includegraphics[height=50mm, page=7]{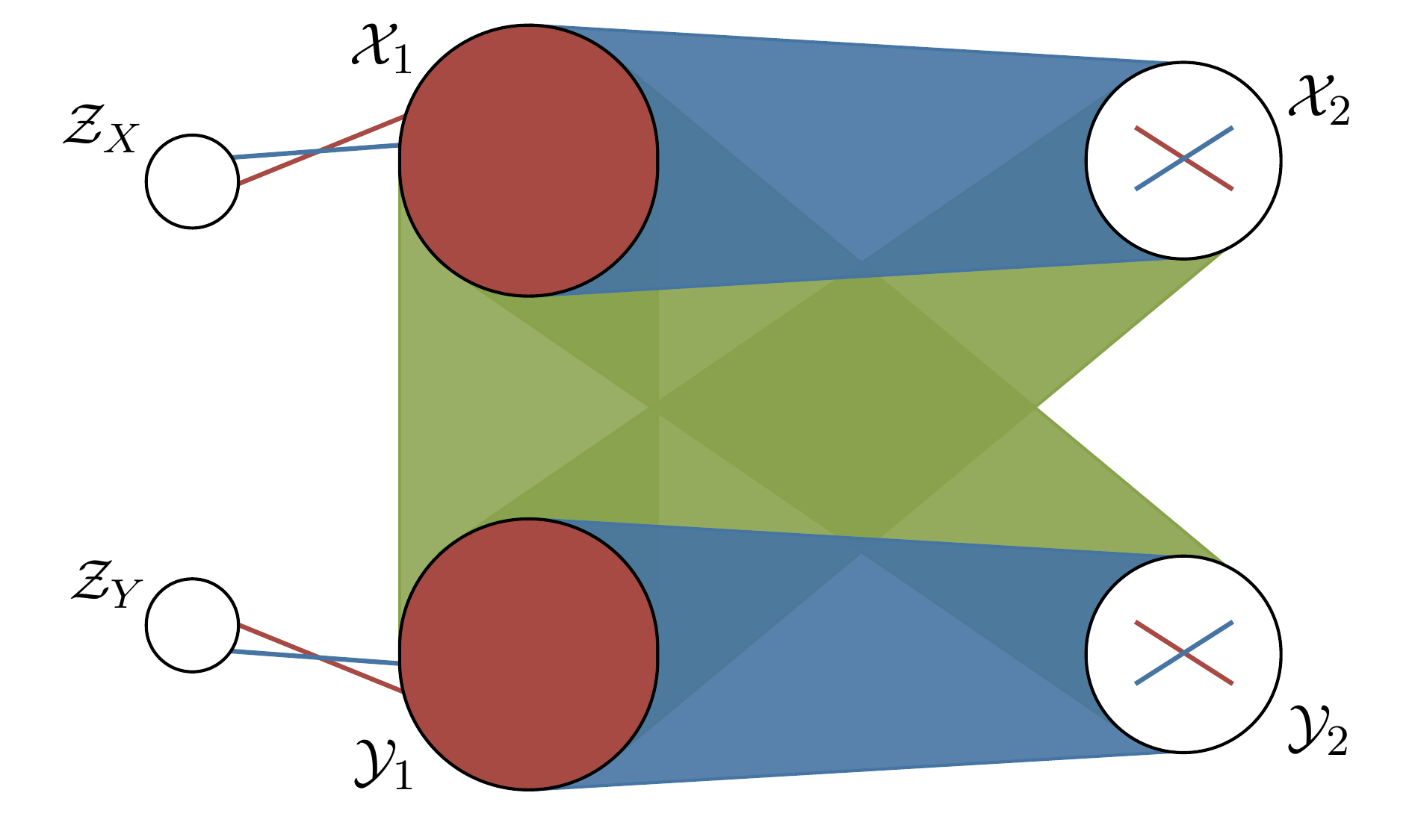}\hspace{18mm}}
\vspace{-1mm}\caption{Partition of the vertices of the reduced-graph. }
  \end{figure}
  
Recalling that $|V(\cG)|\geq(c-\eta)k\geq(4\alpha_1-\eta)k$, without loss of generality, we may assume that 
$$|\cX_1\cup\cX_2\cup \cZ_X|\geq(2\alpha_1-\eta)k$$ 
(since, if not, then $\cY_1 \cup \cY_2\cup \cZ_Y$ is that large instead). 
We further partition $\cZ_X$ into $\cZ_R\cup \cZ_B$ by defining
\begin{align*}
\cZ_B&=\{ z\in \cZ_X \text{ such that } z \text{ has at least } |\cX_1|-16\eta^{1/2} \text{ blue neighbours in } \cX_1\};\text{ and}
\\ \cZ_R&=\cZ\backslash \cZ_B=\{ z\in \cZ_X \text{ such that } z \text{ has at least } 16\eta^{1/2} \text{ red neighbours in } \cX_1\}.
\end{align*}

Given this definition, suppose that $|\cX_1\cup\cZ_R|\geq(\aI+\eta^{1/2})k$. Then, let $\widetilde{X}$ be any subset of $\cX_1\cup\cZ_R$ such that 
$$(\aI+\eta^{1/2})k\leq|\widetilde{X}|\leq(\aI+2\eta^{1/2})k$$
which includes every vertex of $\cX_1$. Given (\ref{IA0a}), we have 
$$|\cX_1\cap\widetilde{\cX}|\geq(\aI-7\eta^{1/64})k\,\,\,\text{and}\,\,\,|\cZ_R\cap \widetilde{\cX}|\leq 8\eta^{1/64} k.$$

By (\ref{HM1}), $\cG_1[\cX_1]$ is $(1-2\eta^{1/64})$-complete and so is $4\eta^{1/64}$-almost-complete. Thus, $\cG_1[\widetilde{\cX}]$ consists of at least $|\widetilde{\cX}|-7\eta^{1/64}n$ vertices of degree at least $|\widetilde{\cX}|-12\eta^{1/64}n$ and at most $8\eta^{1/64}n$ vertices of degree at at least $16\eta^{1/64}n$ and so, by Theorem~\ref{chv}, $\cG_1[\widetilde{\cX}]$ is Hamiltonian and, thus, contains a red connected-matching on at least $\aI k$ vertices.

Thus, we may, instead, suppose that $|\cX_2\cup\cZ_B|\geq(\aI-2\eta^{1/2})k\geq(\half\aII+\eta^{1/2})k$, in which case, we consider the blue graph $\cG_2[\cX_1,\cX_2\cup \cZ_B]$. Given the relative sizes of~$\cX_1$ and~$\cX_2\cup \cZ_B$ and the large minimum-degree of the graph, we can use Theorem~\ref{moonmoser} to give a blue connected-matching on at least~$\aII k $ vertices. Indeed, by~(\ref{sizesIA}) and (\ref{IA0a}), we have 
$|\cX_1|\geq(\aI-7\eta^{1/64})k\geq(\aII-8\eta^{1/64})k$ and may choose subsets $\widetilde{\cX}_1\subseteq \cX_1$, $\widetilde{\cX}_2\subseteq \cX_2\cup \cZ_B$ such that 
\begin{align*}
(\half\aII +\eta^{1/2})k\leq|\widetilde{\cX}_1|&=|\widetilde{\cX}_2|\leq(\half\aII +2\eta^{1/2})k, & |\widetilde{\cX}_2\cup \cZ_B|\leq8\eta^{1/64}k.
\end{align*}
Recall that $\cG_2[\cX_1,\cX_2]$ is $4\eta^{1/64}k$-almost-complete and that all vertices in $\cZ_B$ have blue degree at least $|\widetilde{\cX}_1|-16\eta^{1/64}k$ in $G[\widetilde{X}_1,\widetilde{X}_2]$. Thus, since $|\widetilde{\cX}_2\cup \cZ_B|\leq 8\eta^{1/64}k$ and $\eta\leq(\aII/100)^{128}$, for any pair of vertices $x_1\in\widetilde{\cX}_1$ and $x_2\in\widetilde{\cX}_2$, we have $$d(x_1)+d(x_2)\geq |\widetilde{\cX}_1|+|\widetilde{\cX}_2|-32\eta^{1/64}n\geq(\half\aII+\eta^{1/2}) k+1.$$ Therefore, by Theorem~\ref{moonmoser}, $\cG_2[\widetilde{\cX}_1,\widetilde{\cX}_2]$ contains a blue connected-matching on at least $ \aII k $ vertices, which is odd since every blue component of $\cG[\cW]$ is assumed to be odd, thus completing Part I.A.

\subsection*{Part I.B: $H \in \cH_1, K \in \cH_2$.}
From Theorem~\hyperlink{reD}{D}, we know that this case can arise only when $\aI$ and $\aII$ are close in size, specifically when 
\begin{equation}
\label{a12same}
(1-\eta^{1/16})\aI\leq\aII
\leq (1-\eta^{1/16})^{-1} \aI \leq (1+2\eta^{1/16}) \aI.
\end{equation}

Following the same argument as in Part I.A. but exchanging the roles of red and blue and the roles of~$\aI$ and $\aII$ when necessary, we may obtain a partition of $\cV(\cG)$ into $\cX _1 \cup \cX _2 \cup \cY _1 \cup \cY _2\cup \cZ$ with 
\begin{equation}
\left.
\label{IC0a}
\begin{aligned}
\quad\quad\quad\quad\quad\quad\quad\quad\,\,\,
(\aI-7\eta^{1/64})k\leq|\cX _1|&=p\leq \aI k, 
\quad\quad\quad\quad\quad\quad\quad\,\,\,\\
(\half\aII-7\eta^{1/64})k\leq|\cX _2|&=q\leq \half\aII k, 
\quad\quad\quad\quad\quad\quad\quad\quad\,\,\, \\
(\aII-7\eta^{1/64})k\leq|\cY _1|&=r\leq \aII k, 
\quad\quad\quad\quad\quad\quad\quad\,\,\,\\
(\half\aI-7\eta^{1/64})k\leq|\cY _2|&=s\leq \half\aI k, 
\end{aligned}
\right\}\!
\end{equation}
such that 

\begin{itemize}
\labitem{HB1}{HO1} 
$\cG_1[\cX _1]$ and $\cG_2[\cY _1]$ are each $(1-2\eta^{1/16})$-complete (and thus connected), \\
$\cG_2[\cX _1]$ and $\cG_1[\cY _1]$ are each $2\eta^{1/16}$-sparse, \\
 $\cG_3[\cX _1]$ and $\cG_3[\cY _1]$ each contain no edges;

\labitem{HB2}{HO2} $\cG_2[\cX_1,\cX_2]$ and $\cG_1[\cY_1,\cY_2]$ are each $(1-2\eta^{1/16})$-complete (and thus connected),\\
 $\cG_1[\cX_1,\cX_2]$ and $\cG_2[\cY_1,\cY_2]$ are each   $2\eta^{1/16}$-sparse, \\
$\cG_3[\cX_1,\cX_2]$ and $\cG_3[\cY_1,\cY_2]$ each contain no edges;

\labitem{HB3}{HO4} 
$\cG[\cX_2]$ and $\cG[\cY_2]$ each contain no green edges,
\\ all edges present in $\cG[\cX_1,\cY_2] \cup \cG[\cX_2,\cY_1] \cup \cG[\cX_1,\cY_1]$ are coloured exclusively green;

\labitem{HB4}{HO3} $\cG[\cX_1\cup\cX_2\cup\cY_1\cup\cY_2]$ is $4\eta^4 k$-almost-complete (and thus connected),  \\
every component of $\cG_2[\cX_1\cup \cX_2 \cup \cY_1 \cup \cY_2]$ is odd.
\end{itemize}

By (\ref{sizesI}) and (\ref{IC0a}), since $\eta\leq\left(\frac{\aI}{100}\right)^{128},\left(\frac{\aII}{100}\right)^{128}$, we have 
\begin{align*}
|\cX_1|&\geq(\aI-7\eta^{1/64})k\geq(\half\aIII+4\eta^2)k, \\
|\cY_1|&\geq(\aII-7\eta^{1/64})k\geq(\half\aIII+4\eta^2)k.
\end{align*}

By (\ref{HO3}), $\cG[\cX_1\cup\cX_2\cup\cY_1\cup\cY_2]$ is $4\eta^2 k$-almost-complete and by (\ref{HM4}) all edges present in $\cG[\cX_1,\cY_1]$ are coloured exclusively green. 
By the same argument given in Part I.A,
if there existed a pair of green edges $xz$, $yz$ with $x\in\cX_1\cup\cX_2$, $y\in\cY_1\cup\cY_2$, $z\in\cZ$, 
then $\cG$ would contain an odd green connected-matching on at least $\aIII k$ vertices.

Thus, again, we can partition $\cZ$ into $\cZ_X\cup\cZ_Y$ such that there are no green edges in~$\cG[\cZ_X,\cX_1\cup\cX_2]$ or in~$\cG[\cZ_Y,\cY_1\cup\cY_2]$. 
\begin{figure}[!h]
\centering
\vspace{-1mm}\hspace{2mm}
{\includegraphics[height=50mm, page=8]{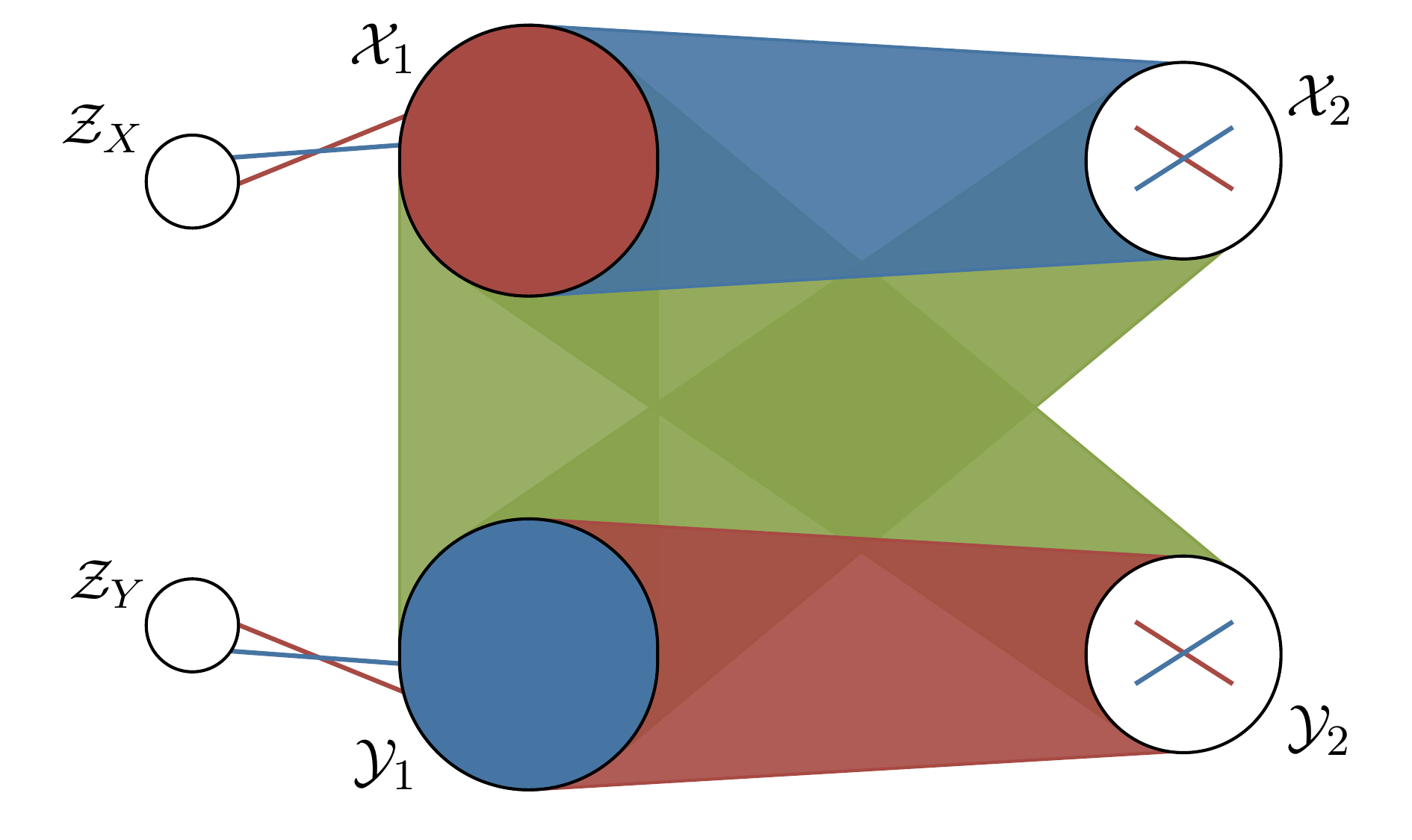}\hspace{18mm}}
\vspace{-1mm}\caption{Partition of the vertices of the reduced-graph. }
  \end{figure}

Recalling that $|V(\cG)|\geq(c-\eta)k\geq(4\alpha_1-\eta)k$, without loss of generality, we may assume that at least one of 
$\cX_1\cup\cX_2\cup \cZ_X$ or $\cY_1 \cup \cY_2\cup \cZ_Y$ contains at least $(2\alpha_1-\eta)k$ vertices.

Suppose, for now, that  $$|\cX_1\cup\cX_2\cup \cZ_X|\geq (2\alpha_1-\eta)k\geq (\aI+\eta^{1/2})k+(\aII-3\eta^{1/16})k.$$
In that case, we further partition $\cZ_X$ into $\cZ_{XR}\cup \cZ_{XB}$ by defining
\begin{align*}
\cZ_{XB}&=\{ z\in \cZ_X \text{ such that } z \text{ has at least } |\cX_1|-16\eta^{1/2} \text{ blue neighbours in } \cX_1\};\text{ and}
\\ \cZ_{XR}&=\cZ\backslash \cZ_{XB}=\{ z\in \cZ_X \text{ such that } z \text{ has at least } 16\eta^{1/2} \text{ red neighbours in } \cX_1\}.
\end{align*}

Given this definition, if $|\cX_1\cup\cZ_{XR}|\geq(\aI+\eta^{1/2})k$, then, 
by the same argument given in Part I.A, $\cG[\cX_1\cup \cZ_{XR}]$ contains
a red connected-matching on at least $\aI k$ vertices.
Thus, we may, instead, suppose that $|\cX_2\cup\cZ_{XB}|\geq(\aI-2\eta^{1/2})k\geq(\half\aII+\eta^{1/2})k$, in which case, 
by the same argument given in Part I.A,
$\cG[\cX_1,\cX_2\cup\cX_{XB}]$
contains a blue connected-matching on at least $ \aII k $ vertices, which is odd since every blue component of $\cG[\cW]$ is assumed to be odd.
This would be sufficient to complete Part~I.B. Therefore we may instead assume that $$|\cY_1\cup\cY_2\cup \cZ_Y|\geq (2\alpha_1-\eta)k\geq(\aII+\eta^{1/2})k+(\aI-3\eta^{1/16})k.$$
In which case, we further partition $\cZ_Y$ into $\cZ_{YR}\cup \cZ_{YB}$ by defining
\begin{align*}
\cZ_{YR}&=\{ z\in \cZ_Y \text{ such that } z \text{ has at least } |\cY_1|-16\eta^{1/2} \text{ red neighbours in } \cY_1\};\text{ and}
\\ \cZ_{YB}&=\cZ_{Y}\backslash \cZ_{YR}=\{ z\in \cZ_Y \text{ such that } z \text{ has at least } 16\eta^{1/2} \text{ blue neighbours in } \cY_1\}.
\end{align*}
Given this definition, if $|\cY_1\cup\cZ_{YB}|\geq(\aII+\eta^{1/2})k$. Then, 
the same argument as used in Part I.A, when considering $\cX\cup\cZ_R$ gives a blue connected-matching in $\cG[\cY_1\cup\cZ_{YB}]$ on 
at least $\aII k$ vertices, which is odd since every blue component of $\cG[\cW]$ is assumed to be odd.
Thus, we may, instead, suppose that $|\cY_2\cup\cZ_{YR}|\geq(\aI-3\eta^{1/2})k\geq(\half\aI+\eta^{1/2})k$, in which case, 
the same argument as used in Part I.A, when considering $\cG[\cX_1,\cX_2\cup\cZ_B]$ gives a red connected-matching in $\cG[\cY_1,\cY_2\cup\cZ_{YR}]$ on
at least $ \aI k $ vertices, completing Part I.B.

\subsection*{Part I.C: $H, K \in \cH_2$.}

From Theorem~\hyperlink{reD}{D}, we know that this case can arise only when 
\begin{equation}
\label{a2big}
\aII\geq(1-\eta^{1/16})\aI.
\end{equation}
Following the same argument as in Part I.A. but exchanging the roles of red and blue and the roles of $\aI$ and $\aII$, we may obtain a partition of $\cV(\cG)$ into $\cX _1 \cup \cX _2 \cup \cY _1 \cup \cY _2\cup \cZ$ with 
\begin{equation}
\left.
\label{IB0a}
\begin{aligned}
\quad\quad\quad\quad\quad\quad\quad\quad\,\,\,
(\aII-7\eta^{1/64})k\leq|\cX _1|&=|\cY _1|=p\leq \aII k, 
\quad\quad\quad\quad\quad\quad\quad\,\,\,\\
(\half\aI-7\eta^{1/64})k\leq|\cX _2|&=|\cY _2|=q\leq \half\aI k, 
\end{aligned}
\right\}\!
\end{equation}
such that 
\begin{itemize}
\labitem{HC1}{HN1} 
$\cG_1[\cX _1]$ and $\cG_1[\cY _1]$ are each $2\eta^{1/16}$-sparse, \\
$\cG_2[\cX _1]$ and $\cG_2[\cY _1]$ are each $(1-2\eta^{1/16})$-complete (and thus connected), \\
 $\cG_3[\cX _1]$ and $\cG_3[\cY _1]$ each contain no edges;

\labitem{HC2}{HN2} $\cG_1[\cX_1,\cX_2]$ and $\cG_1[\cY_1,\cY_2]$ are each $(1-2\eta^{1/16})$-complete (and thus connected),\\
 $\cG_2[\cX_1,\cX_2]$ and $\cG_2[\cY_1,\cY_2]$ are each   $2\eta^{1/16}$-sparse, \\
$\cG_3[\cX_1,\cX_2]$ and $\cG_3[\cY_1,\cY_2]$ each contain no edges;

\labitem{HC3}{HN4} 
$\cG[\cX_2]$ and $\cG[\cY_2]$ each contain no green edges,
\\ all edges present in $\cG[\cX_1,\cY_2] \cup \cG[\cX_2,\cY_1] \cup \cG[\cX_1,\cY_1]$ are coloured exclusively green;

\labitem{HC4}{HN3} $\cG[\cX_1\cup\cX_2\cup\cY_1\cup\cY_2]$ is $4\eta^4 k$-almost-complete (and thus connected),  \\
every component of $\cG_2[\cX_1\cup \cX_2 \cup \cY_1 \cup \cY_2]$ is odd.
\end{itemize}

\begin{figure}[!h]
\centering
\includegraphics[width=64mm, page=31]{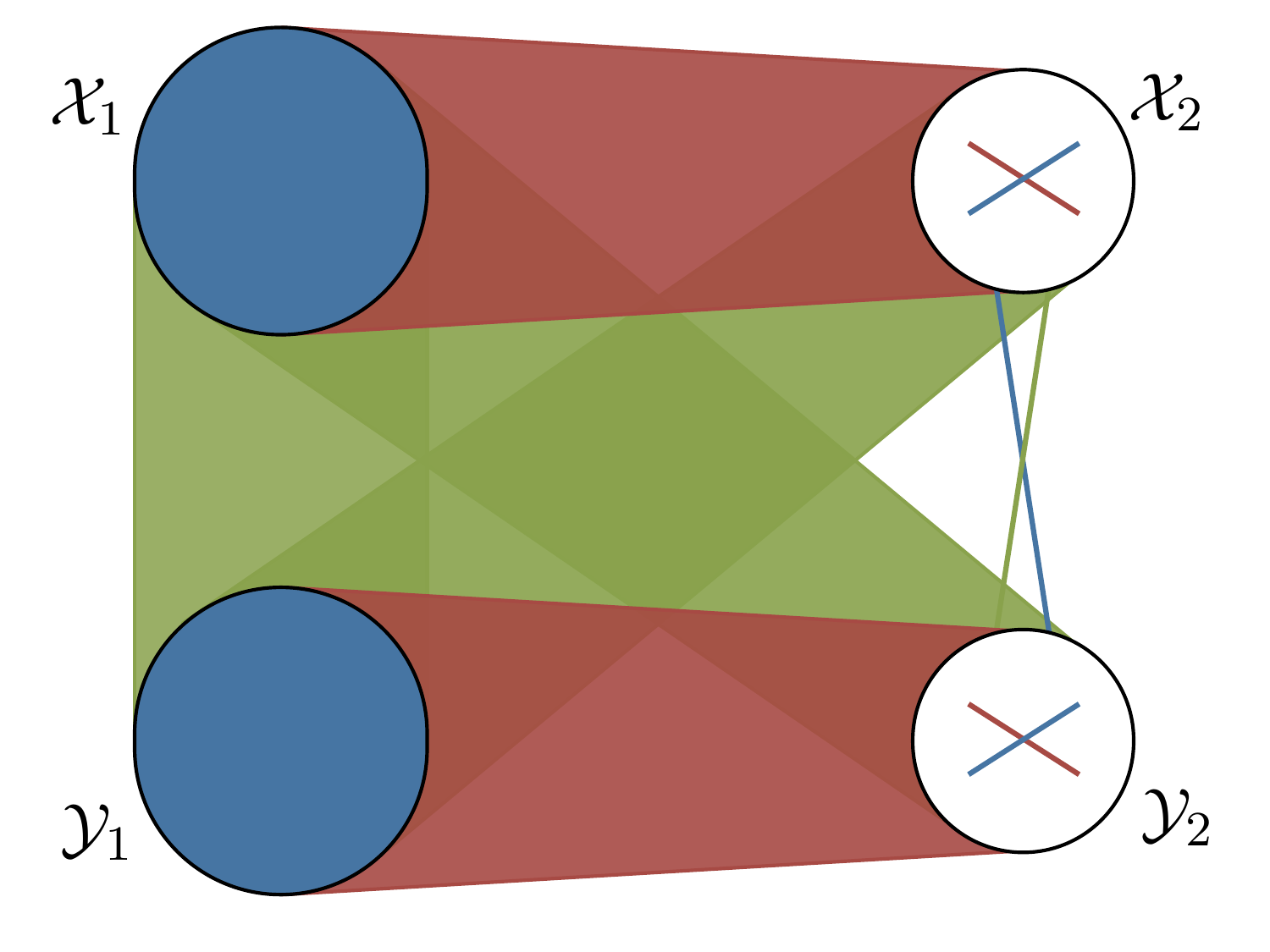}
\vspace{0mm}\caption{Coloured structure of the  reduced-graph in Part I.C.}  
\end{figure}

 The remainder of this section focuses on showing that the original graph must have a similar structure, which can then be exploited in order to force a cycle of appropriate length, colour and parity.

By definition, each vertex $V_{i}$ of $\cG=(\cV,\cE)$ represents a class of vertices of $G=(V,E)$. In what follows, we will refer to these classes as \emph{clusters} (of vertices of~$G$). Additionally, recall, from (\ref{NK}), that
$$(1-\eta^4)\frac{N}{K}\leq |V_i|\leq \frac{N}{K}.$$ 

Since $n> \max\{n_{\ref{th:blow-up}}(4,0,0,\eta), n_{\ref{th:blow-up}}(1,2,0,\eta), n_{\ref{th:blow-up}}(1,0,2,\eta)\}$, we can (as in the proof of Theorem~\ref{th:blow-up}) prove that $$
|V_i|\geq \left(1+\frac{\eta}{24}\right)\frac{n}{k}> \frac{n}{k}.$$

Thus, we can partition the vertices of~$G$ into sets $X_{1}, X_{2}, Y_{1}, Y_{2}$ and $Z$ corresponding to the partition of the vertices of~$\cG$ into $\cX_1, \cX_2, \cY_1,\cY_2$ and $\cZ$. Then,~$X_1, Y_1$ each contain~$p$ clusters of vertices and $X_2, Y_2$ each contain~$q$ clusters and, recalling (\ref{IB0a}), we have
\begin{equation}
\left.
\label{IB1} 
\begin{aligned}
\,\,\,\quad\quad\quad\quad\quad\quad\quad\quad\quad\,\,\,\,
|X_1|,|Y_1| & = p|V_1| \geq (\aII-7\eta^{1/64})n,
\quad\quad\quad\quad\quad\quad\quad\quad
\\
|X_2|,|Y_2| & = q|V_1|\geq (\half \aI-7\eta^{1/64})n.
\end{aligned}
\right\}\!
\end{equation}

In what follows, we will \textit{remove} vertices from $X_1\cup X_2\cup Y_1\cup Y_2$ by moving them into~$Z$ in order to show that, in what remains, $G[X_1\cup X_2\cup Y_1 \cup Y_2]$ has a particular coloured structure.  We begin by proving the below claim which essentially tells us that  $G$ has similar coloured structure to $\cG$:

\begin{claim}
\label{G-struct}
We can \textit{remove} at most $5\eta^{1/128}n$ vertices from each of~$X_1$ and~$Y_1$ and at most $2\eta^{1/128}n$ vertices from each of~$X_2$ and~$Y_2$ so that the following conditions~hold.
 \end{claim}
\begin{itemize}
\labitem{HC5}{HN5Z} $G_2[X_1]$ and $G_2[Y_1]$ are each $4\eta^{1/128}n$-almost-complete; and
\labitem{HC6}{HN5} $G_1[X_1,X_2]$ and $G_1[Y_1,Y_2]$ are each $3\eta^{1/128}n$-almost-complete. 
\end{itemize}

\begin{proof} Consider the complete three-coloured graph $G[X_1]$ and recall from (\ref{HN1}) and (\ref{HN3}) that~$\cG[\cX_1]$ contains only red and blue edges and is $4\eta^2k$-almost-complete. Given the structure of~$\cG$, we can bound the number of non-blue edges in $G[X_1]$ as follows:

Since regularity provides no indication as to the colours of the edges contained within each cluster, these could potentially all be non-blue. There are~$p$ clusters in $X_1$, each with at most $N/K$ vertices. Thus, there are at most $$p\binom{N/K}{2}$$ non-blue edges in~$X_1$ within the clusters.

Now, consider a pair of clusters $(U_1,U_2)$ in $X_1$. If $(U_1, U_2)$ is not $\eta^4$-regular, then we can only trivially bound the number of non-blue edges in $G[U_1,U_2]$ by $|U_1||U_2|\leq (N/K)^2$. However, by~(\ref{HN3}), there are at most $4\eta^2 |\cX_1| k$ such pairs in $\cG$. Thus, we can bound the number of non-blue edges coming from non-regular pairs by
$$4\eta^2 pk \left(\frac{N}{K}\right)^2.$$

If the pair is regular and~$U_1$ and~$U_2$ are joined by a red edge in the  reduced-graph, then, again, we can only trivially bound the number of non-blue edges in $G[U_1,U_2]$ by $(N/K)^2$. However, by~(\ref{HN1}), $\cG_1[X_1]$ is $\eta^{1/64}$-sparse, so there are at most $\eta^{1/64}\binom{p}{2}$ red edges in~$\cG[X_1]$ and, thus, there are at most $$2\eta^{1/64}\binom{p}{2}\left(\frac{N}{K}\right)^2$$ non-blue edges in $G[X_1]$ corresponding to such pairs of clusters.

Finally, if the pair is regular and~$U_1$ and~$U_2$ are not joined by a red edge in the  reduced-graph, then the red density of the pair is at most~$\eta$ (since, if the density were higher, they would be joined by a red edge). Likewise, the green density of the pair is at most~$\eta$ (since there are no green edges in $\cG[X_1]$). Thus, there are at most $$2\eta\binom{p}{2}\left(\frac{N}{K}\right)^2$$ non-blue edges in $G[X_1]$ corresponding to  such pairs of clusters.

Summing the four possibilities above gives an upper bound of 
$$p \binom{N/K}{2} + 4\eta^2 pk \left(\frac{N}{K}\right)^2+2\eta^{1/64}\binom{p}{2}\left(\frac{N}{K}\right)^{2}+2\eta\binom{p}{2}\left(\frac{N}{K}\right)^{2} $$ non-red edges in $G[X_1]$.

Since $K\geq 2k, \eta^{-1}$, $N\leq 4n$ and $p\leq \aII k \leq k$, we obtain
$$e(G_1[X_1])+e(G_3[X_1])\leq [4\eta + 16\eta^{2} +4\eta^{1/64}+4 \eta]n^2\leq 6\eta^{1/64}n^2.$$
Since $G[X_1]$ is complete and contains at most $6\eta^{1/64}n^2$ non-blue edges, there are at most~$3\eta^{1/128}n$ vertices with blue degree at most $|X_1|-1-4\eta^{1/128}n$. Removing these vertices from $X_1$, that is, re-assigning these vertices to~$W$, gives a new~$X_1$  
such that every vertex in $G[X_1]$ has blue degree at least $|X_1|-1-4\eta^{1/128}n$. The same argument works for $G[Y_1]$, thus completing the proof of (\ref{HN5Z}).

Next, consider $G[X_1,X_2]$. Considering (\ref{HN2}), in a similar way to the above, we can bound the number of non-red edges in $G[X_1,X_2]$ by 

$$4\eta^2 pk \left(\frac{N}{K}\right)^2+2\eta^{1/64}pq\left(\frac{N}{K}\right)^{2}+2\eta pq\left(\frac{N}{K}\right)^{2}. $$

Where the first term bounds the number of non-red edges between non-regular pairs, the second bounds the number of non-red edges between pairs of clusters that are joined by a blue edge in the  reduced-graph and the second bounds the number of non-red edges between pairs of clusters that are not joined by a blue edge in the  reduced-graph.

Since $K\geq 2k$, $N\leq 4n$, $p \leq \aII k \leq k$ and $q\leq \half\aI k\leq\tfrac{1}{2}k$, we obtain
$$e(G_2[X_1,X_2])+e(G_3[X_1,X_2])\leq 6 \eta^{1/64}n^2.$$

Since $G[X_1,X_2]$ is complete and contains at most $6\eta^{1/64}n^2$ non-red edges, there are at most $2\eta^{1/128}n$ vertices in~$X_1$ with red degree to~$X_2$ at most $|X_2|-3\eta^{1/128}n$ and at most~$2\eta^{1/128}n$ vertices in~$X_2$ with red degree to~$X_1$ at most $|X_1|-3\eta^{1/128}n$. Removing these vertices results in every vertex in~$X_1$ having degree in $G_1[X_1,X_2]$ at least $|X_2|-2\eta^{1/128}n$ and every vertex in~$X_2$ having degree in $G_1[X_1,X_2]$ at least $|X_1|-2\eta^{1/128}n$.

We repeat the above for~$G[Y_1,Y_2]$, removing vertices such that every (remaining) vertex in~$Y_1$ has degree in $G_1[Y_1,Y_2]$ at least $|Y_2|-3\eta^{1/128}n$ and every (remaining) vertex in~$Y_2$ has degree in $G_1[Y_1,Y_2]$ at least $|Y_1|-3\eta^{1/128}n$, thus completing the proof of (\ref{HN5}).
\end{proof}

Having discarded some vertices, recalling~(\ref{IB1}), we  have
\begin{align}
\label{IB2}
|X_1|,|Y_1| & \geq (\aII-6\eta^{1/128})n,
& |X_2|,|Y_2| & \geq (\half \aI-3\eta^{1/128})n,
\end{align}
 and can proceed to the {end-game}. 
  
 The following pair of claims allow us to determine the colouring of $G[X_1,Y_1]$:
\begin{claim}
\label{nogreen2}
\hspace{-2.8mm} {\rm \bf a.} If there exist distinct vertices $x_1,x_2 \in X_1$ and $y_1, y_2\in Y_1$ such~that~$x_1y_1$ and~$x_2y_2$ are coloured blue, then~$G$ contains a blue cycle of length exactly~$\langle \aII n \rangle$.

{\rm \bf Claim~\ref{nogreen2}.b.} If there exist distinct vertices $x_1,x_2 \in X_1$ and $y_1, y_2\in Y_1$ such~that~$x_1y_1$ and~$x_2y_2$ are coloured red, then~$G$ contains a red cycle of length exactly~$\llangle \aI n \rrangle$.
\end{claim}

\begin{proof}
(a) Suppose there exist distinct vertices $x_1,x_2\in X_1$ and  $y_1,y_2\in Y_1$ such that the edges $x_1y_1$ and $x_2y_2$ are coloured blue. Then, let $\widetilde{X}_1$ be any set of $\half (\langle \aII n \rangle+1) $ vertices in~$X_1$ such that $x_1,x_2 \in \widetilde{X}_1$. 

By (\ref{HN5Z}), 
every vertex in $\widetilde{X}_1$ has degree at least $|\widetilde{X}_1|-1-4\eta^{1/128}n$ in $G_1[\widetilde{X}_1]$. Since $\eta\leq(\aII/100)^{128}$, we have $|\widetilde{X}_1|-1-4\eta^{1/128}n \geq \half |\widetilde{X}_1| +2$. So, by Corollary~\ref{dirac2}, there exists a Hamiltonian path in $G_1[\widetilde{X}_1]$ between $x_1, x_2$, that is, there exists a blue path between~$x_1$ and $x_2$ in $G[X_1]$ on exactly $\half (\langle \aII n \rangle+1)$ vertices. 

Likewise, given any two vertices $y_1,y_2$ in~$Y_1$, there exists a blue path between~$y_1$ and~$y_2$ in~$G[Y_1]$ on exactly $\half (\langle \aII n \rangle-1)$ vertices.  
Combining the edges $x_1y_1$ and $x_2y_2$ with the blue paths gives a blue cycle on exactly $\langle \aII n \rangle$ vertices.

(b) Suppose there exist distinct vertices $x_1,x_2\in X_1$ and  $y_1,y_2\in Y_1$ such that $x_1y_1$ and $x_2y_2$ are coloured red. Then, let $\widetilde{X}_2$ be any set of 
$$\ell_1=\left\lfloor \frac{\llangle \aI n \rrangle-2}{4} \right\rfloor \geq 3\eta^{1/128} n +2 $$ vertices from $X_2$. By~(\ref{HN5}), $x_1$ and $x_2$ each have at least two neighbours in $\widetilde{X}_2$ and, since $\eta\leq(\aII/100)^{128}$, every vertex in $\widetilde{X}_2$ has degree at least 
$|X_1|-3\eta^{1/128}n\geq \half|X_1| +\half|\widetilde{X}_2| +1$ in $G[X_1,\widetilde{X}_2]$. Since $|X_1|>\ell_1+1$, by Lemma~\ref{bp-dir},  $G_1[X_1,\widetilde{X}_2]$ contains a path on exactly $2\ell_1+1$ vertices from~$x_1$ to~$x_2$.

Likewise, given $y_1, y_2 \in Y_1$, for  any set $\widetilde{Y}_2$ of 
$$\ell_2=\left\lceil \frac{\llangle \aII n \rrangle-2}{4} \right\rceil \geq 3\eta^{1/128} n +2 $$ vertices from $Y_2$, $G_1[Y_1,\widetilde{Y}_2]$ contains a a path on exactly $2\ell_2+1$ vertices from~$y_1$ to~$y_2$.

Combining the edges~$x_1y_1$,~$x_2y_2$ with the red paths found gives a red cycle on exactly $2\ell_1+2\ell_2+2=\llangle \aI n \rrangle$ vertices, completing the proof of the claim. 
\end{proof}

The existence of a red cycle on $\aIna$ vertices or a blue cycle on $\langle \aII n\rangle$ vertices would be sufficient to complete the proof of Theorem~C. Thus, there cannot exist such a pair of vertex-disjoint red edges or such a pair of vertex-disjoint blue edges in $G[X_1,Y_1]$.  Thus, after removing at most one vertex from each of~$X_1$ and~$Y_1$, we may assume that the green graph $G_3[X_1,Y_1]$ is complete. 
%

Then, recalling~(\ref{IB2}), we have 
\begin{align}
\label{IB3}
|X_1|, |Y_1|&\geq (\aII-8\eta^{1/128})n,
& |X_2|, |Y_2|&\geq (\half \aI - 4\eta^{1/128})n.
\end{align}

We now consider $Z$. We claim that there can exist no vertex in $Z$ having a green edge to both $X_1$ and $Y_1$. Indeed, suppose that there existed such a vertex $z$ and vertices $x\in X_1$ and $y\in Y_1$ such that $xz$ and $yz$ are both coloured green. We know that $G_3[X_1,Y_1]$ is complete and that $|X_1|,|Y_1|\geq(\aII-8\eta^{1/128})n\geq(\aIII-9\eta^{1/128})n$. Thus, we can greedily construct a path on $\langle \aIII n \rangle -1$ vertices in $G_3[X_1,Y_1]$ from $x$ to $y$ which, together with the edges $xz$ and $yz$ gives a green cycle of length exactly~$\langle \aIII n \rangle$. 

\begin{figure}[!h]
\centering
\vspace{-1mm}\hspace{2mm}
{\includegraphics[height=50mm, page=9]{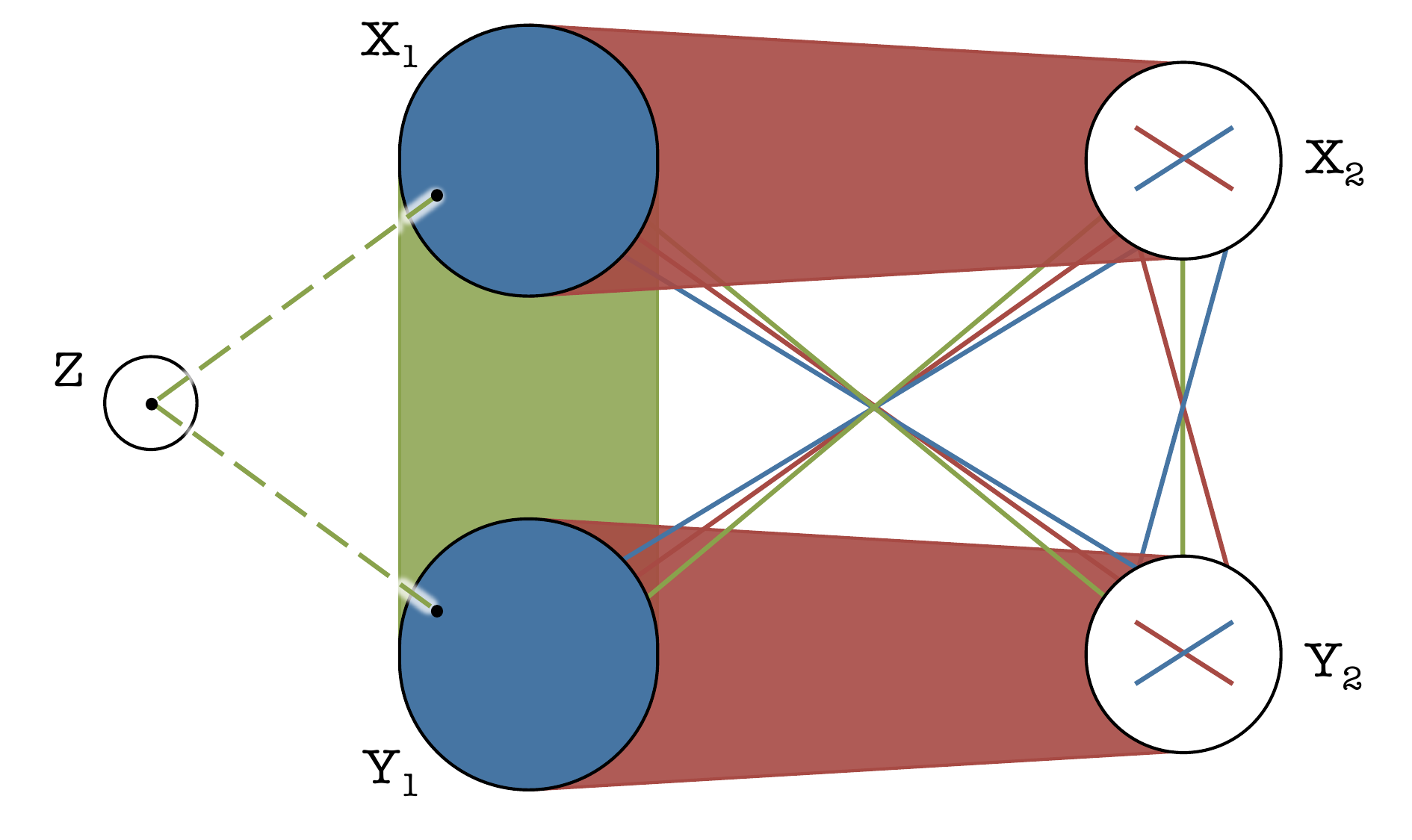}\hspace{18mm}}
\vspace{-1mm}\caption{Using $w\in W_G$ to construct an odd green cycle. }
  
\end{figure}

Thus, defining $Z_X$ to be the set of vertices in $Z$ having no green edges to $X_1\cup X_2$ and $Z_Y$ to be the set of vertices in $Z$ having no green edges to $Y_1\cup Y_2$, we see that, we may assume that $Z_X\cup Z_Y$ is a partition of $Z$.
Then, recalling that $|V(G)|\geq \llangle \aI n\rrangle+2\langle \aII n \rangle-3$, we may assume, without loss of generality that 
$$|X_1 \cup X_2 \cup Z_X|\geq \half\llangle \aI n \rrangle +  \langle \aII n \rangle -1.$$ 

\begin{figure}[!h]
\centering
\includegraphics[height=22mm, page=6]{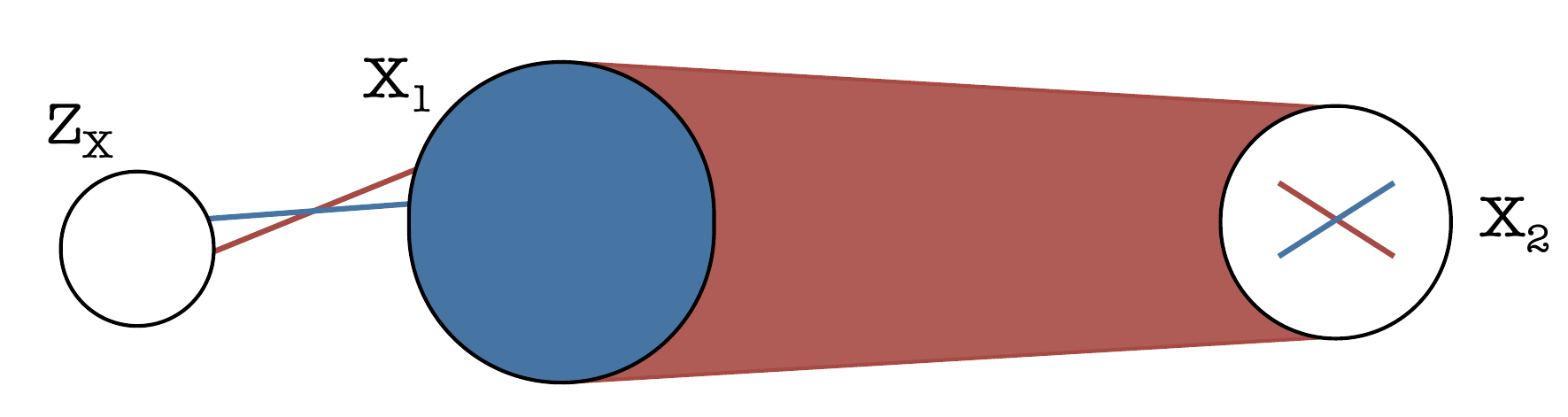}
\vspace{0mm}\caption{Colouring of $G[X_1\cup X_2\cup Z_X]$.}
  
\end{figure}

Given (\ref{HN5Z}) and (\ref{HN5}), we can obtain upper bounds on $|X_1|$, $|X_2|$, $|Y_1|$ and $|Y_2|$ as follows: By Corollary~\ref{dirac1a}, for every integer $m$ such that $8\eta^{1/128}n+2\leq m\leq |X_1|$, we know that $G_2[X_1]$ contains a blue cycle of length $m$. Thus, in order to avoid having a blue cycle of length $\langle \aII n \rangle$, we may assume that $|X_1|<\langle \aII n \rangle$. 
By Corollary~\ref{moonmoser2}, for every even integer $m$ such \mbox{that $12\eta^{1/128}n+2\leq m \leq 2\min\{|X_1|,|X_2|\}$}, we know that $G_1[X_1,X_2]$ contains a red cycle of length $m$. Recalling (\ref{a2big}) and (\ref{IB3}), we have $|X_1|\geq(\aII-8\eta^{1/128})n\geq\half\aI n$. Thus, in order to avoid having a red cycle on exactly $\aIna$ vertices, we may assume that $|X_2|<\half\aIna$.
In summary, we~have
\begin{equation}
\label{IB7}
\left.
\begin{aligned}
\,\,\,\,\quad\quad\quad\quad\quad\quad\quad\quad\quad\,\,\,\,
(\aII-8\eta^{1/128})n&\leq |X_1|< \langle \aII n \rangle, 
\quad\quad\quad\quad\quad\quad\quad\quad\quad\\
(\half \aI - 4\eta^{1/128})n&\leq |X_2|< \half\aIna.
\end{aligned}
\right\}\!
\end{equation}

We now let
\begin{align*}
Z_R&=\{ z\in Z_X \text{ such that } z \text{ has at least } |X_1|-16\eta^{1/128} \text{ red neighbours in } X_1\};\text{ and}
\\ Z_B&=Z\backslash Z_R=\{ z\in Z_X \text{ such that } z \text{ has at least } 16\eta^{1/128} \text{ blue neighbours in } X_1\}.
\end{align*}

\begin{figure}[!h]
\centering
\includegraphics[height=22mm, page=7]{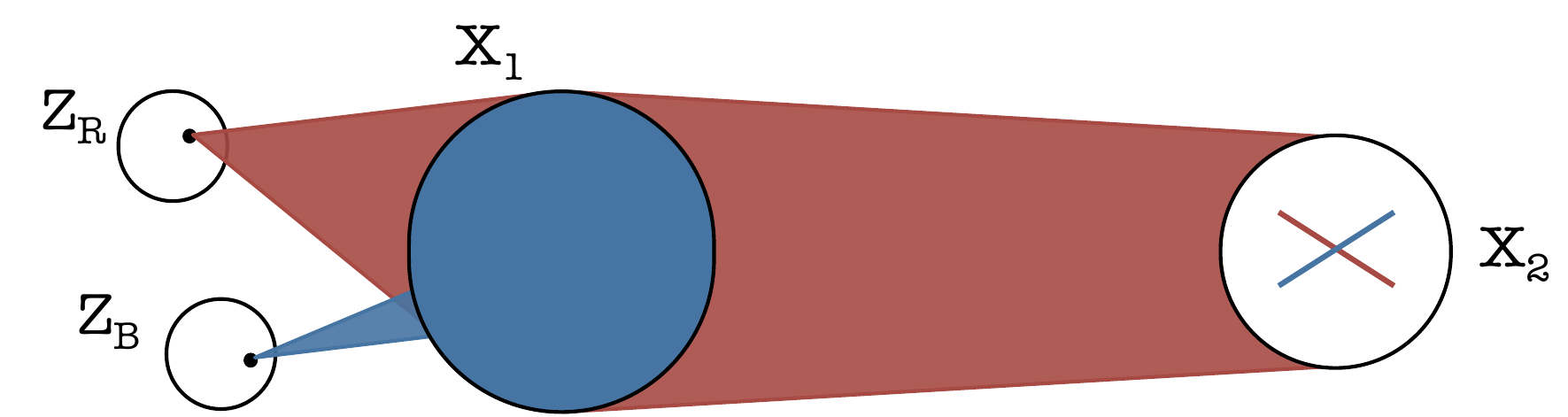}
\vspace{0mm}\caption{Partition of $Z_X$ into $Z_R\cup Z_B$.}
\end{figure}

In which case, we have $Z=Z_B\cup Z_R$ and, so, either $|Z_1\cup Z_B|\geq\langle \aII n \rangle$ or $|Z_2\cup Z_R|\geq\half\llangle \aI n \rrangle$.

If $|X_1\cup Z_B|\geq\langle \aII n \rangle$, then we show that $G_2[X_1\cup Z_B]$ contains a long blue cycle as follows: Let~$X$ be any set of~$\langle \aII n \rangle$ vertices from $X_1 \cup Z_B$ consisting of every vertex from~$X_1$ and~$\langle \aII n \rangle-|X_1|$ vertices from~$Z_B$. By (\ref{HN5Z}) and~(\ref{IB7}), the blue graph~$G_2[X]$ has at least $\langle \aII n \rangle-8\eta^{1/128}n$ vertices of degree at least $|X|-1-4\eta^{1/128}n$ and at most $8\eta^{1/128}n$ vertices of degree at least $16\eta^{1/128}n$. Thus, by Theorem~\ref{chv},~$G[X]$ contains a blue cycle on exactly~$\langle \aII n \rangle$ vertices.

Thus, we may, instead, assume that $|X_2\cup Z_R|\geq\half\llangle \aI n \rrangle$, in which case, we consider the red graph $G_2[X_1,X_2\cup Z_R]$. Given the relative sizes of~$X_1$ and~$X_2\cup Z_R$ and the large minimum-degree of the graph, we can use Theorem~\ref{moonmoser} to give a red cycle on exactly $\llangle \aI n \rrangle$ vertices as follows: By (\ref{a2big}) and~(\ref{IB7}), we have $|X_1|\geq\half\llangle\aI n \rrangle$ and may choose subsets $\widetilde{X}_1\subseteq X_1$, $\widetilde{X}_2\subseteq X_2\cup Z_R$ such that $\wX_2$ includes every vertex of $X_2$, $
|\widetilde{X}_1|=|\widetilde{X}_2|=\half\llangle\aI n \rrangle$ and $|\wX_2\cap Z_R|\leq6\eta^{1/128}n$. Recall, from (\ref{HN5}), that $G_1[X_1,X_2]$ is $3\eta^{1/128}k$-almost-complete and that, by definition, all vertices in $Z_R$ have red degree at least $|\widetilde{X}_1|-1-16\eta^{1/128}n$ in $G[\widetilde{X}_1,\widetilde{X}_2]$. Thus, since $|\wX_2 \cap Z_R|\leq 6\eta^{1/128}n$, for any pair of vertices $x_1\in\widetilde{X}_1$ and $x_2\in\widetilde{X}_2$, we have $d(x_1)+d(x_2)\geq\half\llangle\aI n \rrangle+1$. Therefore, by Theorem~\ref{moonmoser}, $G_1[\wX_1,\wX_2]$ contains a red cycle on exactly~$\llangle \aI n \rrangle$ vertices, thus completing Part I.C. and thus, also, Part I.

\section{Proof of the main result -- Part II -- Case (v)}
\label{s:p12}

Suppose $\cG$ contains 
disjoint subsets of vertices $\cX$ and $\cY$ such that $G[\cX]$ contains a two-coloured spanning subgraph~$H$ from $\cH_{2}^*\cup\cJ_b$ and $G[\cY]$ contains a two-coloured spanning subgraph $K$ from~\mbox{$\cH_{2}^*\cup\cJ_b$}, where
\vspace{-2mm}
\begin{align*}
\cH_{2}^*&=\cH\left((\beta-2\eta^{1/32})k,(\half\aI-2\eta^{1/32})k,4\eta^4 k,\eta^{1/32},\gamma,\text{red}\right),\text{ } 
\\ \cJ_b&=\cJ\left((\aI-18\eta^{1/2}), 4\eta^4 k, \text{red}, \gamma\right),
\end{align*}
and $\beta=\max\{\aII,\aIII\}$ and $\gamma\in\{\text{blue}, \text{green}\}$.

Additionally, from Theorem D, we know that $\aI\leq\beta$. We divide the proof that follows into three sub-parts depending on whether $H$ and $K$ belong to $\cH_2^*$ or~$\cJ_b$:

\subsection*{Part II.A: $H, K \in \cH_2^*$.}

In this case, recalling Theorem~\hyperlink{thD}{D}, we know that $\aI\leq(1-\eta^{1/16})\beta$.

We consider four subcases:

{\it Subcase i:  $\beta=\aII$, $\gamma=\text{blue}$.} 

Since $0<\eta<1$, we have $\cH_{2}^*\subseteq\cH_2$. Notice also that, we have $\aI\leq(1-\eta^{1/16})\aII\leq\aII$. Thus, we have $\aII\geq(1-\eta^{1/16})\aI$ and, in fact, this case has already been dealt with in Part I.C.
Note that, in Part I.C, we knew that all blue components were odd and that here we do not, however, we did use that information.

{\it Subcase ii:  $\beta=\aIII$, $\gamma=\text{blue}$.} 

The vertex set~$\cV$ of~$\cG$ has a natural partition into $\cX _1 \cup \cX _2 \cup \cY _1 \cup \cY _2\cup \cZ$, where $\cX_1\cup\cX_2$ is the partition of the vertices of $H$ given by Definition~\ref{d:H} and  $\cY_1\cup\cY_2$ the corresponding partition of the vertices of $K$. In particular, $|\cX_1|\geq (\aIII-2\eta^{1/32})n$.
By the definition of $\cH_2^*$, we know that  $\cG_2[\cX_1]$ is $(1-\eta^{1/32})$-complete and so, by Theorem~\ref{dirac}, it contains a blue odd connected-matching on at least $|\cX_1|-1\geq(\aIII-3\eta^{1/32})k$ vertices. Thus, in order to avoid having an odd blue connected-matching on at least $\aII k$ vertices, we may assume that $\aII\geq\aIII-3\eta^{1/32}$ and, therefore, $\aII\geq(\aI-3\eta^{1/32})$. We may then follow the same argument given in Part I.C to obtain a monochromatic cycle and complete the proof.

{\it Subcase iii:  $\beta=\aIII$, $\gamma=\text{green}$.} 

We have $\aIII\geq(1-\eta^{1/16})\aI$ and can follow the argument given in Part I.C with the roles of blue and green (and $\aII$ and $\aIII$) exchanged.

{\it Subcase iv:  $\beta=\aII$, $\gamma=\text{green}$.} 

The vertex set~$\cV$ of~$\cG$ has a natural partition into $\cX _1 \cup \cX _2 \cup \cY _1 \cup \cY _2\cup \cZ$, where $\cX_1\cup\cX_2$ is the partition of the vertices of $H$ given by Definition~\ref{d:H} and  $\cY_1\cup\cY_2$ the corresponding partition of the vertices of $K$. In particular, $|\cX_1|\geq (\aII-2\eta^{1/32})n$.
By the definition of $\cH_2^*$, we know that  $\cG_3[\cX_1]$ is $(1-\eta^{1/32})$-complete and so, by Theorem~\ref{dirac}, it contains a green odd connected-matching on at least $|\cX_1|-1\geq(\aII-3\eta^{1/32})k$ vertices. Thus, in order to avoid having an odd green connected-matching on at least $\aIII k$ vertices, we may assume that $\aIII\geq\aII-3\eta^{1/32}$ and, therefore, $\aIII\geq(\aI-3\eta^{1/32})$. We may then follow the same argument given in Part I.C with the roles of blue and green (and $\aII$ and $\aIII$) exchanged to obtain a monochromatic cycle and complete the proof.

\subsection*{Part II.B: $H, \in \cH_2^*, K\in\cJ_b$.}

In this case, recalling Theorem~\hyperlink{thD}{D}, we know that
\begin{equation}
\label{IIBSIZE}
\left.
\begin{aligned}
\,\,\,\,\quad\quad\quad\quad\quad\quad\quad\quad\quad\,\,\,\,
\aI&\leq(1-\eta^{1/16})\beta, 
\quad\quad\quad\quad\quad\quad\quad\quad\quad\quad\quad\quad\quad\\
\beta&<(\tfrac{3}{2}+2\eta^{1/4})\aI.
\end{aligned}
\right\}\!\!\!\!\!\!\!\!\!\!\!\!\!\!\!\!\!\!\!
\end{equation}
Suppose that $\gamma$ is green. The vertex set~$\cV$ of~$\cG$ has a natural partition into $\cX _1 \cup \cX _2 \cup \cY _1 \cup \cY _2\cup \cZ$ where $\cX_1\cup\cX_2$ is the partition of the vertices of $H$ given by Definition~\ref{d:H} and  $\cY_1\cup\cY_2$ the corresponding partition of the vertices of $K$ given by Definition~\ref{d:J}. 
Moving vertices from $\cX_1\cup\cX_2\cup\cY_1\cup\cY_2$ to $\cZ$, we may assume that 
\begin{equation}
\left.
\label{IIBS1}
\begin{aligned}
\quad\quad\quad\quad\quad\quad\quad\quad\,\,\,
(\beta-2\eta^{1/32})k & \leq |X_1| =p \leq \beta k,
\quad\quad\quad\quad\quad\quad\quad\quad\quad\quad\,\,\,\\
(\half\aI-2\eta^{1/32})k & \leq |X_2| =q \leq \half\aI k, \\
 (\aI-18\eta^{1/2}) & \leq |Y_{1}|=|Y_{2}| = r \leq \aI k
 \end{aligned}
\right\}\!\!\!\!\!\!\!\!\!\!\!
\end{equation}
and that
\begin{itemize}
\labitem{JB1}{JB1} $H$ and $K$ are each $4\eta^4 k$-almost-complete; and
\labitem{JB2}{JB2} {~}\!\!\!\!(a) 
$\cG[\cX_{1}]$ is $\eta^{1/32}$-sparse in red, contains no blue edges 
and is $(1-\eta^{1/32})$-complete in green,
\item[{~}] {~\,}(b) 
$\cG[\cX_1,\cX_2]$ is $(1-\eta^{1/32})$-complete in red, contains no blue edges 
and is $\eta^{1/32}$-sparse in green,
\item[{~}] {~\,}(c) 
all edges present in $\cG[\cY_1], \cG[\cY_2]$, are coloured exclusively red,  
\item[{~}] {~\,}(d) all edges present in $\cG[\cY_1,\cY_2]$,  are coloured exclusively green. 
\end{itemize}

By the definition of $\cH_2^*$, we know that  $\cG_3[\cX_1]$ is $(1-\eta^{1/32})$-complete and so, by Theorem~\ref{dirac}, it contains a green connected-matching on at least $|\cX_1|-1\geq(\beta-4\eta^{1/32})k$ vertices and also clearly contains an odd green cycle. By definition of $\cJ_b$, $\cG_3[Y_1,Y_3]$ is $4\eta^4 k$-almost-complete and so, by Lemma~\ref{l:eleven}, it contains a green connected-matching on at least $(2\aI-40\eta^{1/32})k$ vertices. Thus, there can be no green edges present in $\cG[X_1,Y_1\cup Y_2]$.

Again, by the definition of $\cH_2^*$, we know that $\cG_1[\cX_1,\cX_2]$ is $(1-\eta^{1/32})$-complete and so, by Lemma~\ref{l:ten}, it contains a red connected-matching on at least $(\aII-8\eta^{1/64})k$ vertices. 
By definition of $\cJ_b$, $\cG_1[Y_1]$ and $\cG_1[Y_2]$ are each $4\eta^4 k$-almost-complete and so, by Theorem~\ref{dirac}, each contains a red connected-matching on at least $(\aI-24\eta^{1/2})k$ vertices. Thus, there can be no red edges present in $G[X_1\cup X_2,Y_1\cup Y_2]$.

Thus, we know that 
\begin{itemize}
\labitem{JB3}{JB3}
all edges present in $\cG[X_1,Y_1\cup Y_2]$ are coloured exclusively blue.
\end{itemize}

\begin{figure}[!h]
\centering
\includegraphics[width=64mm, page=34]{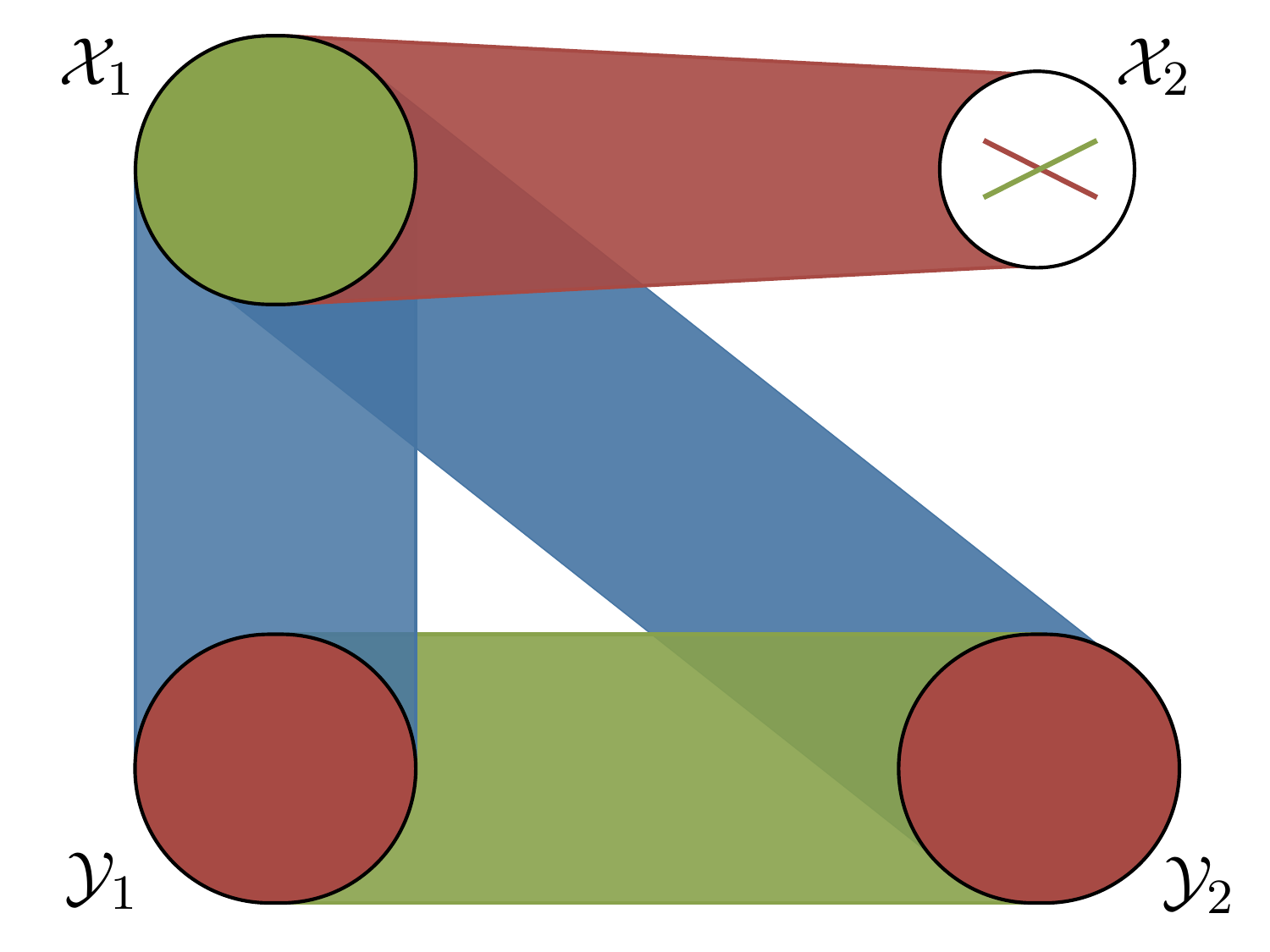}
\vspace{0mm}\caption{Coloured structure of the  reduced-graph in Part II.B.}  
\end{figure}

The remainder of this section focuses on showing that the original graph must have a similar structure, which can then be exploited in order to force a cycle of appropriate length, colour and parity.

As before, each vertex $V_{i}$ of $\cG=(\cV,\cE)$ represents a cluster of vertices of $G=(V,E)$. Additionally, recall, from (\ref{NK}), that
$$(1-\eta^4)\frac{N}{K}\leq |V_i|\leq \frac{N}{K}.$$ 

Since $n> \max\{n_{\ref{th:blow-up}}(4,0,0,\eta), n_{\ref{th:blow-up}}(1,2,0,\eta), n_{\ref{th:blow-up}}(1,0,2,\eta)\}$, we can (as in the proof of Theorem~\ref{th:blow-up}) prove that $$
|V_i|\geq \left(1+\frac{\eta}{24}\right)\frac{n}{k}> \frac{n}{k}.$$ 

Thus, we can partition the vertices of~$G$ into sets $X_{1}, X_{2}, Y_{1}, Y_{2}$ and $Z$  corresponding to the partition of the vertices of~$\cG$ into $\cX_1, \cX_2, \cY_1,\cY_2$ and $\cZ$. Then,~$X_1$ contains $p$ clusters, $X_2$ contains $q$ clusters and $Y_1$ and $Y_2$ each contain~$r$ clusters of vertices and we have
\begin{equation}
\left.
\label{IIB1}
\begin{aligned}
\quad\quad\quad\quad\quad\quad\quad\quad\,\,\,
|X_1|& = p|V_1| \geq (\beta-2\eta^{1/32})n,
\quad\quad\quad\quad\quad\quad\quad\quad\quad\quad\,\,\,\\
|X_2|& = 1|V_1| \geq (\half\aI-2\eta^{1/32})n,\\
|Y_1|=|Y_2| & = r|V_1| \geq (\aI-18\eta^{1/2})n.
 \end{aligned}
\right\}\!\!\!\!\!\!\!\!\!\!\!
\end{equation}

In what follows, we will \textit{remove} vertices from $X_1\cup X_2\cup Y_1\cup Y_2$ by moving them into~$Z$ in order to show that, in what remains, $G[X_1\cup X_2\cup Y_1 \cup Y_2]$ has a particular coloured structure.  We begin by proving the below claim which essentially tells us that  $G$ has similar coloured structure to $\cG$:

\begin{claim}
\label{G-structIIB}
We can \textit{remove} at most 
$9\eta^{1/2}+4\eta^{1/64}n$ vertices from~$X_1$,
at most 
$3\eta^{1/2}n$ vertices from~$X_2$, and
at most 
$16\eta^{1/2}n$ vertices from each of~$Y_1$ and~$Y_2$ such that
the following conditions~hold.
 \end{claim}
\begin{itemize}
\labitem{JB4}{JB4} $G_1[X_1,X_2]$ is $2\eta^{1/2}$-almost-complete and  $G_1[Y_1]$ and  $G_1[Y_2]$ are each $6\eta^{1/2}$-almost-complete; 
\labitem{JB5}{JB5} $G_2[X_1,Y_1\cup Y_2]$ is $9\eta^{1/2}$-almost-complete; and
\labitem{JB6}{JB6} $G_3[X_1]$ is $4\eta^{1/64}$-almost-complete and $G_3[Y_1,Y_2]$ is $6\eta^{1/2}$-almost-complete.
\end{itemize}
\begin{proof}
The proof follows the same pattern as that of Claim~\ref{G-struct} but is made easier in parts by the colouring of the reduced-graph.

Consider $G[X_1,X_2]$. Recall from (\ref{JB1}) that $\cG[X_1,X_2]$ is $4\eta^4 k$-almost-complete and from (\ref{JB2}), that $\cG[X_1,X_2]$ is $(1-\eta^{1/32})$-complete in red, contains no blue edges and is $\eta^{1/32}$ sparse in green. We can bound the number of non-red edges in $G[X_1,X_2]$ by 
$$4\eta^4 pk \left(\frac{N}{K}\right)^2+2\eta^{1/32}pq\left(\frac{N}{K}\right)^{2}+2\eta pq\left(\frac{N}{K}\right)^{2},$$

where the first term bounds the number of non-red edges between non-regular pairs, the second bounds the number of non-red edges between pairs of clusters that are joined by a green edge in the  reduced-graph and the second bounds the number of non-red edges between pairs of clusters that are not joined by a green edge in the  reduced-graph.

Since $K\geq 2k$, $N\leq 4n$, $p \leq \beta k \leq k$ and $q\leq \half\aI k\leq\tfrac{1}{2}k$, we obtain
$$e(G_2[X_1,X_2])+e(G_3[X_1,X_2])\leq 6 \eta^{1/32}n^2.$$

Since $G[X_1,X_2]$ is complete and contains at most $6\eta^{1/32}n^2$ non-red edges, there are at most $3\eta^{1/64}n$ vertices in~$X_1$ with red degree to~$X_2$ at most $|X_2|-2\eta^{1/64}n$ and at most~$3\eta^{1/64}n$ vertices in~$X_2$ with red degree to~$X_1$ at most $|X_1|-2\eta^{1/64}n$. Removing these vertices results in every vertex in~$X_1$ having degree in $G_1[X_1,X_2]$ at least $|X_2|-2\eta^{1/64}n$ and every vertex in~$X_2$ having degree in $G_1[X_1,X_2]$ at least $|X_1|-2\eta^{1/64}n$.

Next, consider the complete three-coloured graph $G[Y_1]$. Recall from (\ref{JB1})and (\ref{JB2}) that there can be no blue or green edges present in $\cG[\cY_1]$ and that $\cG[\cY_1]$ is $4\eta^4k$-almost-complete. Given the structure of~$\cG$, we can bound the number of non-red edges in $G[Y_1]$ by
$$r \binom{N/K}{2} + 4\eta^4 rk \left(\frac{N}{K}\right)^2+2\eta\binom{r}{2}\left(\frac{N}{K}\right)^{2},$$ where the first term counts the number of non-red edges within the clusters, the second counts the number of non-red edges between non-regular pairs of clusters and the third counts the number of non-red edges between regular pairs.

Since $K\geq 2k, \eta^{-1}$, $N\leq 4n$ and $r\leq \aI k \leq k$, we obtain
$$e(G_2[Y_1])+e(G_3[Y_1])\leq [4\eta +16\eta^{4} +16 \eta]n^2\leq 32\eta n^2.$$

Since $G[Y_1]$ is complete and contains at most $32\eta n^2$ non-red edges, there are at most~$8\eta^{1/2}n$ vertices with red degree at most $|Y_1|-1-8\eta^{1/2}n$. Removing these vertices from $Y_1$, that is, re-assigning these vertices to~$Z$, gives a new~$Y_1$  
such that every vertex in $G[Y_1]$ has red degree at least $|Y_1|-1-8\eta^{1/2}n$. The same argument works for $G[Y_2]$, completing the proof of~(\ref{JB4}).

Next, we consider $G[X_1,Y_1\cup Y_2]$. By (\ref{JB2}) all the edges present in $\cG[X_1,Y_1\cup Y_2]$ are coloured exclusively blue. Thus, we can bound the number of non-blue edges in $G[X_1,Y_1\cup Y_2]$ by
$$4\eta^4 rk \left(\frac{N}{K}\right)^2+2\eta p(2r) \left(\frac{N}{K}\right)^{2}. $$

Where the first term bounds the number of non-blue edges between non-regular pairs, the second bounds the number of non-blue edges between regular pairs.

Since $K\geq 2k, \eta^{-1}$, $N\leq 4n$ and $p\leq \beta\leq k$, $r\leq \aI k \leq k$, we obtain
$$e(G_1[X_1,Y_1\cup Y_2])+e(G_3[X_1,Y_1\cup Y_2])\leq 36\eta n^2.$$

Since $G[X_1,Y_1\cup Y_2]$ is complete and contains at most $36\eta n^2$ non-blue edges, there are at most $6\eta^{1/2}n$ vertices in~$X_1$ with green degree to~$Y_1\cup Y_2$ at most $|Y_1+Y_2|-6\eta^{1/2}n$ and at most~$6\eta^{1/2}n$ vertices in~$Y_1\cup Y_2$ with blue degree to~$X_1$ at most $|X_1|-6\eta^{1/2}n$ Removing these vertices results in every vertex in~$X_1$ having degree in $G_2[X_1,Y_1\cup Y_2]$ at least $|Y_1\cup Y_2|-6\eta^{1/2}n$ and every vertex in~$Y_1\cup Y_2$ having degree in $G_2[X_1,Y_1\cup Y_2]$ at least $|X_1|-6\eta^{1/2}n$k, thus completing the proof of (\ref{JB5}).

Next, consider the complete three-coloured graph $G[X_1]$ and recall, from (\ref{JB2}), that $\cG[\cX_1]$ contains only red and green edges and is $4\eta^4 k$-almost-complete. Thus, we can bound the number of non-green edges in $G[X_1]$ by
$$p \binom{N/K}{2} + 4\eta^4 pk \left(\frac{N}{K}\right)^2+2\eta^{1/32}\binom{p}{2}\left(\frac{N}{K}\right)^{2}+2\eta\binom{p}{2}\left(\frac{N}{K}\right)^{2},$$
where the first term counts the number of non-green edges within the clusters,
the second counts the number of non-green edges between non-regular pairs,
the third counts the number of non-green edges between regular pairs which are joined by a red edge and 
the final term counts the number of non-green edges between regular pairs which are not joined by a red edge.

Since $K\geq 2k, \eta^{-1}$, $N\leq 4n$ and $p\leq \aIII k \leq k$, we obtain
$$e(G_1[X_1])+e(G_3[X_1])\leq 8 \eta^{1/32}n^2.$$

Since $G[X_1]$ is complete and contains at most $8\eta^{1/32}n^2$ non-green edges, there are at most~$4\eta^{1/64}n$ vertices with green degree at most $|X_1|-1-4\eta^{1/64}n$. Removing these vertices from $X_1$ gives a new~$X_1$  
such that every vertex in $G[X_1]$ has green degree at least $|X_1|-1-4\eta^{1/64}n$.

Finally, we consider $G[Y_1,Y_2]$, where we can bound the number of non-green edges in $G[Y_1,Y_2]$ by 
$$4\eta^4 rk \left(\frac{N}{K}\right)^2+2\eta r^2\left(\frac{N}{K}\right)^{2}. $$

Where the first term bounds the number of non-green edges between non-regular pairs and the second bounds the number of non-green edges between regular pairs.

Since $K\geq 2k, \eta^{-1}$, $N\leq 4n$ and $r\leq \aI k\leq k$, we obtain
$$e(G_1[Y_1,Y_2])+e(G_2[Y_1,Y_2])\leq 9\eta n^2.$$

Since $G[Y_1,Y_2]$ is complete and contains at most $9\eta n^2$ non-green edges, there are at most $3\eta^{1/2}n$ vertices in~$Y_1$ with green degree to~$Y_1$ at most $|Y_2|-3\eta^{1/2}n$ and at most~$3\eta^{1/2}n$ vertices in~$Y_2$ with blue degree to~$X_1$ at most $|Y_1|-3\eta^{1/2}n$ Removing these vertices results in every vertex in~$Y_1$ having degree in $G_3[Y_1,Y_2]$ at least $|Y_2|-3\eta^{1/2}n$ and every vertex in~$Y_2$ having degree in $G_3[Y_1,Y_2]$ at least $|Y_1|-3\eta^{1/2}n$k, thus completing the proof of (\ref{JB6}) and, also, of the claim.
\end{proof}

Having removed these vertices, recalling (\ref{IIB1}), we have

\begin{equation}
\left.
\label{IIB2}
\begin{aligned}
\quad\quad\quad\quad\quad\quad\quad\quad\,\,\,
|X_1|& = p|V_1| \geq (\beta-6\eta^{1/64})n,
\quad\quad\quad\quad\quad\quad\quad\quad\quad\quad\,\,\,\\
|X_2|& = q|V_1| \geq (\half\aI-4\eta^{1/32})n,\\
|Y_1|=|Y_2| & = r|V_1| \geq (\aI-34\eta^{1/2})n.
 \end{aligned}
\right\}\!\!\!\!\!\!\!\!\!\!\!
\end{equation}

We now consider the remaining vertices in $Z$ and begin by proving the following claim:
\begin{claim}
\label{claimIIBsplit1}
If there are vertices $z\in Z$, $x\in X_1$ and $y\in Y_1\cup Y_2$ such that both the edges $x  z$ and $y z$ are blue, then $G$ contains a blue cycle on exactly $\langle \aII n \rangle$ vertices.
\end{claim}

\begin{proof}
By (\ref{IIBSIZE}) and (\ref{IIB2}), we know that $|X_1|,|Y_1\cup Y_2|\geq (\aII-6\eta^{1/64})n\geq \lfloor \half \langle \aIII n\rangle \rfloor$. We let $\widetilde{X}$ consist of $(\aII-6\eta^{1/64})n$ vertices from $X_1$ including $x$, let $\widetilde{Y}$ consist of $\lfloor \half \langle \aII n\rangle \rfloor$ vertices from $Y$ including $y$ and consider $G[\widetilde{X},\widetilde{Y}]$. Every vertex in $\widetilde{Y}$ has degree in $G_2[\widetilde{X},\widetilde{Y}]$ at least $(\aII-8\eta^{1/64})n\geq
\half(|\widetilde{X}|+|\widetilde{Y}|+1).$
Thus, by Corollary~\ref{bp-dir2}, there exists a blue path from $x$ to $y$ on exactly $\langle \aII n \rangle -1$ vertices which together with $x z$ and $y z$ forms a blue cycle on exactly $\langle \aII n \rangle$ vertices.
\end{proof}

Thus, we may partition the vertices of $Z$ into $Z_X$ and $Z_Y$ where there are no blue edges present in $G[Z_X,X_1]\cup G[Z_Y,Y_1\cup Y_2]$.

Recalling that $|V(G)|\geq 4\llangle \aI n \rrangle-3, \llangle \aI n \rrangle + 2\langle \aIII n \rangle -3$, we know that 
$$|X_1|+|X_2|+|Y_1|+|Y_2|+|Z_X|+|Z_Y|\geq 2.5\llangle \aI n \rrangle+\langle \aIII n \rangle -3.$$
Thus, we know that either $|X_1|+|X_2|+|Z_X|\geq \half \llangle \aI n \rrangle+\langle \aIII n \rangle -1$ or $|Y_1|+|Y_2|+|Z_Y|\geq 2 \llangle \aI n \rrangle -1$.

In the former case, we define a partition of $Z_X$ into $Z_R\cup Z_G$ by
\begin{align*}
Z_R&=\{ z\in Z_X \text{ such that } z \text{ has at least } |X_1|-12\eta^{1/64} \text{ red neighbours in } X_1\};\text{ and}
\\ Z_G&=Z\backslash Z_R=\{ z\in Z_X \text{ such that } z \text{ has at least } 12\eta^{1/64} \text{ green neighbours in } X_1\}.
\end{align*}

Since this is a partition, we have either $|X_1\cup Z_G|\geq\langle \aIII n \rangle$ or $|X_2\cup Z_R|\geq\half\llangle \aI n \rrangle$.

If $|X_1\cup Z_G|\geq\langle \aIII n \rangle$, then, following the  argument given in the penultimate paragraph of Part I.C, we can show that $G_3[X_1\cup Z_G]$ contains a green cycle on exactly $\langle \aIII n \rangle$ vertices. If $|X_2\cup Z_R|\geq\half\llangle \aI n \rrangle$, given the relative sizes of~$X_1$ and~$X_2\cup Z_R$ and the large minimum-degree of the graph, we may follow the same argument as given in the final paragraph of Part I.C. to find that $G[X_1,X_2\cup Z_R]$ contains a a red cycle on exactly $\llangle \aI n \rrangle$ vertices.

 Thus, we may, instead, assume that
 $$|Y_1|+|Y_2|+|Z_Y|\geq 2 \llangle \aI n \rrangle -1.$$ In that case, we can prove the following claim:
 \begin{claim}
If there exist vertices $z\in Z_Y$, $y_1\in Y_1$ and $y_2\in Y_2$ such that both the edges $y_1 z$ and $y_2 z$ are green, then $G$ contains a green cycle on exactly $\langle \aII n \rangle$ vertices.
\end{claim}
\begin{proof}
Follows the same steps as that of Claim~\ref{claimIIBsplit1}.
\end{proof}
Similarly, we can show that the presence of a blue edge in $G[Y_1\cup Y_2]$ would result in a blue cycle on exactly $\langle \aII n \rangle$ vertices and that the  presence of a green edge in $G[Y_1]$ or $G[Y_2]$ would result in a green cycle on exactly $\langle \aIII n \rangle$ vertices.

Thus, we may partition $Z_Y$ into $Z_1$ and $Z_2$ where all edges in $G[Z_1,Y_1]$ and $G[Z_2,Y_2]$ (and indeed also in $G[Y_1]$ and $G[Y_2]$) are red. Since we have $|Y_1|+|Y_2|+|Z_Y|\geq2 \llangle \aI n \rrangle -1$, we may, without loss of generality, assume that $|Y_1|+|Z_1|\geq\llangle\aI n \rrangle$.

By (\ref{IIB2}), we have $|Y_1|\geq(\aI -34\eta^{1/2})n$. Thus, we may choose subsets $\widetilde{Y}_1\subseteq Y_1$ and $\widetilde{Z}_1\subseteq Z_1$ such that $\widetilde{Y}_1$ includes every vertex of $Y_1$ and $|\widetilde{Y}_1|+|\widetilde{Z}_1|=\llangle\aI n\rrangle$. Then, $|\widetilde{Z}_1|\leq 36\eta^{1/2} n$. Thus, every vertex in~$\widetilde{Y}_1\cup\widetilde{Z}_1$ has degree in $\G_1[\widetilde{X}_1\cup\widetilde{Z}_1]$ at least $|\widetilde{X}_1\cup\widetilde{Z}_1|-36\eta^{1/2}n\geq\half|\widetilde{X}_1\cup\widetilde{Z}_1|$. Therefore, by Theorem~\ref{dirac}, $\G_1[\widetilde{X}_1\cup\widetilde{Z}_1]$ contains a red cycle of length exactly $\llangle \aI n \rrangle$, thus completing Part II.B.

\subsection*{Part II.C: $H, K \in \cJ_b$.}
In this case, recalling Theorem~\hyperlink{thD}{D}, we know that
$$\beta=\max\{\aII,\aIII\}<(\tfrac{3}{2}+2\eta^{1/4})\aI.$$
Suppose $\gamma$ is green. The vertex set~$\cV$ of~$\cG$ has a partition into $\cX _1 \cup \cX _2 \cup \cY _1 \cup \cY _2\cup \cZ$, where $\cX_1\cup\cX_2$ is the partition of the vertices of $H$ given by Definition~\ref{d:J} and  $\cY_1\cup\cY_2$ is the corresponding partition of the vertices of $K$. We then know that
\begin{itemize}
\item[(JC0)] $|\cX_{1}|,|\cX_{2}|,|\cY_{1}|,|\cY_{2}|\geq (\aI-18\eta^{1/2})$;
\item[(JC1)] $H$ and $K$ are each $4\eta^4 k$-almost-complete; and
\item[(JC2)] (a) all edges present in $\cG[\cX_1], \cG[\cX_2], \cG[\cY_1], \cG[\cY_2]$ are coloured exclusively red,  
\item[\phantom{(JC2)}] (b) all edges present in $\cG[\cX_1,\cX_2], \cG[\cY_1,\cY_2]$  are coloured exclusively green. 
\end{itemize}

Because $H$ is $4\eta^4 k$-almost-complete, and all edges in $\cH[\cX_1]$ are coloured red, by Theorem~\ref{dirac}, $\cG[\cX_1]$ contains a red connected-matching on at least $|\cX_1|-1\geq(\aI-20\eta^{1/2})k$ vertices. Similarly, each of $\cG[\cX_2]$, $\cG[\cY_1]$ and $\cG[\cY_2]$ contains red connected-matchings on at least $(\aI-20\eta^{1/2})k$ vertices. Thus, the existence of a red edge in $\cG[\cX_1\cup\cX_2,\cY_1\cup\cY_2]$ would imply the existence of a red connected-matching on at least $\aI k$ vertices and, therefore, we may assume that there are no red edges present in~$\cG[\cX_1\cup\cX_2,\cY_1\cup\cY_2]$.

Because $H$ is
$4\eta^4 k$-almost-complete and all edges present in $\cG[\cX_1,\cX_2]$ are coloured green, by Lemma~\ref{l:ten}, $\cG_3[X_1,X_2]$ contains a green connected-matching $\cM$ on at least $(\aI-40\eta^{1/2})k\geq(\aIII+2\eta^{1/2})k$ vertices.

Suppose now that there exists a pair of green edges $x_1y_1^1$, $x_2y_1^2$ such that $x_1\in \cX_1, x_2\in \cX
_2, y_1^1, y_1^2\in \cY
_1$. Then, since $K$ is $4\eta^4 k$-almost-complete and all edges present in $\cG[\cY_1,\cY_2]$ are coloured green, $y_1^1$ and $y_1^2$ have a common neighbour $y_2\in\cY_2$. Similarly, there exists a green path $P=x_1 x_2' x_1' x_2$ in $G[\cX_1,\cX_2]$. Thus, $x_1y_1^1y_2y_1^2x_2x_1'x_2'x_1$ is an odd green cycle contained in the same component as $\cM$ and, thus, $\cG$ contains an odd green connected-matching on at least $\aIII k$ vertices.

Thus, there can exist no such pair, nor can their be any pairs of green edges $x_1y_2^1$, $x_2y_2^2$ such that $x_1\in \cX_1, x_2\in \cX
_2, y_2^1, y_2^2\in \cY
_1$. Therefore, we may assume, without loss of generality, that all edges present in $\cG[\cX_1,\cY_1]$ and $\cG[\cX_2,\cY_2]$ are coloured blue.

Thus, this case can be dealt with alongside the next one.

\section{Proof of the main result -- Part III -- Case (vi)}
\label{s:p13}

Suppose that $\cG$ contains a subgraph $H$ from
$$\cL=\cL\left(x,4\eta^4k, \text{red}, \text{blue}, \text{green}\right).$$
where
\begin{equation}
\label{SIZESIII}
\left.
\begin{aligned}
\!\!\!\!\!\!\!\!\!\!\!\!\!\! \textbf{either}\quad\quad x &\geq(\half\aI+\tfrac{1}{4}\eta)k \text{and} \max\{\aII,\aIII\}\leq\aI \\
\!\!\!\!\!\!\!\!\!\!\!\!\!\! \textbf{or}\quad\quad 
x&\geq(\aI-18\eta^{1/2})k \text{and} \max\{\aII,\aIII\}\leq(\tfrac{3}{2}+2\eta^{1/2})\aI. \quad\quad\quad\quad\quad 
\end{aligned}
\right\}\!\!\!\!\!\!\!\!\!\!\!\!
\end{equation}

Observe that in both cases, provided $k\geq 4/\eta$, since $\eta\leq \tfrac{1}{100}, (\tfrac{\aI}{100})^2, (\tfrac{\aII}{20})^4, (\tfrac{\aIII}{20})^4$, we have 
$$x\geq\max\{\half\aI k+1,(\half\aII+\eta^4)k,(\half\aIII+\eta^4)k\}.$$

Recalling the definition of $\cL$ and Part II.C above, we have a partition of the vertex set $\cV$ of $\cG$ into $\cX_{1}\cup \cX_{2}\cup \cY_{1}\cup \cY_{2}\cup \cZ$
such that
\begin{itemize}
\labitem{L0}{L0} $|\cX_{1}|, |\cX_{2}|, |\cY_{1}|, |\cY_{2}|\geq x$;
\labitem{L1}{L1} $\cG[\cX_{1}\cup \cX_{2}\cup \cY_{1}\cup \cY_{2}]$ is $4\eta^4k$-almost-complete; 
\labitem{L2}{L2} (a) all edges present in $\cG[X_1]$, $\cG[X_2]$, $\cG[Y_1]$ and $\cG[Y_2]$ are coloured exclusively red,
\item[{~}] (b) all edges present in $\cG[X_1,Y_1]$ and $\cG[X_2,Y_2]$ are coloured exclusively blue,
\item[{~}] (c) all edges present in $\cG[X_1,X_2]$ and $\cG[Y_1,Y_2]$ are coloured exclusively green,
\item[{~}] (d) there are no red edges present in $\cG[X_1,Y_2]$ and $\cG[X_2,Y_1]$.
\end{itemize}

\begin{figure}[!h]
\vspace{-1mm}
\centering
\includegraphics[width=64mm, page=1]{eoo-figs.pdf}
\vspace{-3mm}\caption{$H\in \cL(x_1,c,\text{red},\text{blue},\text{green})$.}
\end{figure}

Recall that $\cG$ is a three-edge-multicoloured graph, therefore we seek to strengthen the statements made in (\ref{L3}) to prescribe not only which colours of edges are present but also which are not.

By (\ref{L1}) and (\ref{L2}), $H$ is $4\eta^4 k$-almost-complete, and all edges in $H[\cX_1]$ are coloured red. Thus, by Theorem~\ref{dirac}, $\cG[\cX_1]$ contains a red connected-matching on at least $x-1\geq\half\aI k$ vertices. Similarly, each of $\cG[\cX_2]$, $\cG[\cY_1]$ and $\cG[\cY_2]$ contains red connected-matchings on at least $\half\aI k$ vertices. 
Thus, no red path can join two of $\cX_1, \cX_2, \cY_1, \cY_2$, since this would give a red connected-matching on at least~$\aI k$~vertices.

Similarly, because $H$ is
$4\eta^4 k$-almost-complete and all edges present in $H[\cX_1,\cY_1]$ are coloured blue, by Lemma~\ref{l:eleven}, $\cG_2[X_1,Y_1]$ contains a blue connected-matching on at least $2x-2\eta^4k\geq\aII k$ vertices. Likewise, $\cG_2[X_2,Y_2]$ contains a blue connected-matching on at least $\aII k$ vertices and $\cG_3[X_1,Y_1]$ and~$\cG_3[X_2,Y_2]$ each contain a green connected-matching on at least $\aIII k$ vertices.
Thus, 
no odd blue or green cycle can share any vertices with $\cX_1\cup\cX_2\cup\cY_1\cup\cY_2$. 

We proceed to show that the vertices of $\cZ$ can each be assigned to one of $\cX_1, \cX_2, \cY_1$ or~$\cY_2$ while maintaining these properties:
We begin by labelling the vertices of $\cZ$, $z_1,z_2,...,z_u$ and considering these in turn beginning with $z_1$.
We showed above that, for instance, there can be no red path connecting $\cX_1$ to $\cX_2$. Thus, we know that $z_1$ can have red edges to at most one of $\cX_1, \cX_2, \cY_1, \cY_2$. Therefore, suppose, without loss of generality, that it has no red edges to $\cX_2, \cY_1$ or $\cY_2$. 

Additionally, suppose that $z_1$ has a blue edge to $X_2$.
In that case, all edges in $\cG[z_1,Y_2]$ must be green in order to avoid having a blue triangle. Then, all edges in $\cG[z_1,Y_1]$ must be blue (in order to avoid a green triangle). Knowing this forces all edges in $\cG[X_2,Y_1]$ and $\cG[X_1,Y_2]$ to be green (in order to avoid a blue triangle or pentagon). Considering the colouring obtained, there can then be no blue or green edges in $\cG[z_1,X_1]$. Thus, all edges in $\cG[z_1,X_1]$ must be red. Adding $z_1$ to $X_1$ and exchanging the names of~$X_2$ and $Y_2$ gives a graph $H_1$ on vertex set $\cV(H)\cup\{ z_1 \}$ belonging to~$\cL=\cL\left(x,4\eta^4k, \text{red}, \text{blue}, \text{green}\right)$.

If, instead of supposing that $z_1$ has a blue edge to $X_2$, we suppose that it has a green edge to $Y_2$ we arrive at an analogous situation to the above. Thus, we may instead assume that all edges in $\cG[z_1, X_2]$ are green and all edges in $\cG[z_1,Y_2]$ are blue, in which case, a similar argument (with no need for re-naming) allows us to add $z_1$ to $X_1$, obtaining a a graph $H_1$ on vertex set $\cV(H)\cup\{ z_1 \}$ belonging to $\cL=\cL\left(x,4\eta^4k, \text{red}, \text{blue}, \text{green}\right)$.
Considering each of $z_1, z_2,...,z_u$ in $Z$ in turn allows us to add each vertex to either $\cX_1, \cX_2, \cY_1$ or $\cY_2$, showing that $\cG$ itself belongs to $\cL=\cL\left(x,4\eta^4k, \text{red}, \text{blue}, \text{green}\right)$.

Recall that, because $H$ is $4\eta^4 k$-almost-complete and all edges in $\cG[\cX_1]$ are coloured red, by Theorem~\ref{dirac}, $\cG[\cX_1]$ contains a red connected-matching on at least $|X_1|-1\geq\half\aI k$ vertices. This provides a simple upper bound $$|\cX_1|,|\cX_2|,|\cY_1|,|\cY_2|\leq\aI k+1.$$
Since $|\cV(G)|\geq(4\aI-\eta)k$, provided $k\geq(2/\eta)^{1/2}$, this also provides a corresponding upper bound
$$|\cX_1|,|\cX_2|,|\cY_1|,|\cY_2|\geq(\aI -2\eta^{1/2})k.$$
Discarding as few vertices as necessary, that is, returning them to $\cZ$, we may therefore assume that
$$(\aI -2\eta^{1/2})k\leq |\cX_1|=|\cX_2|=|\cY_1|=|\cY_2|=p\leq \aI k.$$
The remainder of this section focuses on showing that the original graph must have a similar structure, which can then be exploited in order to force a cycle of appropriate length, colour and parity.

Each vertex $V_{i}$ of $\cG=(\cV,\cE)$ represents a cluster of vertices of $G=(V,E)$. Recall, from (\ref{NK}), that
$$(1-\eta^4)\frac{N}{K}\leq |V_i|\leq \frac{N}{K},$$ 
and that, since $n> \max\{n_{\ref{th:blow-up}}(4,0,0,\eta), n_{\ref{th:blow-up}}(1,2,0,\eta), n_{\ref{th:blow-up}}(1,0,2,\eta)\}$, we can prove that $$
|V_i|\geq \left(1+\frac{\eta}{24}\right)\frac{n}{k}> \frac{n}{k}.$$

We partition the vertices of~$G$ into sets $X_{1}, X_{2}, Y_{1}, Y_{2}$ and $Z$  corresponding to the partition of the vertices of~$\cG$ into $\cX_1, \cX_2, \cY_1,\cY_2$ and $\cZ$. Then,~$X_1, X_2, Y_1$ and $Y_2$ each contain~$p$ clusters of vertices and we have
\begin{equation}
\label{III1} 
|X_1|,|Y_1|,|X_2|,|Y_2|  = p|V_1| \geq (\aI-2\eta^{1/2})n.
\end{equation}

In what follows, we will \textit{remove} vertices from $X_1\cup X_2\cup Y_1\cup Y_2$ by moving them into~$Z$ in order to show that, in what remains, $G[X_1\cup X_2\cup Y_1 \cup Y_2]$ has a particular coloured structure.  We begin by proving the below claim which essentially tells us that  $G$ has similar coloured structure to $\cG$:

\begin{claim}
\label{G-structL}
We can \textit{remove} at most $14\eta^{1/2}n$ vertices from each of~$X_1, X_2, Y_1$ and~$Y_2$  so that the following conditions~hold.
 \end{claim}
\begin{itemize}
\labitem{L3}{L3} $G_1[X_1], G_1[X_2], G_1[Y_1]$ and $G_1[Y_2]$ are each $6\eta^{1/2}$-almost-complete; 
\labitem{L4}{L4} $G_2[X_1,Y_1]$ and $G_2[X_2,Y_2]$ are each $4\eta^{1/2}$-almost-complete; and
\labitem{L5}{L5} $G_3[X_1,X_2]$ and $G_3[Y_1,Y_2]$are each $4\eta^{1/2}$-almost-complete.
\end{itemize}
\begin{proof} Consider the complete three-coloured graph $G[X_1]$. Recall that there can be no blue or green edges present in $\cG[\cX_1]$ and that $\cG[\cX_1]$ is $4\eta^4k$-almost-complete. Given the structure of~$\cG$, we can bound the number of non-red edges in $G[X_1]$ by
$$p \binom{N/K}{2} + 4\eta^4 pk \left(\frac{N}{K}\right)^2+2\eta\binom{p}{2}\left(\frac{N}{K}\right)^{2}.$$ 
Since $K\geq 2k, \eta^{-1}$, $N\leq 4n$ and $p\leq \aI k \leq k$, we obtain
$$e(G_2[X_1])+e(G_3[X_1])\leq [4\eta + 16\eta^{2} +8 \eta]n^2\leq 18\eta n^2.$$

Since $G[X_1]$ is complete and contains at most $18\eta n^2$ non-red edges, there are at most~$6\eta^{1/2}n$ vertices with red degree at most $|X_1|-1-6\eta^{1/2}n$. Removing these vertices from $X_1$, that is, re-assigning these vertices to~$Z$, gives a new~$X_1$  
such that every vertex in $G[X_1]$ has red degree at least $|X_1|-6\eta^{1/2}n$. The same argument works for $G[X_2], G[Y_1], G[Y_2]$, thus completing the proof of (\ref{L3}).

Next, consider $G[X_1,Y_1]$.We can bound the number of non-red edges in $G[X_1,Y_1]$ by 

$$4\eta^4 pk \left(\frac{N}{K}\right)^2+2\eta p^2\left(\frac{N}{K}\right)^{2}. $$

Since $K\geq 2k, \eta^{-1}$, $N\leq 4n$ and $p\leq \aI k \leq k$, we obtain
$$e(G_1[X_1,Y_1])+e(G_3[X_1,Y_1])\leq 16\eta n^2.$$

Since $G[X_1,Y_1]$ is complete and contains at most $16\eta n^2$ non-blue edges, there are at most $4\eta^{1/2}n$ vertices in~$X_1$ with blue degree to~$Y_1$ at most $|Y_1|-4\eta^{1/2}n$ and at most~$4\eta^{1/2}n$ vertices in~$Y_1$ with blue degree to~$X_1$ at most $|X_1|-4\eta^{1/2}n$. Removing these vertices results in every vertex in~$X_1$ having degree in $G_2[X_1,Y_1]$ at least $|Y_1|-4\eta^{1/2}n$ and every vertex in~$Y_1$ having degree in $G_2[X_1,Y_1]$ at least~$|X_1|-4\eta^{1/2}n$.

We repeat the above for~$G[X_2,Y_2]$, removing at most 
$4\eta^{1/2}n$ 
vertices from each of $X_2, Y_2$ such that every (remaining) vertex in~$X_2$ has degree in $G_2[X_2,Y_2]$ at least $|Y_2|-4\eta^{1/2}n$ and every (remaining) vertex in~$X_2$ has degree in $G_2[X_2,Y_2]$ at least $|X_2|-4\eta^{1/2}n$ thus completing the proof of (\ref{L4}).
Finally, we repeat these steps for $G_3[X_1,X_2]$ and $G_3[Y_1,Y_2]$ removing at most $4\eta^{1/2}n$ vertices from each of $X_1,X_2,Y_1$ and $Y_2$ such that every (remaining) vertex in~$X_1$ has degree in $G_3[X_1,X_2]$ at least $|X_2|-4\eta^{1/2}n$, every (remaining) vertex in~$X_2$ has degree in $G_3[X_1,X_2]$ at least $|X_1|-4\eta^{1/2}n$,
every (remaining) vertex in~$Y_1$ has degree in $G_3[Y_1,Y_2]$ at least $|Y_2|-4\eta^{1/2}n$ and every (remaining) vertex in~$Y_2$ has degree in $G_3[Y_1,Y_2]$ at least $|Y_1|-4\eta^{1/2}n$,
 thus completing the proof of (\ref{L5}).
\end{proof}
Having removed these vertices, we have
\begin{equation}
\label{III2} 
|X_1|,|Y_1|,|X_2|,|Y_2|   \geq (\aI-16\eta^{1/2})n,
\end{equation}
We now define four (possibly overlapping) subsets of $Z$
\begin{align*}
Z_1^X&=\{ z\in Z \text{ such that } z \text{ has at least } 40\eta^{1/2} \text{ red neighbours in } X_1\}; \\
Z_2^X&=\{ z\in Z \text{ such that } z \text{ has at least } 40\eta^{1/2} \text{ red neighbours in } X_2\}; 
\end{align*}

{~}
\vspace{-11mm}
\begin{align*}
Z_1^Y&=\{ z\in Z \text{ such that } z \text{ has at least } 40\eta^{1/2} \text{ red neighbours in } Y_1\}; \\
Z_2^Y&=\{ z\in Z \text{ such that } z \text{ has at least } 40\eta^{1/2} \text{ red neighbours in } Y_2\},
\end{align*}
and proceed to prove that $Z\backslash (Z_1^X \cup Z_2^X \cup Z_1^Y \cup Z_2^Y)=\emptyset.$ 

Indeed, suppose $z$ is a vertex in $Z$ not belonging to any of the newly defined sets, then $z$ has only small red degree to each of $X_1,X_2,Y_1,Y_2$. Thus, in particular, $z$ must have a blue or green edge to each of $X_1, X_2, Y_1$ and $Y_2$. 

\begin{claim}
If there exist vertices $z\in Z\backslash (Z_1^X \cup Z_2^X \cup Z_1^Y \cup Z_2^Y)$, $x_1\in X_1$ and $x_2\in X_2$ such that both the edges $x_1 z$ and $x_2 z$ are green, then $G$ contains a green cycle on exactly $\langle \aIII n \rangle$ vertices.
\end{claim}

\begin{proof}
Recalling (\ref{SIZESIII}) and (\ref{III2}), we know that $|X_1|,|X_2|\geq (\tfrac{5}{8}\aIII-16\eta^{1/2})n$. We let $\widetilde{X}_1$ consist  $(\tfrac{5}{8}\aIII-16\eta^{1/2})n$ vertices from $X_1$ including $x_1$, let $\widetilde{X}_2$ consist of $\lfloor \half \langle \aIII n\rangle \rfloor$ vertices from $X_2$ including $x_2$ and consider $G[\widetilde{X}_1,\widetilde{X}_2]$. Every vertex in $\widetilde{X}_2$ has degree in $G_3[\widetilde{X}_1,\widetilde{X}_2]$ at least $(\tfrac{5}{8}\aIII-22\eta^{1/2})n\geq
\half(|\widetilde{X_1}|+|\widetilde{X_2}|+1).$
Thus, by Corollary~\ref{bp-dir2}, there exists a green path from $x_1$ to $x_2$ on exactly $\langle \aIII n \rangle -1$ vertices which together with $x_1 z$ and $x_2 z$ forms a green cycle on exactly $\langle \aIII n \rangle$ vertices.
\end{proof}

Analogous results can be proved for similar pairs of green edges to $Y_1$ and $Y_2$ or blue edges to $X_1$ and~$Y_1$ or $X_2$ and $Y_2$. Thus, the existence of a green (resp. blue) edge in $G[z,X_1]$ implies that all edges in $G[z,X_1]$ and $G[z,Y_2]$ must be green (resp. blue)  and that all edges in $G[z,X_2]$ and $G[z,Y_1]$ must be blue (resp. green).
At this stage, the existence of a blue or green edge between $X_2$ and $Y_1$ could be used in a similar manner to prove the existence of a blue cycle on exactly $\langle \aII n \rangle$ vertices or a green cycle on exactly~$\langle \aIII n \rangle$ vertices.
Thus, all edges present in $G[X_2,Y_1]$ must be red. However, the existence of even two independent red edges in $G[X_2,Y_1]$ would imply the existence of a red cycle on exactly $\llangle \aI n \rrangle$ vertices in $G[X_2\cup Y_1]$. 

Indeed, suppose $x_1,x_2$ are distinct vertices in $X_2$ and $y_1, y_2$ are distinct vertices in $Y_1$ such that $x_1 y_1$ and $x_2 y_2$ are both coloured red and let $\widetilde{X}_2$ be any set of $\half \llangle \aI n \rrangle $ vertices in~$X_2$ such that $x_1,x_2 \in \widetilde{X}_2$. 
By (\ref{L3}), every vertex in $\widetilde{X}_2$ has degree at least $|\widetilde{X}_2|-8\eta^{1/2}n$ in $G_1[\widetilde{X}_2]$. Since $\eta\leq (\aI/100)^2$, we have $|\widetilde{X}_2|-8\eta^{1/2}n \geq \half |\widetilde{X}_2| +2$. So, by Corollary~\ref{dirac2}, there exists a Hamiltonian path in $G_1[\widetilde{X}_2]$ between $x_1, x_2$, that is, there exists a red path between~$x_1$ and $x_2$ in $G[X_1]$ on exactly $\half \llangle \aI n \rrangle$ vertices. Likewise, there exists a red path between~$y_1$ and~$y_2$ in~$G[Y_1]$ on exactly $\half \llangle \aI n \rrangle$ vertices.  Combining the edges $x_1y_1$ and $x_2y_2$ with the red paths gives a red cycle on exactly $\llangle \aI n \rrangle$ vertices.

Thus, in fact, all vertices in $Z$ belong to one of the previously defined sets $ Z_1^X$,  $Z_2^X$,  $Z_1^Y$ or $ Z_2^Y$. Therefore at least one of $X_1\cup Z_1^X, X_2\cup Z_2^X, Y_1\cup Z_1^Y$ or $Y_1\cup Z_2^Y$ contains at least $\llangle \aI n \rrangle$ vertices.

Without loss of generality, suppose that $|X_1\cup Z_1^X|\geq\llangle\aI n \rrangle$. We show that $G_1[X_1\cup Z_1^X]$ contains a long red cycle as follows: Let~$X$ be any set of~$\llangle \aI n \rrangle$ vertices from $X_1 \cup Z_1^X$ consisting of every vertex from~$X_1$ and~$\llangle \aI n \rrangle-|X_1|$ vertices from~$Z_1^X$. By (\ref{L3}) and~(\ref{III2}), the red graph~$G_1[X]$ has at least $\llangle \aI n \rrangle -16\eta^{1/128}n$ vertices of degree at least $|X|-6\eta^{1/2}n$ and at most $16\eta^{1/2}n$ vertices of degree at least $40\eta^{1/2}n$. Thus, by Theorem~\ref{chv},~$G[X]$ contains a red cycle on exactly~$\llangle \aI n \rrangle$ vertices, completing the proof of Theorem~C.

\qed

\section{Conclusions}

Together, \cite{KoSiSk}, \cite{BenSko}, \cite{DF1}, \cite{DF2} and this paper give exact values for the Ramsey number of any triple of sufficiently long cycles (except when all three cycles are even but of different lengths). We now discuss briefly what is known for four or more colours beginning with the case when all the cycles in question are of odd length.

In \cite{BonErd}, Bondy and Erd\H{o}s gave the following bounds for the \mbox{$r$-coloured} Ramsey number of odd cycles $$2^{r-1}(n-1)+1\leq R(C_n,C_n,\dots,C_n)\leq(r+2)!n$$
and conjectured that the lower bound gives the true value of the Ramsey number. 

In 2012, \L uczak, Simonovits and Skokan \cite{LSS} gave an improved asymptotic upper bound. For $n$ odd and $r\geq 4$, they proved that  $$R(C_n,C_n,\dots,C_n)\leq r2^r n + o(n)$$ as $n \rightarrow \infty$. 

Note that the conjecture still stands and has been confirmed for three colours by Kohayakawa, Simonovits, and Skokan \cite{KoSiSk}. The structures providing the lower bound are well known and easily constructed. For two colours, the structure is simply two classes of $n-1$ vertices coloured such that all edges within each class are coloured red and all edges between classes are coloured blue (see Figure~\ref{twocole}).

\begin{figure}[!h]
\centering
\includegraphics[width=70mm]{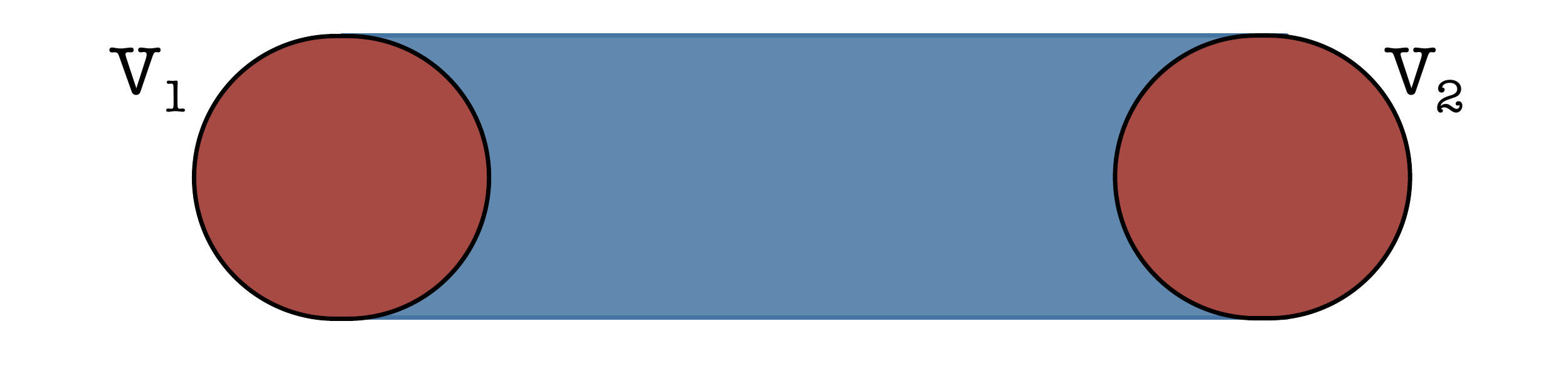}
\caption{Coloured structure giving the lower bound for two colours. }
\label{twocole}
\end{figure}

The relevant $r$-coloured structure is obtained by taking two copies of the $(r-1)$-coloured structure and colouring all edges between the copies with colour $r$ (see Figure~\ref{4cole}).

\begin{figure}[!h]
\centering
\includegraphics[height=50mm]{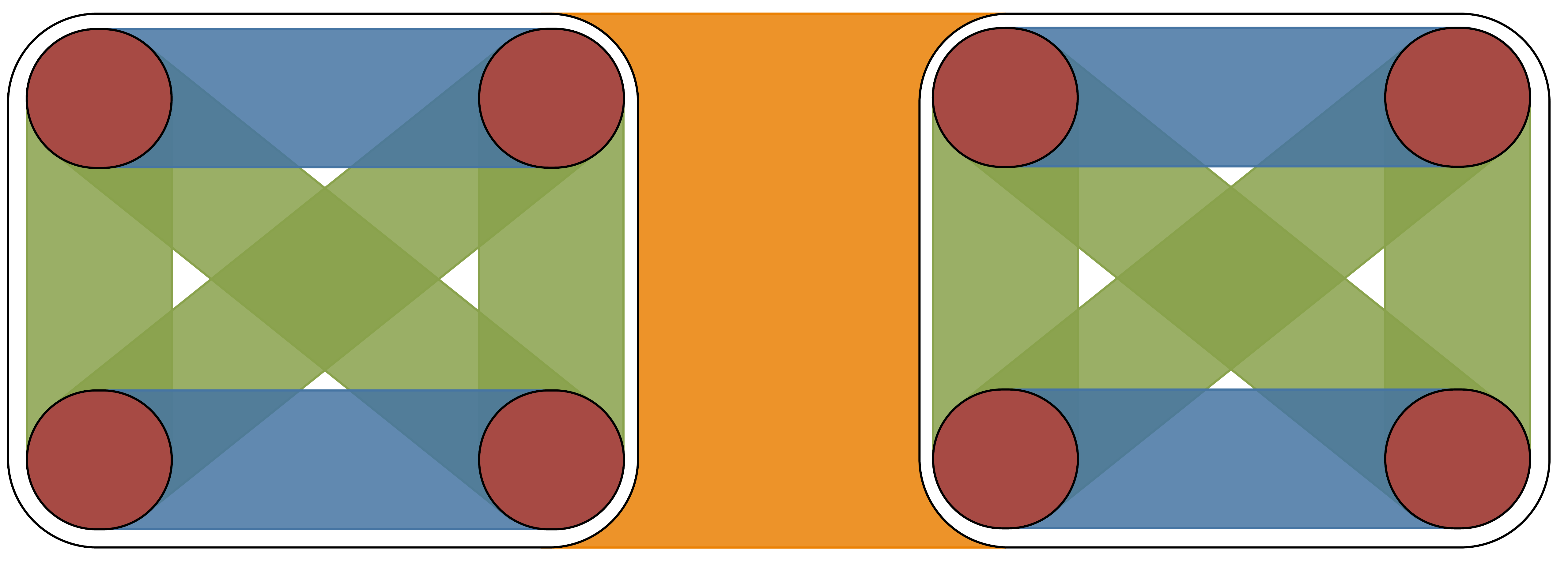}
\caption{Coloured structure giving the lower bound for  four colours. }
\label{4cole}
\end{figure}

Notice that these structures also give a lower bound  for the $r$-coloured Ramsey number  when the cycles have different lengths. Thus, for $n_1\geq n_2\geq \dots \geq n_r$ all odd, we have $$R(C_{n_1},C_{n_2},\dots,C_{n_r})\geq 2^{r-1}(n_{1}-1)+1.$$ 

In 1976, Erd\H{o}s, Faudree, Rousseau and Schelp \cite{EFRS1} considered the case when one cycle is much longer than the others, proving in the case of odd cycles that, if $n$ is much larger than $m,\ell, k$ all odd, then 
$$R(C_n,C_m,C_\ell,C_k)=8n-7,$$
thus showing that the above bound is tight in that case.

This `doubling-up' process can also be used to provide structures giving sensible lower bounds for mixed parity multicolour Ramsey numbers. For example, consider the case of two even and two odd cycles. The three-coloured graph shown in Figure~\ref{fig214-4} below was used earlier to provide a lower bound for $R(C_n,C_m,C_\ell)$ in the case that $n,m$ are even and $\ell$ is odd. Taking two copies of the graph and colouring all the edges between the copies with a fourth colour gives a four-coloured graph providing a lower bound for $R(C_n,C_m,C_\ell,C_k)$, in the case that $n,m$ are even and $\ell, k$ are odd.

\begin{figure}[!h]
\centering
\vspace{1.5mm}

\includegraphics[width=64mm, page=1]{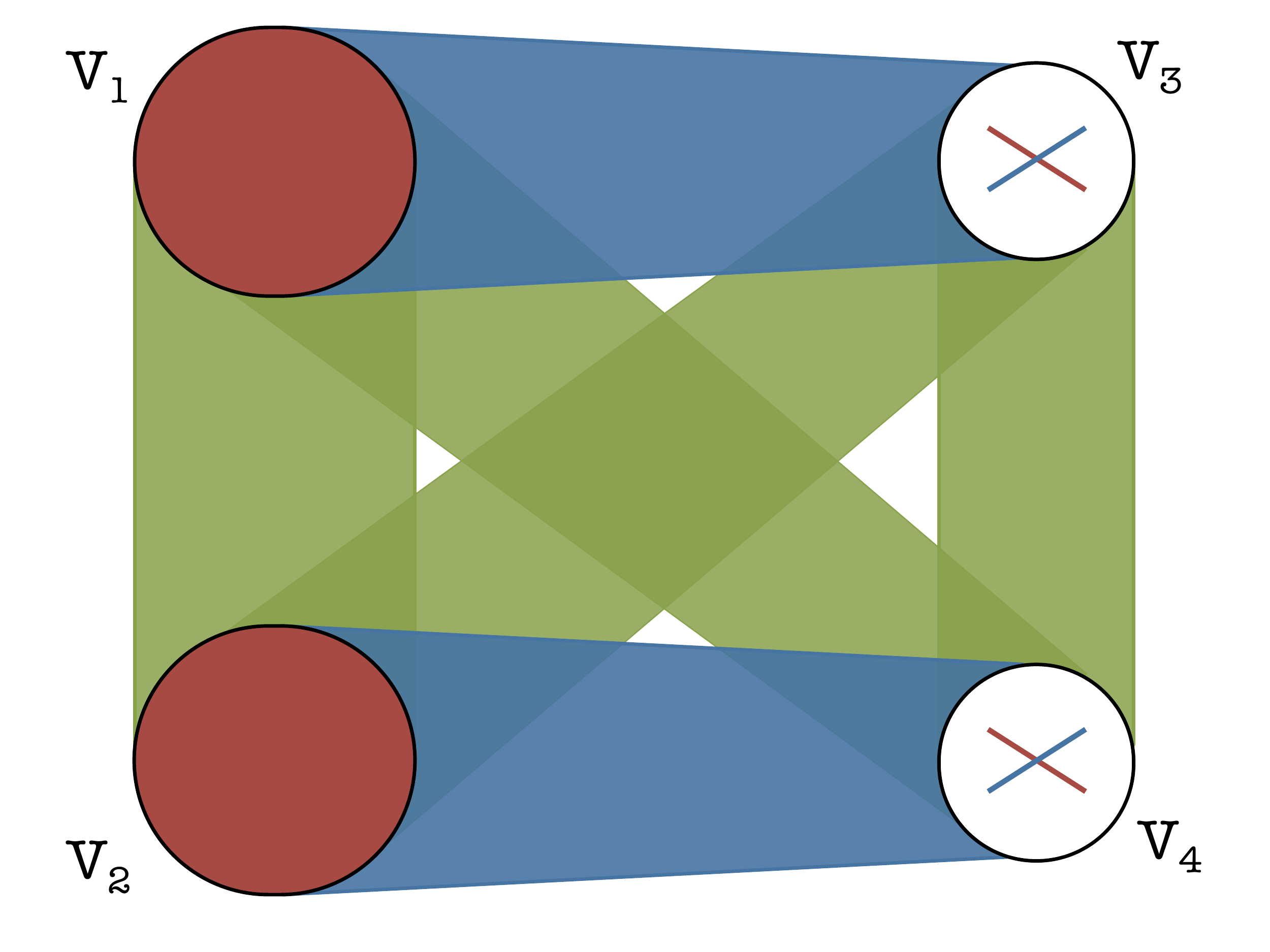}

\vspace{0.5mm}
\caption{Structure providing a lower bound for even-even-odd case.}
\label{fig214-4}
\end{figure}

As the number of colours increases, there are simply too many mixed parity cases to discuss each one here or to give conjectures for the exact or asymptotic Ramsey numbers. However, looking at the structures already seen for three colours and `doubling-up' would seem to be a good place to start.

For even cycles, this `doubling-up' is not an option and the Ramsey numbers grow more slowly as the number of colours increases. Indeed, \L uczak, Simonovits and Skokan \cite{LSS}, proved that the $r$-coloured Ramsey number for even cycles essentially grows no faster than linearly in $r$, proving that, for $n$ even, $$R(C_n,C_n,\dots,C_n)\leq rn + o(n)$$ as $n \rightarrow \infty$. 

Recall that Bondy and Erd\H{o}s \cite{BonErd} proved that, for $n\geq 3$ even, $$R(C_n,C_n)=\tfrac{3}{2}n-1.$$

The simple structure providing the lower bound is shown in Figure~\ref{fig214-5} below. It has two classes of vertices, the first containing $n-1$ vertices and the second $\half n -1$ vertices. It is coloured such that all edges within the first class receive one colour and all other edges receive the second.

\begin{figure}[!h]
\centering
\vspace{-1mm}
\includegraphics[width=64mm, page=1]{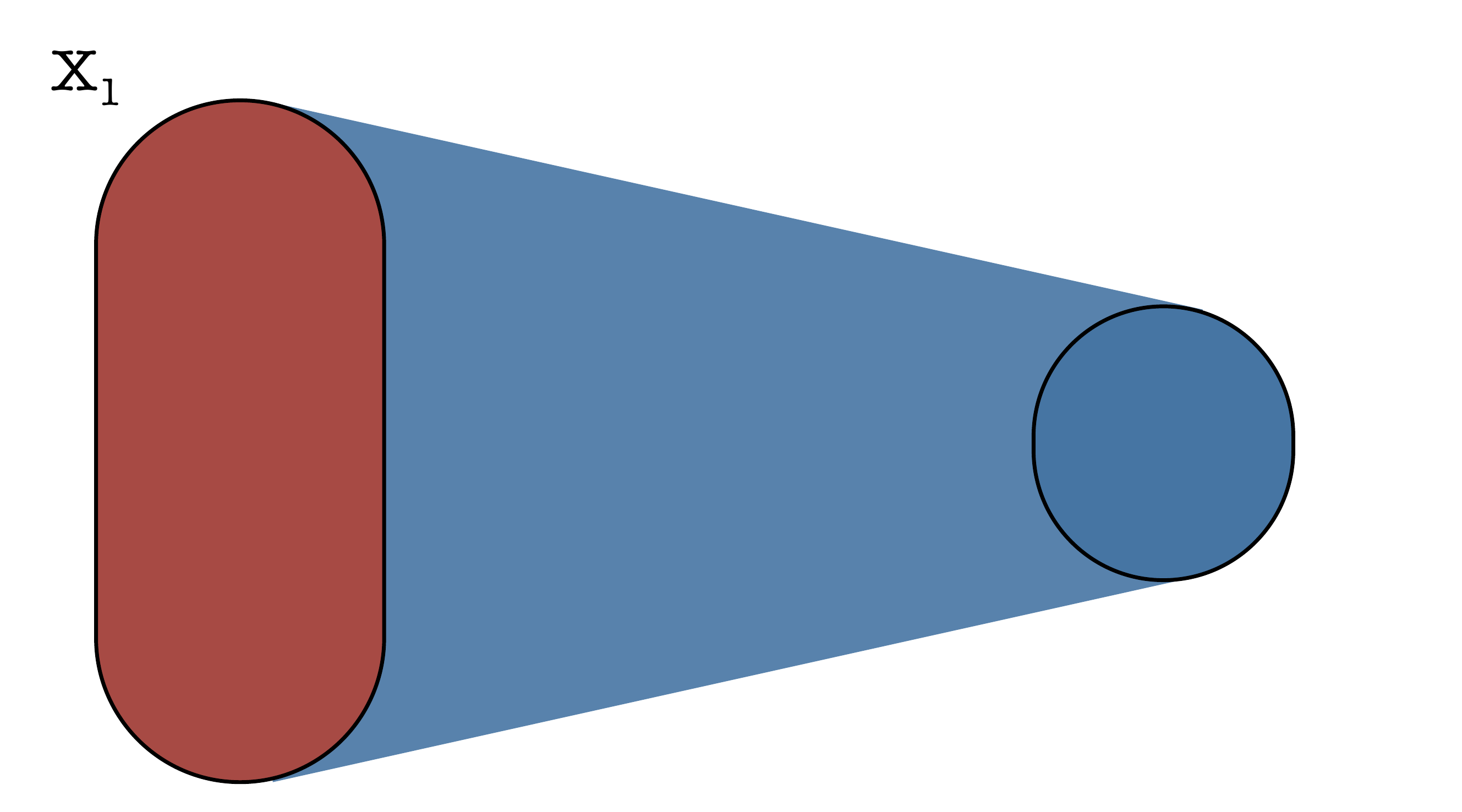}

\vspace{1mm}
\caption{Coloured structure giving the lower bound in the two coloured even case.}
\label{fig214-5}
\end{figure}

This structure is easily extended to give a lower bound for the multicoloured odd cycles (see Figure~\ref{fig214-6}) showing that for $r$ colours $$R(C_n,C_n,\dots,C_n)\geq \half(r+1)n - r +1.$$

\begin{figure}[!h]
\centering
\vspace{-1mm}
\includegraphics[width=64mm, page=1]{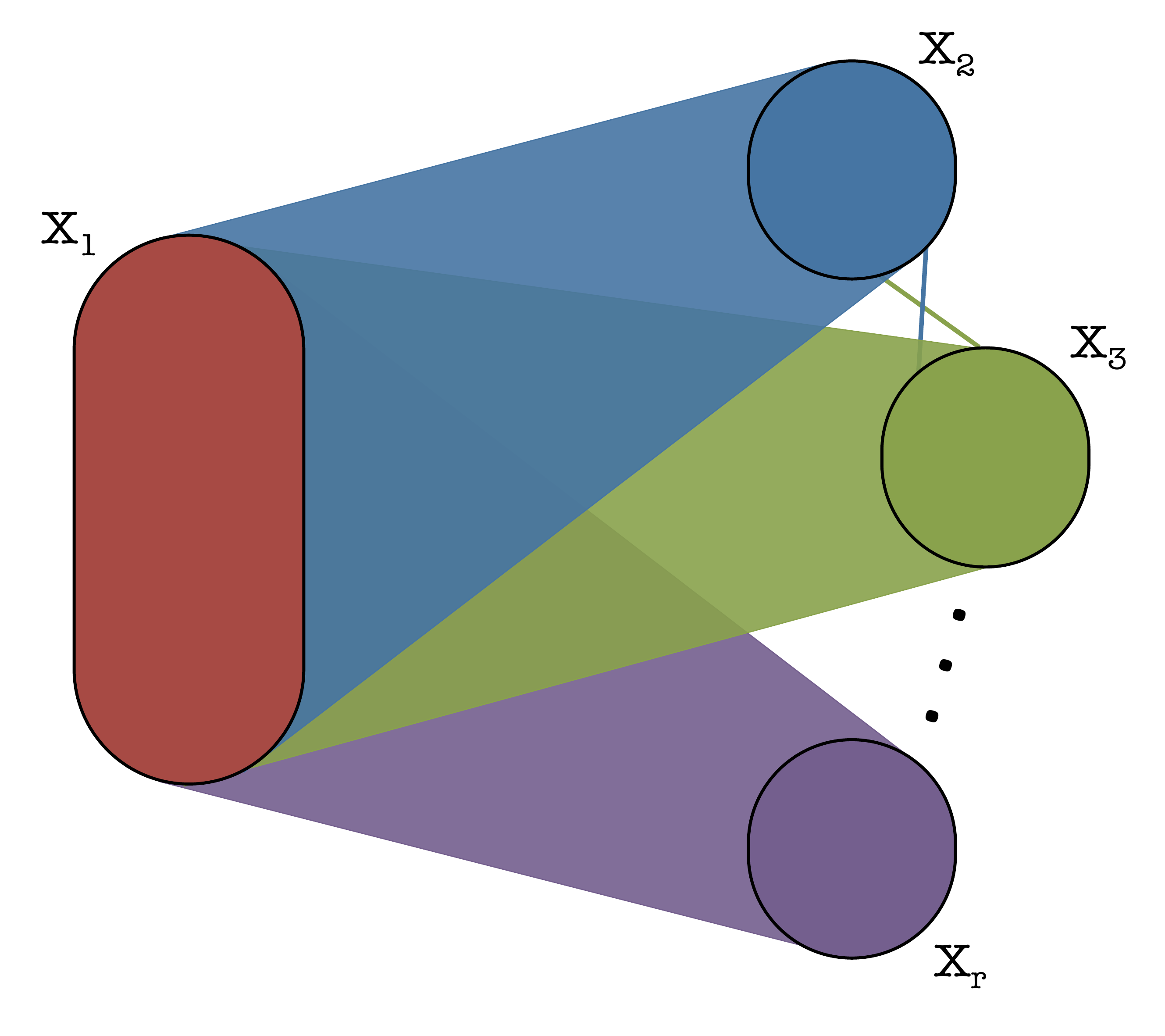}
\vspace{1mm}
\caption{Structure providing a lower bound for $r$-coloured even cycle result.}
\label{fig214-6}
\end{figure}

It can also be adapted to provide a lower bound in the case that the cycles are all even but are of different lengths, showing that, for $n_1\geq n_2,\dots,n_r$ all even, 
\begin{equation}
\label{lastoneprob}
R(C_{n_1},C_{n_2},\dots,C_{n_r})\geq n_1+\half n_2+\dots+\half n_r -r+1.
\end{equation}

Also in \cite{EFRS1},  Erd\H{o}s, Faudree, Rousseau and Schelp showed that, for $n$ much larger than $m,\ell,k$ all even,
\begin{align*}
R(C_n,C_m,C_\ell)&=n+\half n +\half \ell -2, \\
R(C_n,C_m,C_\ell,C_k)&=n+\half n +\half \ell + \half k-3.
\end{align*}
Thus, for two or three colours, the bound in (\ref{lastoneprob}) is tight when one of the cycles is much longer than the others. Notice, also, that this bound agrees with the asymptotic result of Figaj and \L uczak in~\cite{FL2007}. 

However, as shown by the exact result of Benevides and Skokan \cite{BenSko}, this bound can be beaten slightly in the case of three even cycles of equal length. In that case, the less easily extended structure shown in Figure~\ref{fig214-7} gives $R(C_n,C_n,C_n)=2n$.

\begin{figure}[!h]
\centering
\includegraphics[height=50mm, page=1]{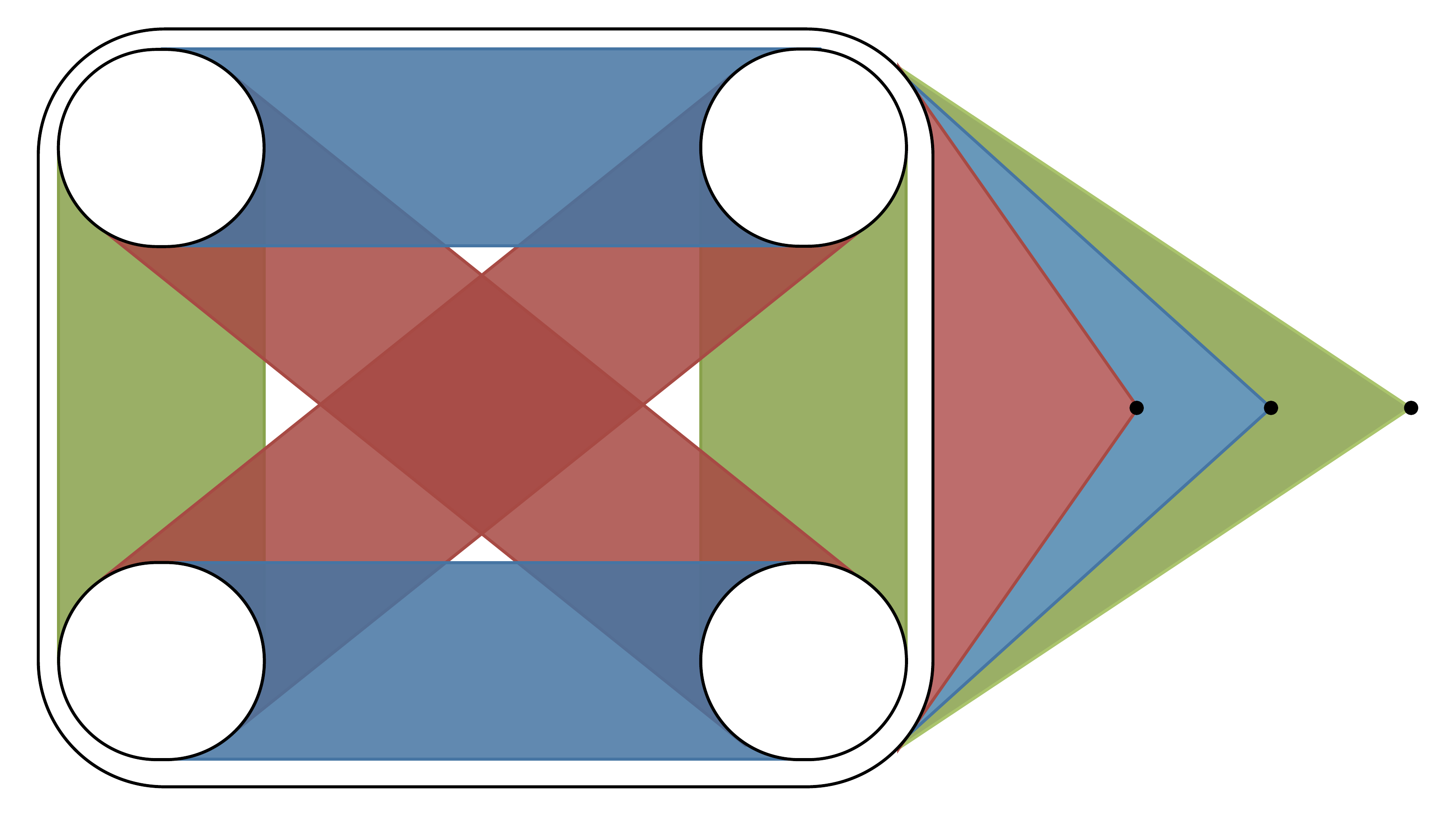}
\vspace{1mm}
\caption{Structure providing the lower bound in the paper of Benevides and Skokan.}
\label{fig214-7}
\end{figure}

Based on the results discussed, one might be tempted to conjecture an asymptotic r-colour result for even cycles of the form $$R(C_{\llangle\aI n\rrangle}, C_{\llangle\aII n\rrangle},\dots,C_{\llangle\alpha_r n\rrangle})=\half\left(\aI+\aII+\dots+\alpha_r+\max\{\aI,\aII,\dots,\alpha_r)\right)n + o(n).$$ 

However, in 2006, in the case of $r$ cycles of equal even length $n$, Sun Yongqi, Yang Yuansheng, Xu Feng and Li Bingxi \cite{SYFL} proved, that 
$$R(C_n,C_n,\dots,C_n)\geq (r-1)n-2r+4,$$
suggesting that the true form of such a result for even cycles is more complicated.

There is potential to apply the methods used in this chapter to the case of four or more colours but there are limitations which could make this quite difficult. For instance. two key sets of tools used to prove the stability result (Theorem B) were decompositions (which were used to find large two-coloured subgraphs within three-coloured graphs) and connectivity results (which reduce the difficulty of finding a connected-matching in a two-coloured graph). The most basic such connectivity result states that a two-coloured graph is connected in one of its colours. Such results do not apply (or are much more complicated) for three-coloured graphs. Therefore, an alternative approach or (even) more case analysis could well be necessary. 

\renewcommand{\baselinestretch}{1.15}\small\normalsize

\cleardoublepage
\phantomsection
\addcontentsline{toc}{section}{References}
\clearpage
\bibliographystyle{halpha2}
\bibliography{test}

\begin{thebibliography}{YYXB06}

\bibitem[BS09]{BenSko}
F.~S. Benevides and J.~Skokan.
\newblock The 3-colored {R}amsey {N}umber of even cycles.
\newblock {\em J. Combinatorial Theory Ser. B}, 99(4):690--708, 2009.

\bibitem[BE73]{BonErd}
J.~A. Bondy and P.~Erd{\H{o}}s.
\newblock Ramsey numbers for cycles in graphs.
\newblock {\em J. Combinatorial Theory Ser. B}, 14:46--54, 1973.

\bibitem[Chv72]{Chv72}
V.~Chv{\'a}tal.
\newblock On {H}amilton's ideals.
\newblock {\em J. Combinatorial Theory Ser. B}, 12:163--168, 1972.

\bibitem[Dir52]{Dirac52}
G.~A. Dirac.
\newblock Some theorems on abstract graphs.
\newblock {\em Proc. London Math. Soc. (3)}, 2:69--81, 1952.

\bibitem[EFRS76]{EFRS1}
P.~Erd{\H{o}}s, R.~J. Faudree, C.~C. Rousseau, and R.~H. Schelp.
\newblock Generalized {R}amsey theory for multiple colors.
\newblock {\em J. Combinatorial Theory Ser. B}, 20(3):250--264, 1976.

\bibitem[EG59]{ErdGall59}
P.~Erd{\H{o}}s and T.~Gallai.
\newblock On maximal paths and circuits of graphs.
\newblock {\em Acta Math. Acad. Sci. Hungar}, 10:337--356 (unbound insert),
  1959.

\bibitem[Fer13]{FERT}
D.~G. Ferguson.
\newblock {\em Topics in {G}raph {C}olouring and {G}raph {S}tructures}.
\newblock PhD thesis, London School of Economics, 2013.

\bibitem[Fer15a]{DF1}
D.~G. Ferguson.
\newblock The {R}amsey number of mixed-parity cycles {I}.
\newblock {P}reprint at \texttt{arXiv [math.CO]}, 2015.

\bibitem[Fer15b]{DF2}
D.~G. Ferguson.
\newblock The {R}amsey number of mixed-parity cycles {II}.
\newblock {P}reprint at \texttt{arXiv [math.CO]}, 2015.

\bibitem[F{\L}07a]{FL2007}
A.~Figaj and T.~{\L}uczak.
\newblock The {R}amsey number for a triple of long even cycles.
\newblock {\em J. Combin. Theory Ser. B}, 97(4):584--596, 2007.

\bibitem[F{\L}07b]{FL2008}
A.~Figaj and T.~{\L}uczak.
\newblock The {R}amsey \vphantom{z}number for a triple of large cycles.
\newblock {P}reprint at \texttt{arXiv:0709.0048 [math.CO]}, 2007.

\bibitem[GRSS07]{GyarSzem}
A.~Gy{\'a}rf{\'a}s, M.~Ruszink{\'o}, G.~N. S{\'a}rk{\"o}zy, and
  E.~Szemer{\'e}di.
\newblock Three-color {R}amsey numbers for paths.
\newblock {\em Combinatorica}, 27(1):35--69, 2007.

\bibitem[KSS09a]{KoSiSk}
Y.~Kohayakawa, M.~Simonovits, and J.~Skokan.
\newblock The 3-colored {R}amsey number of odd cycles.
\newblock {\em J. Combinatorial Theory Ser. B}, 2009.
\newblock {A}ccepted.

\bibitem[KSS09b]{KoSiSk2}
Y.~Kohayakawa, M.~Simonovits, and J.~Skokan.
\newblock Stability of the {R}amsey number of cycles.
\newblock {M}anuscript, 2009.

\bibitem[KS96]{KomSim}
J.~Koml{\'o}s and M.~Simonovits.
\newblock Szemer\'edi's regularity lemma and its applications in graph theory.
\newblock In {\em Combinatorics, Paul Erd\H os is eighty, Vol.\ 2 (Keszthely,
  1993)}, volume~2 of {\em Bolyai Soc. Math. Stud.}, pages 295--352. J\'anos
  Bolyai Math. Soc., Budapest, 1996.

\bibitem[{\L}uc99]{Lucz}
T.~{\L}uczak.
\newblock {$R(C\sb n,C\sb n,C\sb n)\leq(4+o(1))n$}.
\newblock {\em J. Combin. Theory Ser. B}, 75(2):174--187, 1999.

\bibitem[{\L}SS12]{LSS}
Tomasz {\L}uczak, Mikl{\'o}s Simonovits, and Jozef Skokan.
\newblock On the multi-colored {R}amsey numbers of cycles.
\newblock {\em J. Graph Theory}, 69(2):169--175, 2012.

\bibitem[MM63]{moonmoser}
J.~Moon and L.~Moser.
\newblock On {H}amiltonian bipartite graphs.
\newblock {\em Israel J. Math.}, 1:163--165, 1963.

\bibitem[Sze78]{SzemRegu}
E.~Szemer{\'e}di.
\newblock Regular partitions of graphs.
\newblock In {\em Probl\`emes combinatoires et th\'eorie des graphes (Colloq.
  Internat. CNRS, Univ. Orsay, Orsay, 1976)}, volume 260 of {\em Colloq.
  Internat. CNRS}, pages 399--401. CNRS, Paris, 1978.

\bibitem[YYXB06]{SYFL}
Sun Yongqi, Yang Yuansheng, Feng Xu, and Li~Bingxi.
\newblock New lower bounds on the multicolor {R}amsey numbers {$R_r(C_{2m})$}.
\newblock {\em Graphs Combin.}, 22(2):283--288, 2006.

\end{thebibliography}

\end{document}